\documentclass[11pt,reqno]{amsart}
\usepackage{cite}

\usepackage[all]{xypic}
\usepackage{amscd}

\usepackage{amsfonts}
\usepackage{mathrsfs}
\usepackage{amssymb}
\usepackage{mathabx}
\usepackage{pifont}
\usepackage{graphicx}
\usepackage{color}
\RequirePackage[raiselinks=false,colorlinks=true,citecolor=blue,urlcolor=black,linkcolor=blue,bookmarksopen=true]{hyperref}

\usepackage{tikz}
\usetikzlibrary{decorations.markings}
\usetikzlibrary{cd,fit}
\input{epsf.sty}
\catcode `\@=11
\def\numberbysection{\@addtoreset{equation}{section}
         \renewcommand{\theequation}{\thesection.\arabic{equation}}}
\numberbysection
\def\subsubsection{\@startsection{subsubsection}{3}%
  \normalparindent{.5\linespacing\@plus.7\linespacing}{-.5em}%
  {\normalfont\bfseries}}
\usepackage[margin=1in]{geometry}
\newtheorem{thm}{Theorem}[section]

\newtheorem{lem}[thm]{Lemma}

\newtheorem{prop}[thm]{Proposition}

\newtheorem{cor}[thm]{Corollary}

\theoremstyle{definition}

\newtheorem{df}[thm]{Definition}

\newtheorem{rmk}[thm]{Remark}

\newtheorem{ex}[thm]{Example}

\newcommand{\leftsub}[2]{{\vphantom{#2}}_{#1}{#2}}

\def\Cc{\mathcal{C}}

\def\Aut{ \mathrm{Aut} }

\def\ff{ \mathfrak{f} }

\def\O{\mathcal O}

\def\del{\partial}
\def\G{\Gamma}

\def\O{\mathcal{O}}

\usepackage{bold}
\def\unit{\Eins}
\def\ghost{\mathbbnew{\Gamma}}

\def\eps{\epsilon}
\def\G{\Gamma}

\def\Vect{\mathcal{V}ect}

\def\la{\langle}
\def\ra{\rangle}

\def\Graphs{\Graph}
\def\GGraph{{\mathcal PR}}

\def\CalC{{\mathcal C}}

\def\Agg{\rm{A}gg}
\def\agg{\it agg}
\def\Graph{{\rm Gr}}
\def\GGraph{\mathfrak{Gr}}
\def\Set{\mathrm {Set} }

\def\Crl{\mathrm {Crl} }
\def\Ctd{\mathrm {Ctd} }

\def\F{\mathcal F}
\def\FF{\mathfrak{F}}

\def\sg{\sigma}
\def\Sg{\Sigma}
\def\final{1\!\!1}

\def\C{\CalC}
\def\Z{{\mathbb Z}}
\def\N{{\mathbb N}}
\def\G{\Gamma}
\def\del{\partial}
\def\colim{\mathrm{colim}}
\def\FinSet{\mathcal{F}in\mathcal{S}et}

\def\a{\alpha}

\newcommand{\op}{\mathcal}

\def\O{{\mathcal O}}
\def\P{{\mathcal P}}

\def\SS{{\mathbb S}}

\def\form{\la \; ,\; \ra}

\def\deg{\mathrm{deg}}

\newcommand\ccirc[2]{\, \leftsub{#1}{\circ}_{#2}}
\def\scirct{\ccirc{s}{t}}

\newcommand\mge[2]{\, \leftsub{#1}{\boxminus}_{#2}}

\def\V{\asts}
\def\asts{{\mathcal V}}
\def\F{\clusters}
\def\clusters{{\mathcal F}}

\def\opcat{{\mathcal O }ps}
\def\op{op}

\def\inthom{\underline{Hom}}

\def\alg{\textrm{-}{\opcat}}
\def\oper{\rm \op}

\def\mdash{\text{-}}

\def\ot{\otimes}

\newcommand{\Arc}{\mathcal{A}rc}

\newcommand{\ph}{\phi}

\newcommand{\End}{\mathcal{E}nd}
\newcommand{\Iso}{\text{Iso}}

\newcommand{\Fe}{\mathfrak{F}}
\newcommand{\Op}{\mathcal{O}}

\newcommand{\Fepair}{\Fe_{dec}(\Op)}

\def\FFdec{\FF_{dec}}
\def\FFpdec{\FF'_{dec}}
\def\Fdec{\F_{dec}}
\def\Vdec{\V_{dec}}
\def\idec{\iota_{dec}}

\def\Po{\P}

\def\gh{\ghost}

\def\Cor{Cor}
\def\B{\mathcal B}
\def\corra{\la \; \,\; \ra}

\def\Fcyc{\F^{\rm cyc}}

\def\Rib{\rm{R}ib}
\def\FFgraph{\FF^{\Graph}}
\def\FFagg{\FF^{\rm nc \, ng\mdash mod}}

\def\FFoper{\FF^{\rm opd}}
\def\FFnsoper{\FF^{\rm \neg\Sg\mdash opd}}
\def\FFcyc{\FF^{\rm cyc}}
\def\FFmod{\FF^{\rm mod}}
\def\FFnscyc{\FF^{\rm pl\mdash cyc}}
\def\FFnsmod{\FF^{ \rm surf\mdash mod}}
\def\FFncmod{\FF^{\rm nc\mdash mod}}

\def\GG{\mathfrak G}

\def\GGctd{\FF^{\rm  ng\mdash mod}}

 \def\trivial{\O_{\unit}}
  \def\final{\O_{\ast}}
\def\Oio{\O_{\io}}

\def\Ass{\O_{\rm ass}}	
\def\CycAss{\O_{\rm cycass}}

\def\genus{\O_{\rm genus}}
\def\genusnc{\O^{\rm nc}_{\rm genus}}

\def\eulerpoly{\O_{\rm Euler, poly}}

\def\surf{\O_{\rm surf}}
\def\poly{\O_{\rm poly}}
\def\polyN{\O_{\N, \rm poly}}

\def\next{\sigma}
\def\gh{\ghost}
\def\kdk{,\dots,}

\def\cord{\circlearrowleft} 
\def\io{{\rm i/o}}
\def\dec{\rm dec}

\newcommand{\bisub}[3]{{\vphantom{#2}}_{#1}{#2}_{#3}}
\newcommand{\merger}[2]{\bisub{#1}{\boxminus}{#2}}
\newcommand{\gl}[2]{\bisub{#1}{\ominus}{#2}}
\newcommand{\con}[1]{c_{#1}}
\def\vc{\mathrm v}
\def\vmgew{\mge{v}{w}}
\def\Igusa{{\rm I}}
\def\RIgusa{{\rm RI}}

\begin{document}

\title[Trees, graphs and aggregates]{Trees, graphs and aggregates:\\
a categorical perspective on combinatorial surface topology,  geometry, and algebra}

\author[Clemens Berger]{Clemens Berger}
\email{cberger@math.unice.fr}
\address{Universit\'e C\^ote d'Azur, Lab. J.A. Dieudonn\'e, Parc Valrose, 06108 Nice Cedex}

\author[Ralph M.\ Kaufmann]{Ralph M.\ Kaufmann}
\email{rkaufman@math.purdue.edu}
\address{Purdue University Department of Mathematics, West Lafayette, IN 47907}

\begin{abstract}
Taking a Feynman categorical perspective, several key aspects of the geometry of surfaces are deduced from combinatorial constructions with graphs.
This provides a direct route from combinatorics of graphs  to string topology operations via topology, geometry and algebra.
 In particular, the inclusion of trees into graphs and the dissection of graphs into aggregates yield a concise formalism for cyclic and modular operads as well as  their polycyclic and surface type generalizations. The latter occur prominently in two-dimensional topological field theory and in string topology. The categorical viewpoint allows us to use left Kan extensions of Feynman operations as an efficient computational tool. The computations involve the study of certain categories of structured graphs which are expected to be of independent interest.
\end{abstract}

\dedicatory{Dedicated to Dennis Sullivan on the occasion of his 80th birthday}
\maketitle


\section*{Introduction}
Graphs are an ubiquitous tool in mathematics. In geometry, for instance, they arise in the description of surfaces, and in algebra via flow--charts of
compositions.  The latter point of view is what is  formalized with operads\cite{BoVo,Mayoperad,Markl}. Graph theoretically operads  deal with rooted trees or forests.
Forgetting the root, and with it direction,  one considers  trees   and, dropping the condition of being simply connected, graphs in general.  Operadically this corresponds to cyclic operads \cite{GKcyclic} and types of modular operads \cite{Schwarz, GKmodular,KWZ}.
Adding a cyclic ordering for each vertex-corolla yields the notion of ribbon graph, which is central to the theory of Riemann surfaces. Special ribbon graphs, the Sullivan graphs, underlie string topology operations \cite{CS,TZ,hoch1,hoch2}. Adding further data or forgetting some of the data leads to a host of other graphical structures, which appear and  are useful in specific contexts.

Beyond the notion of a graph, the notion of a graph morphism is of prime importance. The graph morphisms of Borisov-Manin \cite{BM} are  adapted to capture all relevant aspects. Their level of sophistication allows to compute the automorphisms correctly and formalizes the operations of  contracting,  grafting and merging. Importantly, such a graph morphism  defines an underlying graph  which will allow us to define \emph{graph insertion} in a precise way.  Indeed, the first result of this article realizes these graph morphisms as the two--morphisms of a double category in which horizontal composition is graph insertion, while vertical composition is the usual composition of graph morphisms restricted to aggregates, where throughout the text an {\em aggregate} is a disjoint union of corollas.  \vspace{1ex}

{\bf Theorem A.}
{\em Each Borisov-Manin graph morphism functorially defines source and target morphisms of aggregates obtained by cutting, respectively contracting, all  edges.
This is part of a  category internal to the Feynman categories  whose horizontal composition corresponds to graph insertion
If suitably restricted this double category has holonomy and connections in the sense of \cite{BrownGhafar}.
} \vspace{1ex}

In particular, we show Borisov-Manin's category of  graphs  $\Graph$   yields a \emph{Feynman category}\cite{feynman} $\FFgraph$ whose monoidal structure of  is disjoint union.  We will call the morphisms between Feynman categories {\em Feynman functors} and  a strong monoidal functor out of a Feynman category will be called a {\em Feynman operation}. The property of being a Feynman category  allows one to use several key results, notably the existence of pull--backs and push--forwards of Feynman  operations, generalizing Frobenius reciprocity for  group operations, and  a factorization system of Feynman functors between Feynman categories into connected Feynman functors and coverings \cite{feynman,decorated,BergerKaufmann}.

All  Feynman categories relevant for this article are graphical in the sense that they
are obtained from $\FFgraph$ either as Feynman subcategories or as coverings of such. Restricting the {\em objects} to aggregates, we recover the Feynman category $\FFagg$ ---called $\GG$ in \cite{feynman}---
central to operad--like theories.
Restriction of the type  of the {\em underlying graphs of basic morphisms}, called ghost graphs,
 defines subcategories while decorations of graphs with additional data are handled by coverings. The relevant categories and their operations are listed in Table \ref{table:types}.
The approach presented here is a bootstrap, whose  ingredients are only the adding/forgetting of roots, inclusion of trees into graphs and the existence of cyclic orders which provides the Feynman operation $\CycAss$ for $\FFcyc$. Denoting the terminal $\Set$ valued Feynman operation for $\FFcyc$ by $\final^{\rm cyc}$, these graphical Feynman categories are related by structure preserving functors as summarized below.\vspace{1ex}

\begin{table}
\begin{tabular}{llll}
$\FF$&$\FF$-operations&ghost graph type and decoration&\\
\hline
$\FFoper$&non--unital symmetric operads&rooted trees\\
$\FFnsoper$&non--unital non--symmetric operads&planar rooted trees\\
$\FFcyc$&non--unital cyclic operads&trees \\
$\FFnscyc$&non--unital planar cyclic operads&planar trees\\
$\FFagg$&unmarked nc modular operads& graphs \\
$\GGctd$&unmarked modular operads&connected graphs \\
$\FFmod$&modular operads&connected genus labelled graphs \\
$\FFnsmod$&surface-modular operads&connected genus/puncture labelled\\
&& polycyclic graphs \\
\end{tabular}
\caption{\label{table:types}Feynman categories and their operations. Planar cyclic operads are also known as non--Sigma cyclic operads and
surface modular operads as non--Sigma modular operads.}
\end{table}

{\bf Theorem B.}
{\em There is a commutative diagram of Feynman categories and Feynman functors,

 \begin{equation}
\label{eq:decodiagintro}
\xymatrix{
\FFnsoper\ar[r]^{i'}\ar[d]^{\pi_1}&\FFnscyc \ar[d] \ar[r]^{j'} \ar[d]^{\pi_2} & \FFnsmod\ar[d]^{\pi_3} \\
\FFoper\ar[r]_i&\FFcyc\ar[r]_j\ar[dr]_k &\FFmod\ar[d]^{\pi}\\
&&\GGctd\\
}
\end{equation}in which $i,i'$ correspond to forgetting the root and $j,j'$ are defined by the inclusion of trees into graphs.
The vertical functors are coverings and $j,j'$ are connected and $j\pi$ is the unique factorization of $k$ into  a connected morphisms and a covering.
In particular, there are equivalences of Feynman categories:

\begin{equation}
\begin{aligned}
\FFoper_{\rm dec}(\Ass)&\simeq \FFnsoper&\FFcyc_{\rm dec}(\CycAss)&\simeq\FFnscyc\\
\GGctd_{\rm dec}(\genus)&\simeq\FFmod& \FFmod_{\rm dec}(\surf)&\simeq\FFnsmod
\end{aligned}\end{equation}
where the subscript ${\rm dec}$ indicates a covering obtained by a decorations with the indicated Feynman operation.
 We have the following identifications of Feynman operations
\begin{equation}
i^*(\CycAss)\simeq \Ass,\quad  k_!(\final^{\rm cyc})\simeq \genus, \quad  k_!(\CycAss)\cong \surf
\end{equation}
}\vspace{1ex}

There are several intermediate coverings that arise naturally on the modular side, which allow us to address different constructions that have appeared in the literature, cf. Table \ref{graphdecotable}.
\begin{table}
\begin{tabular}{l|l|l}
decorating&covering&underlyin ghost graph of\\Feynman category&Feynman category&basic morphism\\
\hline
$\FFoper_{dec}(\Ass)$&$\FFnsoper$&planar rooted tree\\
$\FFcyc_{dec}(\CycAss)$&$\FFnscyc$&planar tree\\
$\GGctd_{dec}(\genus)$&$\FFmod$&genus labelled tree\\
$\GGctd_{dec}(\surf)$&$\FFnsmod$&genus/puncture labelled polycyclic graph\\
$\GGctd_{dec}(\polyN)$&$\FF^{\rm gen poly}$&genus labelled polycyclic graph\\
$\GGctd_{dec}(\poly)$&$\FF^{\rm poly}$&polycyclic graph\\
\end{tabular}
\caption{\label{graphdecotable}Decorated Feynman categories and their ghost graphs for the intermediate covers, cf. \eqref{eq:agghex}}
\end{table}

In this framework, everything boils down to the computation of left Kan extensions. This is possible as soon as the relevant slice categories are well understood. Interestingly, these slice categories are often equivalent to certain categories of structured graphs, e.g. categories of ribbon graphs with subforest contractions as morphisms,  as they appear in the theory of moduli spaces and in physics. Taking a more topological approach, the same categories can also be represented by surfaces with extra structure, often explicitly given in form of a system of arcs or curves. This is what lends the theory to applications in topology and geometry.

For instance, the central computation of the pushforward $k_!\CycAss$ can be done using several different but equivalent combinatorial objects. The calculation of the push-foward can be done in
 graphs,  where the calculation involves the category of spanning forest contractions of ribbon graphs as they appear in the work of Igusa \cite{Igusa}. This is novel and important in the relationship to moduli spaces.
 We present a computation based on cyclic words, closely related to the classification of oriented surfaces, see e.g.  \cite{munkres}. Other presentiations are in \cite{KP,ChuangLazarev,Marklnonsigma,Doubek}. We outline also the relationship with chord diagrams thereby obtaining a link with string topology \cite{cact} and knot theory \cite{Bar-Natan}.

Algebraically, the adjunction between induction and restriction functors (aka \emph{Frobenius reciprocity}) gives new insight into well known results linking 1+1 d Topological Quantum Field Theory, resp. open/closed TQFT to commutative, resp. symmetric Frobenius algebras. To obtain these results, we generalize the notion of an algebra by introducing reference functors. We show that for undirected graphical Feynman categories the natural reference functors are given by pairs consisting of an object of the target category and a propagator; see Table \ref{opertable} for examples. This also formalizes the correlation functions of \cite{hoch2} with values in twisted hom operads, which are necessary to formulate Deligne's conjecture. Let $\trivial^\FF$ denote the trivial Feynman operation for $\FF$. \vspace{1ex}

{\bf Theorem C.}
  Unital algebras over $\trivial^{\FFcyc}$ are commutative Frobenius algebras.
  Unital algebras over $\trivial^{\FFnscyc}$ are symmetric Frobenius algebras.
 Algebras over $\trivial^{\FFcyc}$ (resp. $\trivial^{\FFnscyc}$) are commutative (resp. symmetric) Frobenius objects with a trace and  a propagator.

By adjunction, that is Frobenius reciprocity, unital
 algebras over $\trivial^{\FFmod}$, i.e.\ closed 1+1 d TQFTs are equivalent to commutative Frobenius algebras.
  Unital algebras over $\trivial^{\FFnsmod}$, equivalently over $\surf$,  i.e.\ open  1+1 d TQFTs
  are equivalent to  symmetric Frobenius algebras.
  Without the unit assumption these are commutative, resp.\ symmetric Frobenius objects with a trace and a propagator.

\vspace{1ex}
.

\begin{table}[h]

\begin{tabular}{l|l|l}
Feynman operation&of&algebras\\
\hline
$\trivial^{\FFoper}
$&$\FFoper$&commutative monoids\\
$\Ass$&$\FFoper$&associative monoids\\
$\trivial^{\FFcyc}$&$\FFcyc$&commutative Frobenius algebras\\
$\CycAss$&$\FFcyc$&symmetric Frobenius algebras\\
$\trivial^{\FFmod}$&$\FFmod$&2d closed TFTs\\
$\surf$&$\FFmod$&2d open TFTs\\
\end{tabular}
\caption{\label{opertable}Feynman operations and their algebras.}
\end{table}

In this formalism, the correlation functions underlying the algebraic string topology operations of \cite{hoch2,hochnote} become the pullback to graphs. This completely characterizes them in terms of the bootstrap from graphs and cyclic orders.\vspace{1ex}

{\bf Theorem D.} {\em The  correlation functions for a symmetric Frobenius algebra $A$ are given by a natural transformation $s^*(Y)\in Nat[s^*\surf,\Cor_{A,P}]$, where $s$ is the source functor of Theorem A.
}\vspace{1ex}

Suitably interpreted, the correlations functions furthermore induce actions on the Hochschild chain {\em and} cochain complexes as well as on the Tate--Hochschild complex, cf. \cite{KWang}.

This approach allow the results to transfer to other areas in {\em further work}. One,  \cite{Ddec2}, will deal with PROP actions, such as the one of string topology, cf.\ \cite{hoch1,hoch2} and its generalization. The compositions are intricate, as they are along {\em cycles, not along tails}. The theorems and importantly the computations also allow us to construct moduli spaces \cite{Ddec} using the $W$--construction of \cite{feynman}. This is a generalization of the theorem of Igusa \cite{Igusa} that the moduli spaces can be constructed as the nerve of categories of ribbon graph with subforest contractions.
Using the diagram \eqref{eq:decodiagintro} and denoting a surface type of a surface of genus $g$ with $b$ boundaries marked by
points sets $S_1,\dots, S_b$ and $p$ unmarked boundaries by $(g,p,S_1\dots S_b)$.
\vspace{1ex}

 {\bf Theorem E}.{\em
\begin{enumerate}
\item We have the following chain of inclusions
\begin{equation}
Wj_!(\CycAss)(*_{g,n}) =Cone(\bar M_{g,n}^{K/P})\supset \bar M_{g,n}^{K/P}\supset M_{g,n}
\end{equation} identified as spaces of metric surface marked graphs where the cone point is the corolla of the given type.
\item We have the identification $j'_!(W{\final})(*_{g,p,S_1\amalg \dots \amalg S_b})\simeq M_{S,g,S_1\amalg\dots\amalg S_b}$.
\end{enumerate}
where $\bar M_{g,n}^{K/P}$ is the Kontsevich/Penner/combinatorial compactification of moduli space.
}\vspace{1ex}

The text is organized as follows:

\noindent In \S1 we introduce the relevant notion of graphs including structured graphs such as ribbon graphs. Additionally, an interpretation of these graphs in terms of surfaces with extra data is furnished.

\noindent In \S2 we discuss Borisov-Manin graph morphisms and organise all data into a double category. This section also contains explicit presentations in terms of generators and relations needed for later computations.

\noindent In \S3 we present the correspondence between Feynman operations and coverings. We furthermore discuss a commutative hexagon of coverings relating our approach to others occurring in literature.

\noindent In \S4 key aspects of graphical Feynman categories are established including Theorem A.

\noindent In \S5 Theorem B is proved. One key issue is the computation of the pushforward $k_!(\CycAss)$.
One central technical result is the equivalence of a comma category needed to compute the push--forward and a category of Ribbon graphs with spanning forest contractions. The section is closed by the generalization to non--connected structures such as  disconnected surfaces.

\noindent In \S6 Frobenius algebras and open/closed TFT are linked via Theorems C and D.

\section*{Acknowledgments}
RK would like to thank Dennis Sullivan for his support throughout the years and the wonderful mathematics he has put forth into this world. The current work is heavily influenced by string topology, which has functioned as a continuous inspiration. It is a privilege to dedicate this paper to him.
He  also wishes to acknowledge Yu.\ I.\ Manin, M.\ Kontsevich and B.\ Penner for the continuing conversations and sharing of insights about aspects of moduli spaces and Teichm\"uller theory which have been equally influential for the following text as well as K.~Igusa, D.~Kreimer and K.\ Yeats for related conversations.

RK would like to thank the MPI for Mathematics and HIM in Bonn, the IHES and the KMPB in Berlin for support, and both
 authors would like to thank Universit\'e C\^ote d'Azur. The stays at these institutions were vital for this project.

\section{Graphs}

\subsection{Basic definitions}

\label{sctGraphs}
A {\em graph} $\G=(F,V,\del,\imath)$ is given by the following data: a set of flags $F$, a set of vertices $V$, a boundary map $\del:F\to V$ indicating the incidence of a flag to a vertex, and an involution $\imath:F\to F$ whose two-element orbits are the \emph{edges} of $\G$. Each edge $e=\{f,\imath(f)\}$ is thus formed by two flags, also called half-edges or \emph{inner flags}  of the graph. The fixpoints of the involution are the  tails or \emph{outer flags} of the graph (aka legs, hairs, leads, external flags).

We let $E$ be the set of edges. An edge is called a {\em loop} if its two flags are incident to the same vertex. The flags incident to $v$ form the set $F_v:=\del^{-1}(v)$. The cardinality of $F_v$ is called the \emph{valency} of the vertex $v$.

The {\em disjoint union} of two graphs is given by taking the disjoint unions of the flag and vertex sets and extending boundary map and involution accordingly. A graph is {\em  connected} if it is not the disjoint union of two non--empty subgraphs. Any graph decomposes into a disjoint union of connected components $\G=\bigsqcup_{\bar v\in \bar V}\G_{\bar v}$ where $\G_{\bar v}$ is the maximal connected subgraph containing $v$ and  $\bar V=V/\sim$ where $v\sim w$ if there is an edge path from $v$ to $w$.

A graph is said to be a {\em corolla} if it has a single vertex and no edges. We denote such a corolla by $v_F=(\{v\},F,\partial_F,id_F)$ where $\partial_F$ is the unique map $F\to\{v\}$ and $id_F$ is the identity. An {\em aggregate} is a disjoint union of corollas. A {\em rose} is a one vertex graph, which is not necessarily a corolla.

A \emph{subgraph} $(V',F',\del',\imath')$ of a graph $(V,F,\del,\imath)$ is given by  subsets of vertices and flags such that the edges of the subgraph form a subset of the edges of the ambient graph. Formally,  $V'\subset V$, $F'=\amalg_{v'\in V'}F(_v')$, $\del'=\del|_{F'}$, and $\imath'(f)=\imath(f)$ or $\imath'(f)=f$.
Each vertex $v\in V$ defines a subgraph $v_{F_v}$, the so-called \emph{vertex-corolla}.
A subgraph is {\em spanning},  if contains all vertices. A {\em spanning tree/forest} is a spanning subgraph that is a tree/forest.
Any subgraph $\G'$ of a graph $\G$  can be completed to the spanning subgraph  $\G'\sqcup \bigsqcup_{v\notin V'}v_{F_v}$.

The {\em contraction} of a graph along a spanning subgraph $\G$ of $\G'$ is the graph $\G/\G':=(\bar V_{\G}',F, \bar \imath)$, where $\bar V_{\G'}$ denotes the connected components of $\G'$, $\bar F$ are the flags not belonging to edges of $\G'$ and $\bar \imath_\G$ is the restriction of $\imath_\G$ to $\bar F$.
The disjoint union of the vertex-corollas $\agg(\G)=\amalg_{v\in V} v_{F_v}$ is a spanning subgraph of $\G$, which in general is distinct from $\G$ because its involution is the identity.

The \emph{Euler characteristic} of a graph is defined by $\chi(\G)=|V|-|E|$. Let $b_0$ be the number of connected components of $G$ and $b_1$ be the loop number of $\G$, i.e. the number of edges in the complement of a spanning forest of $G$. Then $\chi(\G)=b_0-b_1$. We call the pair $(b_0,b_1)$ the {\em topological type of $\G$}. A graph is a {\em forest} if and only if $b_1=0$ and a tree if moreover $b_0=1$.

There are two inclusion chains
\begin{equation}
\text{Trees} \subset \text{Connected Graphs} \subset \text{Graphs}\quad\text{and}\quad\text{Trees} \subset \text{Forests} \subset \text{Graphs}
\end{equation}

If $\G$ is connected and $T \subset \G$ is a spanning tree, then $\G/T$ is a rose with $b_1(\G)$ loops. More generally, for a spanning forest $F\subset \G$ the topological types of $\G$ and $\G/F$ are the same.

\begin{rmk} The two \emph{orderings} of the pair $\{f,\imath(f)\}$ forming an edge can also be identified with the two ways of directing the edge, i.e. $(f,\imath(f))$ and $(\imath(f),f$). This gives a precise meaning to the orientation of a loop.
\end{rmk}

\subsection{Topological realization}\label{par:realization}
Each graph can be realized as a one-dimensional topological space: this space is defined by attaching to the (discrete) set of vertices one closed interval for each edge and one semi-open interval for each outer flag, with attaching maps induced by $\del$. Observe that outer flags are attached on one side only so that the topological realization of a graph with outer flags is not a $CW$-complex. Retracting the outer flags, one does obtain a CW complex. For the topological realization of a graph, $b_0$ and $b_1$ are the Betti numbers. The topological realization of a subforest contraction is a deformation retraction and does not change the topological type.

The distinction between edges and inner/outer flags is crucial for  graphical Feynman categories. At some places in literature, outer flags are represented by edges having a univalent vertex on one side. Such a choice yields a $CW$-structure on the topological realization and is understandable from a geometric point of view but is source of confusion from a combinatorial point of view. One is then compelled to distinguish between ``inner'' and ``outer'' vertices, and ``inner'' and ``outer'' edges, while in our setting these distinctions are built into the structure of a graph via its flag involution.

\subsection{Ribbon, polycyclic graphs and genus/puncture labeling}

A {\em cyclic ordering} of a finite set $S$ is given by a permutation $\next:S\to S$ such that the order of $\next$ equals the cardinality $\# S$ of $S$. Starting with an element $s\in S$, we obtain a cycle $s,\, \next(s),\, \dots,\, \next^{\# S-1}(s),\, \next^{\# S}(s)=s$ which we shall represent as a \emph{cyclic word} $(s\, \next(s)\,\cdots\,\next^{\# S-1}(s))$.

A {\em polycyclic ordering} of a finite set $S$ is given by a general \emph{permutation} $\next:S\to S$. In this case there might be several orbits of $\next$ decomposing $S$. A polycyclic ordering of $S$ is then equivalent to an unordered partition of $S$, $S=C_1\sqcup \dots \sqcup C_b$, together with a cyclic ordering of each of the pieces $C_i\,(i=1,\dots,b)$---whence the terminology ``polycyclic''. If $S=\{1,\dots,n\}$ then this decomposition is the cycle decomposition of the permutation $\next$. Observe that each cycle is cyclically ordered and the cycles commute with each other so that the coproduct itself has no preferred order; as it should be.

\begin{df}\mbox{}

\emph{Genus and puncture labelings} are maps $g,p:V\to \N_0$.

A {\em ribbon graph} is a graph  $\G$ together with a cyclic order on each of the sets $F_v$.

A {\em polycyclic graph} is a graph $\G$ together with a polycyclic order on each of the sets $F_v$.

Polycyclic (and a fortiori ribbon) graphs have {\em boundary cycles}: these are the orbits of $\imath\sigma:F\to F$ where $\sigma$ denotes the coproduct of the permutations of the $F_v$ making up the polycyclic structure.

A {\em Sullivan graph} is a ribbon graph such that its boundary cycles are distinguished into ``in--'' and ``out--'' cycles and there are no edges both of whose flags belong to in--cycles.

A {\em surface-marked graph} is a polycyclic graph together with a genus and a puncture labeling.
\end{df}

\begin{ex}
\label{ribex}
We consider a graph $\G$ with one vertex and two loops, i.e.\ $\{1,2,3,4\}\stackrel{\del}{\to} \{*\}$ with $\imath(1)=2,\imath(3)=4$.
There are six ribbon structures on $\G$ represented respectively by the cyclic permutations $(1234)$, $(1243), (1324)$, $(1342), (1423), (1432)$. These fall into two isomorphism classes.
The first isomorphism class has $3$ boundary cycles, while the second isomorphism class has a single boundary cycle, cf. Figure \ref{ribfig}.
\begin{figure}
\includegraphics[scale=1]{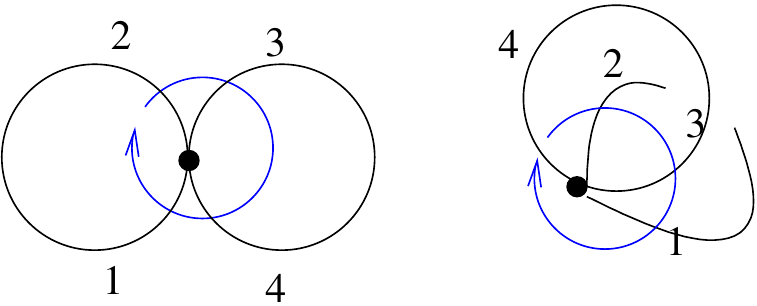}
\caption{\label{ribfig} Two non-isomorphic ribbon structures on the same graph. The blue arrow indicates the cyclic order at the vertex.
The respective surfaces are the sphere with three holes and a torus with one boundary.}
\end{figure}

Likewise, we could define polycyclic structures. Up to isomorphism, there is the trivial polycyclic structure $(1)(2)(3)(4)$, in which the number of boundary cycles is $2$. Furthermore, we have the permutations  $(12)(3)(4), (13)(2)(4), (123)(4), (12)(34), (13)(24)$. The first has $3$, the second $1$, the third $2$, the fourth $4$ and the last $2$ boundary cycles.

The automorphism group, defined below, for the $(1)(2)(3)(4)$ polycyclic structure has the full automorphism group of the underlying graph while $(123)(4)$ has trivial automorphism group.
\end{ex}

Ribbon graphs are ubiquitous in the theory of moduli spaces \cite{Penner,Strebel, Harer, KontsevichAiry}, as is genus labeling \cite{DM,Knudsen}, see \cite{Mondello} for a survey. The polycyclic structure and the puncture labeling are needed for combinatorial compactifications \cite{KontsevichAiry,Penner,Looijenga, Zuniga, postnikov}. Sullivan graphs are relevant for string topology \cite{CS,TZ,hoch1}.

The genus labeling also arises naturally from non--forest contractions as $b_1(\G/\G')=b_1(\G)-b_1(\G')$. In particular, considering a rose $r$,
 $r/r$ is an one vertex aggregate without flags. To keep track of the rose structure one can simply label the aggregate by the loop number $b_1(r)$. This will be formalized in \S\ref{par:genus}.

\subsection{Surface realizations}
\label{par:surfinterpret}
There are several surface realizations associated to structured graphs.
\subsubsection{Surface with curve system associated to a graph}

Each graph defines a curve system $\a(\G)$, i.e an element in the curve complex \cite{Harer} of a topological {\em oriented surface} $\Sigma(\G)$ with labelled boundary, such that the curve system cuts the surface into topological spheres with boundaries. For this replace each $k$--valent vertex  $v$ by a 2--sphere $S^2_v$ with $k$ discs removed. Label the resulting boundaries by the flag set $F_v$. For each edge $\{f,\imath(f)\}$ glue these spheres together at the corresponding boundary components labelled by $f$ and $\imath(f)$ and let $C_e$ be the image curve of the glued boundary. The remaining boundary components are labelled by the outer flags.
We have $b_0(\Sigma(\G))=b_0(\G)$ and $b_1(\Sigma(\G))=2b_1(\G)$. This readily implies that $\chi(\Sigma(\G))=2\chi(\G)$.
An example is given in Figure.

The gluing along boundaries can be thought of as a connected sum operation. Indeed,
the curve system $\a(\G)$ induces a connected sum decomposition of $\Sigma(\G)$ (whose boundary $\del\Sigma(\G)$ is labelled by the set of outer flags of $\G$) according to the formula $\Sigma(\G)\#_{e\in E} \sqcup_{v\in V}S_v^2$. This is a higher generalisation of a pair of pants decomposition where surfaces with a higher number of boundaries, but not with higher genus, are allowed.

To obtain a graph from such a curve system, one takes a vertex for each component of the  surface obtained by cutting along the curve system. The flags are the boundary components of the cut surface. The original boundary components are the outer flags and are labelled. The remaining boundary components are the two sides of a cut curve interchanged by the involution $\imath$.
An example is given in  Figure \ref{fig:surfaceglue}.
\begin{figure}
    \centering
    \includegraphics[width=.7\textwidth]{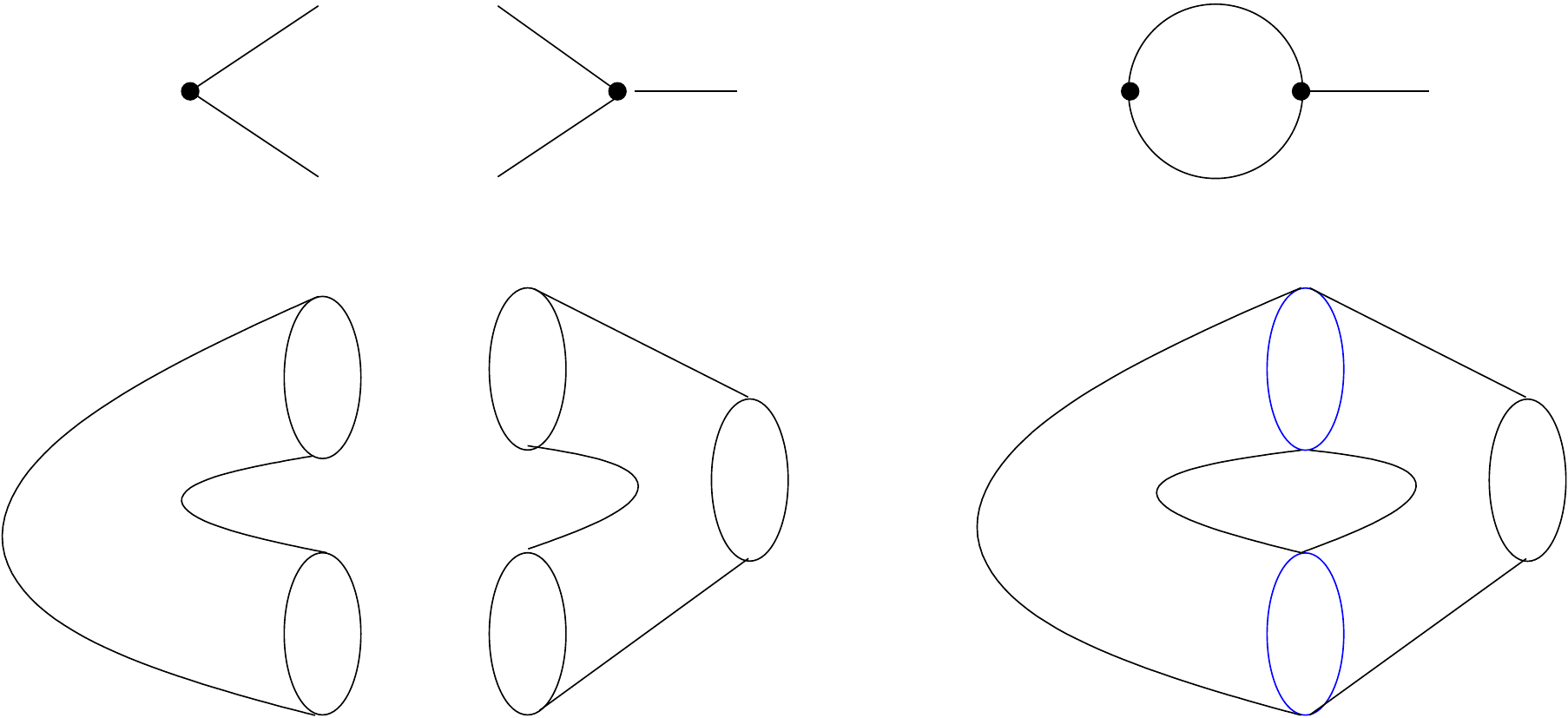}
    \caption{The result of gluing a two-- and a threevalent vertex as a decomposition of the torus with one boundary component with two cut curves}
    \label{fig:surfaceglue}
\end{figure}
An alternative way is to use an appropriate height function and to take its Reeb graph.

\subsubsection{Surface with arc system associated to a ribbon graph}\label{sct:borderedsurface}
 Each ribbon graph $R$ defines an arc system, i.e. an element of the arc complex, $\a(R)$ on a surface $\Sigma(R)$ with boundaries \cite{Strebel,Penner}, where now arcs run in between boundary components and cut the surface into polygons. This property is called quasi--filling.   For each vertex $v$ of valence $n_v$  take a $2n_v$-gon and mark the sides of this polygon with the elements of $F_v$ in the given cyclic order, marking only each second side. Glue these polygons together according to $\imath$ by gluing (for each flag $f$) the $f$-marked side to the $\imath(f)$-marked side. Then identify the glued sides as an arc on the resulting surface $\Sigma(R)$. The boundary components of this surface correspond one-to-one to the boundary cycles of the ribbon graph $R$. In particular, they are polygonal circles. The outer flags give rise to marked intervals on the boundary. The surface $\Sigma(R)$ has thus two types of boundary components. Those containing marked intervals, and those not containing any. We shall call the former boundary components marked and the latter unmarked. If $R$ has no outer flags then all boundary components of $\Sigma(R)$ are unmarked. Note that regardless of any marking all boundary components are hit by at least one arc. See Figure \ref{arcgraph} for an example.

\begin{figure}
    \centering
    \includegraphics[width=.8\textwidth]{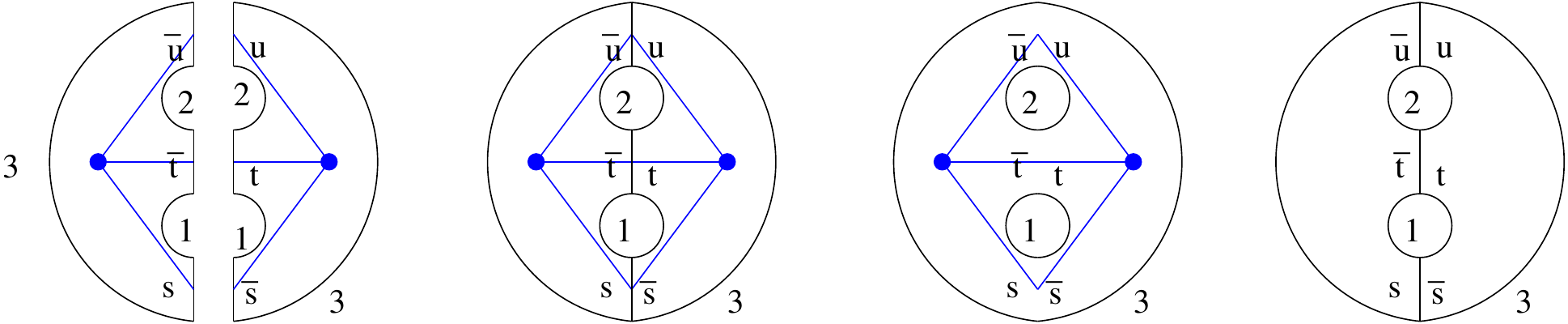}
    \caption{The result of gluing two threevalent {\em cyclic} vertices together to the $\Theta$ graph and its thickening it to a sphere with three holes and its decomposition into two hexagons.
    The blue graphs are original graphs embedded as a spine in the surface, $s,t,u,\bar s,\bar t,\bar u$ are the flags of the graph and equivalently the markings of the marked intervals of the polygons}
    \label{arcgraph}
\end{figure}

Conversely, given a quasi-filling arc system $\a$ on an oriented surface $\Sigma$ with boundary, the ribbon graph $R(\a)$ is constructed as follows. Cut $\Sigma$ along the arcs into $2n_v$-gons. The centers of these $2n_v$-gons define the set of vertices $v$. The alternating sides of the polygons are the flags which inherit a cyclic ordering from the orientation of the surface. They are labelled either by intervals from the boundary, corresponding to outer flags, or by the edges of an arc. Choose a point on each marked boundary interval and on each arc. Insert an arc from the central vertex to each marked point of the boundary intervals or marked point of an arc, such that these inserted arcs do not intersect except at vertices. This yields a ribbon graph $R(\a)$ with a topological realization on the given surface $\Sigma$. Note that $\Sigma$ deformation retracts onto $R(\a)$ which is therefore often called the \emph{spine} of $\Sigma$. It is unique up to isotopy and combinatorially transverse to the given arc system $\a$.

If $n(R)$ is the number of boundary cycles of the ribbon graph $R$, then the \emph{Euler characteristic} of the closed surface $\bar \Sigma(R)$ obtained by gluing in discs to the boundary components equals $\chi(\bar \Sigma(R))=\chi(R)+n(R)$. For a connected ribbon graph $R$, the surface $\Sigma(R)$ is connected, and genus and Euler characteristic of the aforementioned closed surface determine each other by the formula $\chi(\bar \Sigma(R))=2-2g(\bar\Sigma(R))$. For a {\em connected} ribbon graph $R$ we set $g(R)=1-\frac{1}{2}(\chi(R)+n)$.
We call $(b_0(R),b_1(R),n(R))$ the \emph{topological type} of the ribbon graph $R$. If $R$ is connected, then the pair $(g(R),n(R))$ is an alternative way of representing the topological type of $R$.
Note that $n(R)$ depends on the ribbon structure of $R$ while $(b_0(R),b_1(R))$ only depends on the graph underlying $R$.

This construction is related to the previous one by doubling. One can double each $2n$--gon and glue together corresponding marked sides. This gives a sphere with labelled boundaries. The gluing is the gluing of this doubled graph. This explains the factor of $2$ in the formula for the Euler characteristic.

 \subsubsection{Surface with arc system associated to a genus/puncture labelled polycyclic graph}

 Finally each genus and puncture marked polycyclic graph $P$ defines a surface with boundaries, punctures and an arc system, \cite{hoch1,postnikov}. There is now no filling constraint on the arc system. For each vertex $v$ with genus/puncture labeling $(g_v,p_v)$ and orbit decomposition $S_1\amalg\cdots\amalg S_b$ take a topological surface $\Sigma_v$ of genus $g_v$ with $p_v$ unmarked boundary components ---which can topologically be considered as being equivalent to punctures in the interior--- and $b$ marked boundary components, of which the $i$--boundary is a $2|S_i|$--gon whose alternating sides are intervals marked by the elements of $S_i$ in the cyclic order. Glue these surfaces together as above using $\imath$ and keep the glued intervals as arcs as above. The result is an arc system $\a(P)$ on a surface $\Sigma(P)$ whose boundaries of this surface are again either marked or unmarked.

The converse construction again takes a vertex for each region cut out by the arc system and a flag for each marked boundary interval and each side of an arc. The incidence relations $\del$ and $\imath$ as are as above. Since the surface is oriented, this gives a polycyclic decomposition of the flags at each vertex.

There are now several ways to realize the dual graph on the surface. Since the arc system is not quasi--filling, a dual graph can only be constructed by adding arcs until one reaches a quasi--filling arc system. For this one can quasi--triangulate the surfaces $\Sigma_v$. This is a choice, however, and to undo this choice one should either consider an equivalence relation on the vertices \cite{KontsevichAiry} or take equivalence classes under Whitehead moves on the triangulation of the $\Sigma_v$ \cite{PennerCell}. Here a triangulation is given by a system of arcs running from boundary to boundary cutting the surface into hexagons. Shrinking the boundaries to points one obtains a system of arcs running between marked points decomposing the surface into triangles, whence the name. A Whitehead move replaces one edge by another edge as shown in Figure \ref{fig:Whitehead}.

\begin{figure}
    \centering
    \includegraphics[width=.8\textwidth]{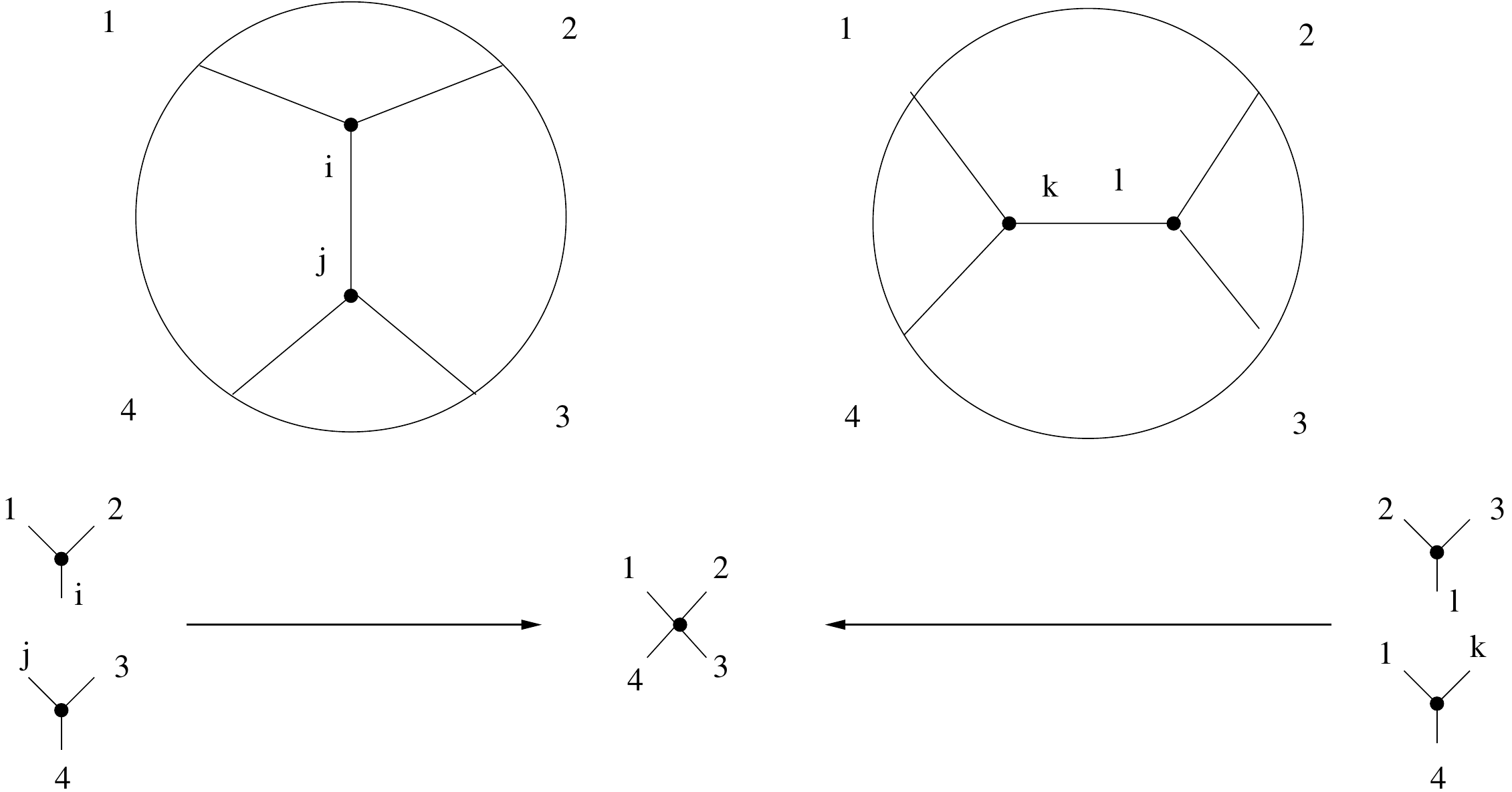}
    \caption{A Whitehead move where the circle is given to show cyclic structure. Below the move on the underlying aggregate }
    \label{fig:Whitehead}
\end{figure}

\begin{prop}
\label{prop:KPK}
 The following notions are equivalent.
\begin{enumerate}
\item A genus and puncture marked polycyclic graph.
\item A ribbon graph with an equivalence relation on the vertex--set and a genus marking for each equivalence class.
\item Equivalence classes of pairs consisting of a ribbon graph and a trivalent spanning ribbon subgraph $\G_{\rm ph}\subset \G$, where two pairs are equivalent if and only if they transform into each other by Whitehead moves on the ``phantom'' part $\G_{\rm ph}$.\end{enumerate}

\end{prop}
The first notion is what is used in this article, cf. \cite{Barrannikov,postnikov,hoch1,hoch2}.
\begin{proof}(1) $\iff$ (2) For a polycyclic order on $F_v$ with $b$ orbits and puncture marking $p_v$, replace the vertex $v$ by $b+p_v$
distinct vertices $v_1,\dots,v_b,w_1,\dots, w_{p_v}$ and attach to each $v_i$ the flags of $F_v$ belonging to the $i$-th orbit, keeping the same cyclic order. The $w_i$ will have no flags.
This defines a ribbon graph together with an equivalence relation on its vertices, and a genus for each equivalence class.

Conversely, given such an equivalence relation on the vertices of a ribbon graph, we can define in a straightforward way a polycyclic structure on the graph obtained by identifying the vertices in the same equivalence class. The resulting polycylic graph inherits the genus-labeling, the puncture marking is the number of $0$--valenced vertices.
These processes are inverses to each other.

(1) $\iff$ (3). Given a polycyclic vertex with a genus marking and puncture marking, again decompose  $F_v=S_1\amalg \dots \amalg S_b$
replace each vertex $v$ with a trivalent connected  ribbon graph $\G_v$ whose associated surface $\Sigma(\G_v)$ has genus $g_v$ and $b+p_v$ boundaries, such that the orbit--decomposition $S_1\amalg\cdots\amalg S_b$ of $F_v$ corresponds one-to-one to the outer flags of the boundary cycles of $G_v$. This choice is not unique, but
 Whitehead moves act transitively on the set of trivalent ribbon graphs of a given topological type and  the polycyclic structure of the outer flags is an invariant, so the process is well defined on equivalence classes.

Conversely, consider $\G_{\rm vir} \subset G$, the quotient  $\G/\G_{\rm ph}$ is a polycyclic graph. The genus  and puncture marking of its vertices  are those  of the corresponding connected components of $\G_{\rm vir}$. As the topological type of the connected components and the polycyclic structure is invariant under Whitehead moves, this construction passes to the equivalence classes.
  These processes are again inverses to each other.
\end{proof}

\begin{rmk} \mbox{}
\begin{enumerate}
\item Note that if the genus and puncture marked polycyclic graph is connected, its representation in the other two points of view need not be.
\item  In these equivalences unmarked boundary components are treated as punctures. It is possible to treat {\em both} unmarked boundary components {\em and} extra internal punctures, this becomes necessary if one considers open/closed theories, cf. \cite{KP,ochoch}.
\item The mapping class group acts on the curve and arc complexes.
The underlying graphs are invariant under the mapping class group action.
This means that the graphs without additional markings such as a fixed embedding into a surface, can be only be used to reconstruct moduli spaces as opposed to Teichm\"uller spaces.
\end{enumerate}
\end{rmk}

\section{Categories of Graphs}

\subsection{Graph morphisms and compositions}

A \emph{graph morphism} $\phi:(V,F,\partial,\imath)\to(V',F',\partial',\imath')$  is given by a triple $(\phi_V,\phi^F,\imath_\phi)$,
consisting of a covariant surjection of vertices $\phi_V:V\twoheadrightarrow V'$,
a contravariant injection of flags $\phi^F:F'\hookrightarrow F$ and a fixed point--free involution $\imath_\phi$ on the set $F\setminus \phi^F(F')$ of flags {\em not} contained in the image of $\phi^F$. The following constraints have to be satisfied:

\begin{enumerate}\item $\phi_V\circ \del\circ\phi^F= \del'$, and on the complement of image of $\phi^F$: $\phi_V\circ\del=\phi_V\circ\del\circ\imath_\phi$.
\item If a flag $f$ does not belong to the image of $\phi^F$ then either $\{f,\imath(f)\}$ is an edge of $\G$ (in which case $\phi$ is said to contract the edge), or
both $f$ and $\imath(f)$ are outer flags (in which case $\{f,\imath_\phi(f)\}$ is called a \emph{ghost edge} virtually contracted by $\phi$).
\item Edges of $\G$ that are not contracted are preserved. That is if $\{f,\imath(f)\}$ form an edge of $\G$ and $f$ is in the image of $\phi^F$ then so is $\imath(f)$, and
 $\imath'(\phi^F)^{-1}(f)=(\phi^F)^{-1}(\imath(f))$.
\end{enumerate}

The information about the involution $\imath_\phi$ is encoded in the {\em ghost graph} $\G(\phi)$ of $\phi$, which is defined by $\ghost(\phi)=(V,F,\hat \imath_\phi)$ where $\hat \imath_\phi$ is the extension of $\imath_\phi$ to all of $F$ by the identity.

For two graph morphisms $\phi=(\phi_V,\phi^F,\iota_\phi):\G\to \G'$ and $\psi=(\psi_V,\psi^F,\iota_\psi):\G'\to \G''$, the {\em composition} $\psi\phi:\G\to \G''$ is defined by setting $(\psi\phi)_V=\psi_V\phi_V$ and $(\psi\phi)^F=\phi^F\psi^F$. The involution $\iota_{\psi\phi}$ pairs two flags of the source graph if  they are  either paired by $\iota_\phi$ or they belong to the image of $\phi^F$ and their preimages are paired by $\iota_\psi$. Graphs and their morphisms form the category $\Graph$.

The {\em composition product} of two ghost graphs is defined to be $\ghost(\psi)\circ\ghost(\phi):=\ghost(\psi  \phi)$.
The composition product has the following description. The vertices and flags of $\ghost(\psi)\circ\ghost(\phi)$ are those of $\ghost(\phi)$. The definition of $\imath_{\psi\phi}$ says that the edges of $\ghost(\psi\phi)$ are the disjoint union of  those of $\ghost(\phi)$ and those of $\ghost(\psi)$ pulled back along $\psi^F$. Starting from $\ghost(\psi)$, the composition product expands the vertices $v'\in V'$ into the graphs $\ghost(\phi_v)$. An example can be seen in Figure \ref{phi3fig}.

\begin{figure}[h]
    \centering
    \includegraphics[width=.8\textwidth]{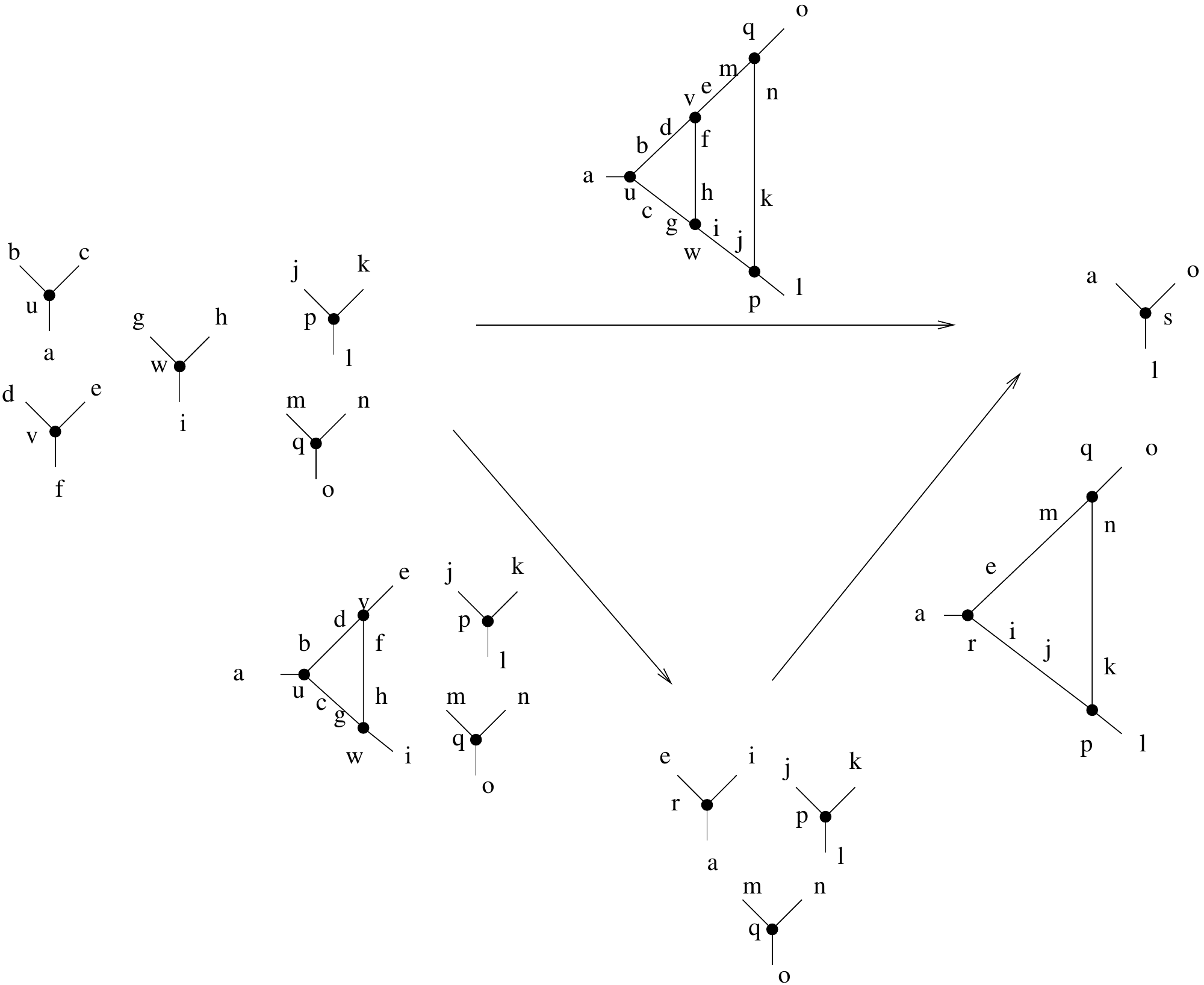}
    \caption{A composition of graph morphisms and ghost graphs. The graph morphisms are virtual edge contractions.
The composition either adds the edges $\{e,m\}$ and $\{i,j\}$ to $\ghost(\phi)$ or inserts the left triangular graph with vertices $u,v,w$ into the vertex $r$ of the right triangular graph respecting flag identifications.}
    \label{phi3fig}
\end{figure}

The disjoint union endows the category $\Graph$ with a monoidal structure. Every morphisms decomposes according to the connected components of its target.
Given $\phi:\G\to \G'$ the fiber $\G_{\overline{v'}}$ over a connected component $\overline{v'}$ of $\G'$  is the subgraph given by the vertices $\phi_V^{-1}(\overline {v'})$ ---thinking of $\overline{v'}$ as the set of vertices in the equivalence class--- together with all their flags and the restriction of $\imath$. This restriction is possible, since all edges of $\G$ are either preserved or contracted.
In this notation:
\begin{equation}
\label{eq:phidecomp}
\phi=\bigsqcup_{\overline{v'}\in \overline{V'}} \phi_{\overline{v'}}
\end{equation}
Notice that the preimages $\G_{\overline{v'}}$ need not be connected.

\subsection{Special types of morphisms}
From the definition it follows that
graph {\em isomorphisms} are given by triples $(\phi_V,\phi^F,\imath_{\emptyset})$ such that $\phi_V$ and $\phi^F$ are bijections, $\imath_{\emptyset}$ is the identity of the \emph{empty set}, and $\phi^F$ induces a bijection on the  edges of the target graph to edges of the source graph. In particular, graph automorphisms may permute edges, flags and vertices as long as the incidence relations and the flag involutions are preserved.

{\em Graftings} are graph morphisms for which $\phi_V$ and $\phi^F$ are bijections, but $\phi^F$ is not necessarily edge-preserving, i.e. source and target may have different flag involutions. The source edges are preserved, but the target  may contain additional edges comprised of outer flags of the source which are called {\em  grafted edges}.

{\em Mergers} are graph morphisms, for which $\phi^F$ is an edge-preserving bijection while $\phi_V$ may be arbitrary.

{\em Contractions} are morphisms, in which $F\setminus \phi^F(F')$ is a collection of edges of $\G$ and two vertices are in the same fiber of $\phi_V$ only if they are joined by a path of these edges.

It is readily verified that graph isomorphisms belong to all of the three classes, and that the three classes are closed under composition. We denote by $\Graph^{\rm graft},\Graph^{\rm merge},\Graph^{\rm contr}$ the \emph{wide subcategories} of $\Graph$ generated by graftings, mergers, and contractions respectively, cf.\ Corollary \ref{cor:graphcatdef}. Moreover, we will see below that graftings and contractions together also generate a wide subcategory which we shall denote $\Graph^{\rm ctd}$ because the morphisms in $\Graph^{\rm ctd}$ may be characterized as those graph morphisms whose fibres are connected subgraphs of the source graph.

A \emph{loop contraction} means that all contracted edges are loops and a \emph{forest contraction} is a contraction none of whose ghost edges   form a cycle. This condition is equivalent to the condition that the ghost graph is a forest. Forest contractions do not change the topological type.

\begin{ex}The single rose graph given by $V=\{v\}, F=\{s,t\}$ and $\imath(s)=t$ has automorphism group $\Z/2\Z$ with the generator given by $\phi^F(s)=t$.
The automorphism group of the $n$--fold rose, given by $V=\{v\}$, $F=\{s_1, t_1, \dots s_n,t_n\}$ and $\imath(s_i)=t_i$, is $(\Z/2\Z)^n\wr \Sigma_n$, where the first factor switches the $s_i$ and $t_i$ and the $\Sigma_n$ action permutes the edges $\{s_i,t_i\}$.
\end{ex}

\begin{lem}
\label{lem:treeaut}
A tree automorphism is determined by its action on outer flags and on univalent vertices. The only tree automorphism fixing outer flags and univalent vertices is the identity.
\end{lem}
\begin{proof}By hypothesis, the automorphism is determined on univalent vertices and on outer flags. It follows that the automorphism is also determined on inner flags attached to univalent vertices and on vertices attached to outer flags. Therefore, deleting all univalent vertices and outer flags from the tree leaves us with a strictly smaller tree restricted to which the automorphism satisfies the same hypothesis. An easy induction allows us to conclude the statments.\end{proof}

The inclusion of a spanning subgraph $i_{\G'}:\G'\to \G$ is a   grafting   of the edges not in $\G'$.
 Dually, the  {\em dissection} along such a subgraph is given by cutting the edges  of $\G'$, that is the graph $(V_\G,F_\G,\del_\G, \imath^{\G'}_\G)$ where the new flag involution is the identity on the flags of $\G'$ and equal to $\imath_\G$ otherwise. If $\G'=\G$ then we get the \emph{total dissection} (or underlying aggregate) $\agg(\G)$ of $\G$ which comes equipped with a canonical grafting $i_\G:\agg(\G)\to\G$.

The quotient $\G/\G'$ is defined to be the graph whose vertices are the connected components of $\G'$. The flags of $\G/\G'$ are the outer flags of $\G'$. The flag involution on $\G/\G'$ is given as restriction of $\imath_\G$. The \emph{total contraction} is the quotient $c_\G:\G\to \G/\G$ where $\G/\G$ is the aggregate whose vertices are the connected components of $\G$ and whose flags are the outer flags of $\G$, each outer flag being attached to its connected component.
It follows from a simple computation that $c_{\G'}$ preserves the number of connected components, and if $\G'$ has loop number $b_1(\G')$ then $b_1(\G)=b_1(\G/\G')+b_1(\G')$. Since $b_1(\G')=b_1(\ghost(c_{\G'}))$, we see that the drop in loop number is encoded in the morphism and kept track of by the ghost graph. If $\G'$ is a spanning subforest, then the contraction does not change the topological type.

The total dissection $\agg(\G)$ and total contraction $\G/\G$ are   \emph{aggregates}. As  the graftings $i_\G:\agg(\G)\to\G$ and the contractions $c_\G:\G\to \G/\G$  are natural in $\G$, these constructions actually define  functors $\Graph\to\Agg$.

\subsection{Simple generators for graph morphisms}\label{relsec}
Although we have a global presentation of the category $\Graph$, for further analysis and to perform calculations, it is useful to give a presentation of the morphisms in terms of generators and relations. To this end, we provide a structural theorem which refines and concretizes statements about generators and decompositions made in \cite{BM}.
This also allows us to analyze several wide subcategories.

There are the following three standard simple  morphisms which together with the isomorphisms generate all graph morphisms:
\begin{enumerate}
\renewcommand{\theenumi}{\roman{enumi}}
\item A {\em simple grafting} is the grafting of two flags $s,t$ of a graph $\G$ into an edge. As a morphism $\gl{s}{t}:\G \to \G'$ is given by $(id_V,id_F,\imath)$ where $\G'=(V,F,\imath')$ with $\imath'(s)=t$ and $\imath'(f)=\imath(f)$ if $f\neq s,t$.

\item A {\em simple contraction} $c_e=c_{\{s,t\}}$ of an edge $e=\{s,t\}$ of a graph $\G$ is given by $\G\to \G'$ where $V'=V/\sim$ where $\del(s)\sim\del (t)$, $F'=F\setminus \{s,t,\}$, $\imath'=\imath|_{F'}$, and the morphism is given by the quotient map $\phi_V:V\to V'$, the inclusion  $\phi^F(F')\hookrightarrow F$ and  $\imath_\ph(s)=t$.  Note that if $\{s,t\}$ is a loop then $V'=V$.
\item  A {\em simple merger} is the merging of two vertices $v$ and $w$. As a morphism it is given by $\merger{v}{w}:\G\to \G'$, where $V'=V/\sim$ where $v\sim w$, $F'=F$ and $\imath'=\imath$ with the morphism given by $\pi:V\to V'$ the projection, $id_F$ and the empty map $\imath_\emptyset$.
\end{enumerate}

A morphism which is the composition of simple morphisms is called {\em pure}. Both the inclusions $i_\G$   and the total contractions $c_\G$  are pure.

\begin{prop}
\label{prop:rel}
The following relations hold.

\begin{enumerate}

\item {\em Generators} commute among themselves. Given two pairs of outer flags $s,t$ and $s',t'$, two edges $e_1=\{s_1,t_1\}$ and $e_2=\{s_2,\bar t_2\}$, two pairs of vertices $v,w$ and $v',w'$
\begin{equation}
\label{eq:grquadratic1}
\gl{s}{t}\gl{s'}{t'} =\gl{s'}{t'}\gl{s}{t}\quad c_{e_2}c_{e_1}=c_{e_1}c_{e_2} \quad \mge{v}{w}\mge{v'}{w'} =\mge{v'}{w'}\mge{v}{w}
\end{equation}
Note that these morphisms form  commutative squares, which is hidden in the compact notation. For instance, the targets of $c_{e_1}$ and of $c_{e_2}$  do not coincide.

\item {\em Mixed relations.} Graftings commute with the other generators.
\begin{equation}
\label{eq:grquadratic2}
\gl{s'}{t'}c_e=c_e\gl{s'}{t'} \quad \gl{s}{t}\mge{v}{w}=\mge{v}{w}\gl{s}{t}
\end{equation}
For simple mergers there are two different relations:
\begin{equation}
c_e\mge{v}{w}=\begin{cases} \mge{v}{w}c_e &\text{ If } \{v,w\}\neq \{\del s,\del t\}\\
c_e&\text{ if }   \{v,w\}= \{\del s,\del t\}
\end{cases}
\end{equation}
Note in the last relation on the left hand side there is a simple loop contraction, while on the right hand side there is a simple edge contraction.

\item {\em Isomorphisms}. Isomorphisms and simple morphisms are crossed  in the following sense. Given an isomorphism $\sigma$ there exists a unique isomorphism $\sigma'$ such that
\begin{equation}
\label{eq:iso}
 \gl{s}{t}\sigma = \sigma \gl{\sigma^F(s)}{\sigma^F(t)} \quad   \mge{\sigma_V(v)}{\sigma_V(w)}\sigma = \sigma'  \mge{v}{w} \quad c_{\{s,t\}}\sigma=\sigma' c_{\{\sigma^F(s),\sigma^F(t)\}}
\end{equation}
\end{enumerate}
In the case of a merger $\sigma^{\prime F}=\sigma^F$, $\sigma'_V(w)=\sigma(w)$ for $w\neq u,v$ and $\sigma'(\{u,v\})=\{\sigma_V(u),\sigma_V(v)\}$ to conform with the conventions of the  simple mergers.
In the case of a contraction $\sigma^{\prime F}=\sigma^F|_{F\setminus \{\sigma^F(s),\sigma^F(t)\}}$ and in the case of a loop contraction $\sigma'_V=\sigma_V$
while for a non--loop contraction $\sigma'_V(w)=\sigma(w)$ for $w\neq u,v$ and $\sigma'(\{u,v\})=\{\sigma_V(u),\sigma_V(v)\}$.
\end{prop}
\begin{proof}
These are straightforward computations.
\end{proof}

An example of a grafting composed with an edge contraction is given in Figure \ref{fig:ghostgraph}.
\begin{figure}[h]
    \centering
    \includegraphics[width=.45\textwidth]{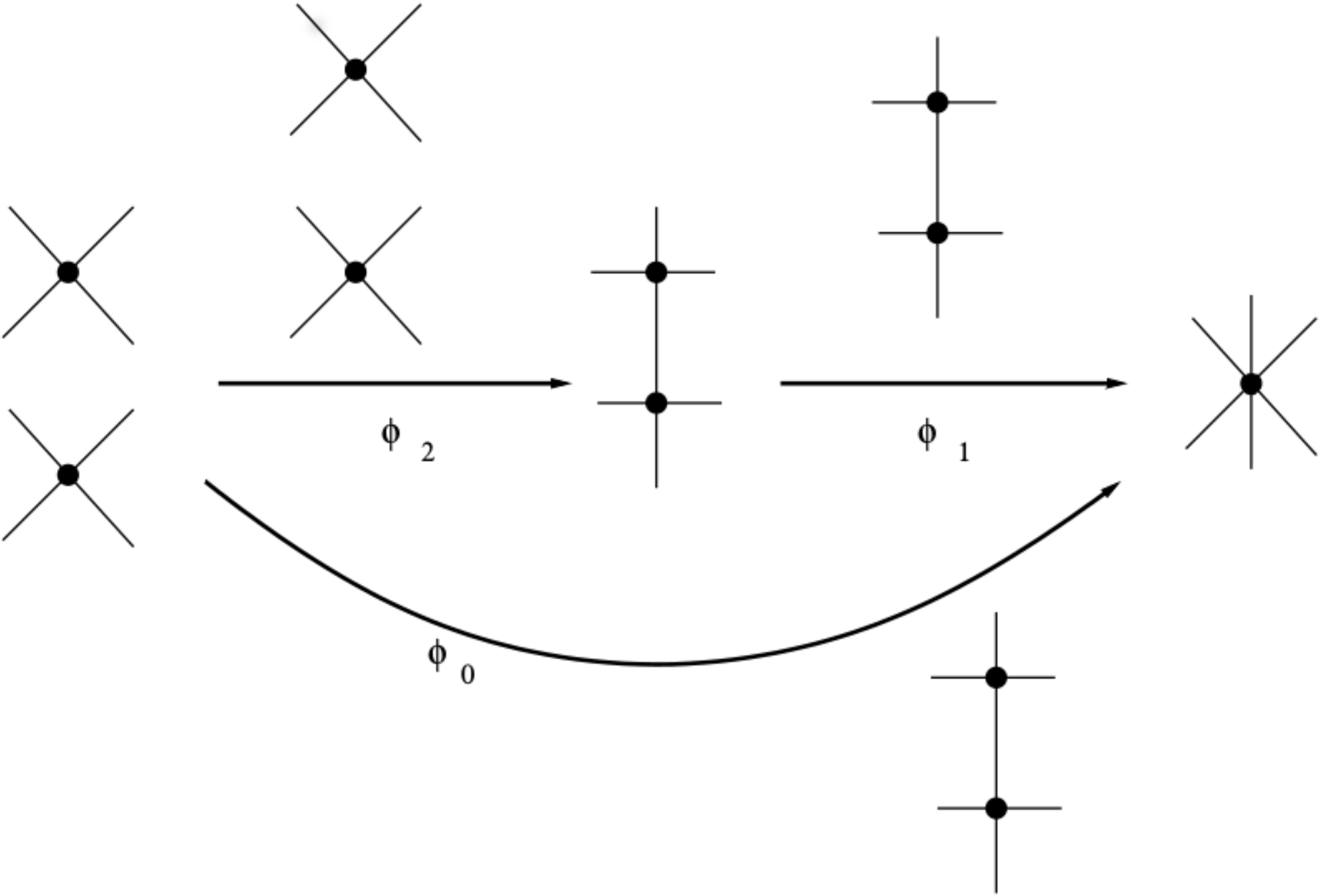}
    \caption{A composition of morphisms with ghost graphs, the first is a grafting, the second an edge contraction, the composition is a virtual edge contraction.}
    \label{fig:ghostgraph}
\end{figure}

\begin{thm}[Structure Theorem I]
\label{thm:graphstructure}
Every morphism in $\Graph$ can be uniquely factored in two ways
\begin{equation}
\phi=\sigma\phi_{con}\phi_{gr}\phi_m=\sigma\phi_m\phi_{con}\phi_{gr}
\end{equation}
where $\sigma$ is an isomorphism, $\phi_m$ is a pure merger, $\phi_{gr}$ is a pure grafting and $\phi_{con}$ is a pure contraction. In the second decomposition  $\phi_{con}$ can be further decomposed non--uniquely as $\phi_{con\mdash l}\phi_{con\mdash f}$ where $\phi_{con\mdash l}$ is a pure loop contraction and $\phi_{con\mdash f}$ is a pure forest contraction.\end{thm}
 \label{thm:grstructure}
\begin{proof} For $\phi:\G\to \G'$ factor $\phi_V:V\twoheadrightarrow V'$ as $\phi_V=\sigma_V\pi$ where $\pi$ is the quotient map $V\to \bar V$ identifying the fibers and set
 $\phi_m=(\pi,id_F,\imath_\emptyset):(V,F,\imath)\to \G_m:= (\bar V,F,\imath)$, which is a composition of simple mergers.
  Next, factor the inclusion $\phi^F:F'\to F$ as $\sigma^Fi$ where $i:\phi^F(F')\hookrightarrow F$ is the inclusion of the image. Define $\G_{gr}$  to be the graph $(\bar V,F,\imath_{gr})$ which has the edges of $\G'$ and the ghost edges of $\phi$. That is
     $\imath_{gr}(\phi^F(f'))=\phi^F(\imath'(f'))$  and $\imath_{gr}(f)=\imath_\phi(f)$ if $f\notin \phi^F(F')$ and let $\phi_{gr}:\G_m\to \G_{gr}$ be given by the map $(id_{\bar v},id_F,\imath_{\emptyset})$.
     In the next step, let $\bar \G=(\bar V,\phi^F(F'),\phi^F\imath')$ and
   define $\phi_{con}:\G_{gr}\to \bar \G$ by $(id_{\bar V},i,\imath_\phi)$. Finally define $\sigma=(\sigma_V,\sigma^F,\imath_\emptyset)$.
   It is clear that this is a decomposition. The uniqueness follows from the construction.

     The rest of the statement follows from Proposition \ref{prop:rel}.
\end{proof}

\begin{rmk}
Note that postcomposing with an isomorphism leaves the ghost graph invariant $\ghost(\phi)=\ghost(\sigma\phi)$ while precomposing with an isomorphisms yields an isomorphic ghost graph with the isomorphism of graphs induced by $\sigma$:
$\sigma: \ghost(\phi)\stackrel{\sim}{\to} \ghost(\phi\sigma)$  where by abuse of notation $\sigma$ is the morphism having the same components, but having a different source and target.
The non--uniqueness of the decomposition of $\phi_{con}$ into loop and subforest contractions is precisely given by choosing one edge in  each cycle of $\ghost(\phi)$ which when contracted last is the loop contraction. Equivalently, this decomposition is fixed by a spanning tree for each component of $\ghost(\phi)$.
\end{rmk}

The following two corollaries are immediate.
 \begin{cor}
 \label{cor:grcrossed}
 The isomorphisms together with $\mge{v}{w},\gl{s}{t},c_e$ generate $\Graph$ with quadratic and triangular relations.
The isomorphisms together with $\mge{v}{w},\gl{s}{t}$ and simple loop contractions generate a subcategory with quadratic relations.

For any $\phi$ and any precomposable isomorphism $\sigma$ there are unique $\phi'$ and $\sigma'$ with $\gh(\phi\sigma)=\gh(\phi')$ such that

\begin{equation}\label{eq:crossed}
\phi \sigma =\sigma' \phi'
\end{equation}
and, moreover, this makes $\Graph$ into a crossed product of pure morphisms and isomorphisms. \qed
\end{cor}

\begin{cor}
\label{cor:graphcatdef}
Restricting the types of generators yield wide subcategories:
\begin{enumerate}
\renewcommand{\theenumi}{\roman{enumi}}
\item Isomorphisms, graftings and contractions define a wide subcategory $\Graph^{\rm ctd}$.\\
 Every morphism in $\Graph^{\rm ctd}$ can be uniquely factorized as $\phi=\sigma\phi_{con}\phi_{gr}$.
\item Isomorphisms and contractions define a wide subcategory $\Graph_{contr}$ of $\Graph^{\rm ctd}$.
\item Isomorphisms and forest contractions define a wide subcategory $\Graph^{\rm forest}$ of $\Graph^{\rm contr}$.
\item Isomorphisms  and graftings define a wide subcategory $\Graph^{\rm graft}$ of $\Graph^{\rm ctd}$.
\end{enumerate}
 \end{cor}

\subsection{Simple generators for morphisms between aggregates}\label{par:agg}
An important role is played by the full subcategory $\Agg$ of $\Graph$ whose objects are aggregates. It underlies the categories relevant for operadic structures including those in algebra and geometry. In this category, general graphs  appear as ghost graphs of the morphisms.

Graftings do not preserve aggregates and aggregates have no edges which can be contracted. Instead the following morphism takes over their role as generators.
\begin{enumerate}
\renewcommand{\theenumi}{\roman{enumi}}
\setcounter{enumi}{3}
\item A  {\em  simple virtual edge/loop contraction} of two flags is $\scirct:=\con{\{s,t\}}\gl{s}{t}$. If $\{s,t\}$ is a loop, we write $\circ_{st}=\con{\{s,t\}}\gl{s}{t}$. (This notation is in accordance with the standard notation for modular operads.)
\end{enumerate}
An example of a virtual edge contraction as a grafting followed by a merger is depicted in Figure \ref{fig:ghostgraph}.
The generating morphisms are given in Figure \ref{fig:basicmorphisms}.

\begin{figure}[h]
    \centering
    \includegraphics[width=.8\textwidth]{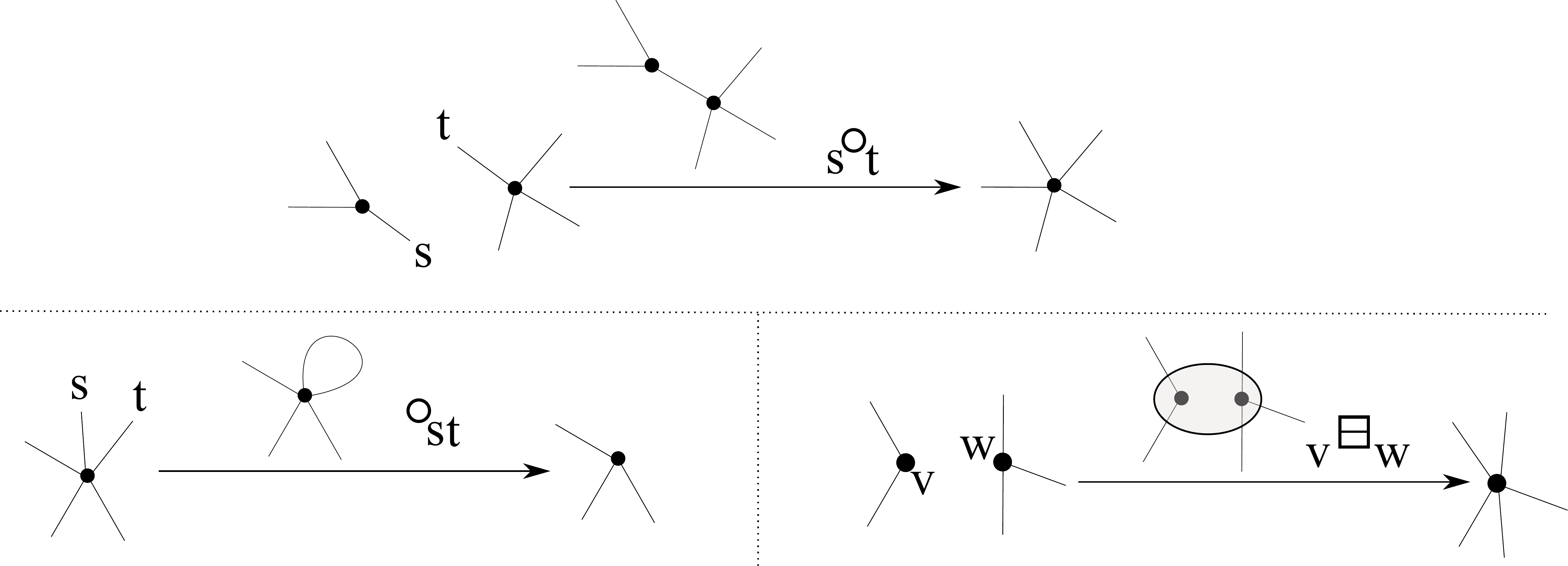}
\caption{Generating morphisms of $\Agg$, virtual edge contractions ${}_s\circ_t$, virtual loop contractions $\circ_{st}$
and mergers ${}_v\boxminus_w$ and their ghost graphs.
    The shaded region for the merger indicates the data $\phi_V$
    \label{fig:basicmorphisms}}
\end{figure}
\begin{cor}
The relations among the virtual edge/loop contractions are
\begin{equation}
\begin{aligned}
\scirct \ccirc{s'}{t'} =  \ccirc{s'}{t'} \scirct & \text{ if } \{\del s,\del t\}\neq \{\del s',\del t'\} \\
\circ_{s't'} \scirct = \circ_{st} \ccirc{s'}{t'}&\text{ if } \{\del s,\del t\}= \{\del s',\del t'\}\\
\end{aligned}
\end{equation}
with the mixed relations
\begin{equation}
\begin{aligned}
\label{eq:triangle}
\scirct\mge{v}{w} = \mge{v}{w}\scirct  &\text { if }  \{\del s,\del t\}\neq\{v,w\}\\
\scirct = \circ_{st}\mge{v}{w} &\text { if }  \{\del s,\del t\}=\{v,w\}\\
\circ_{st}\mge{v}{w}=\mge{v}{w}\circ_{st}\\
\end{aligned}
\end{equation}
The relation with isomorphisms are
\begin{equation}
\scirct \sigma=\sigma'\ccirc{\sigma^F(s)}{\sigma^F(t)} \quad \circ_{st}\sigma=\sigma' \circ_{\sigma^F(s)\sigma^F(t)}
\end{equation}
with $\sigma'$ as in \eqref{eq:iso}.
\end{cor}
\begin{proof} These follow directly from Proposition \ref{prop:rel}.
\end{proof}
  This description was obtained directly in \cite[\S 5]{feynman}.
A pure morphism is again a morphism that is a composition of thes $\mge{v}{w},\circ_{st}$ and $\scirct$.
Note that any graph $\G$ appears as the ghost graph of the pure morphism $\vc_\G=c_\G i_\G$ (called
 the virtual contraction of $\G$) whose source aggregate is the total dissection of $\G$ and whose target aggregate is its total contraction. As the ghost graph is invariant under post--composing with isomorphisms or mergers  this is not the only morphism whose ghost graph is $\G$, in general.

Let $\Agg^{\rm ctd}$ be the full subcategory of $\Graph^{\rm ctd}$ whose objects are aggregates. The morphisms in $\Agg^{\rm  ctd}$ are precisely those for which the connected components of the ghost graphs are connected graphs. Note that this property fails for the category $\Agg$ because of the existence of mergers.  Define the subcategory $\Agg^{\rm forest}$
by restricting to forests as ghost graphs.

\begin{cor}\label{cor:ghostgraph}
For a pure morphism $\phi:X\to Y$ in $\Agg^{\rm ctd}$, the source may be identifed with the total dissection $\agg(\ghost(\phi))$, the target with the total contraction $\ghost(\phi)/\ghost(\phi)$, and the morphism itself with $\vc_{\ghost(\phi)}=c_{\ghost(\phi)}i_{\ghost(\phi)}$. Therefore, any pure morphism in $\Agg^{\rm ctd}$ is uniquely determined by its ghost graph.\end{cor}
\begin{proof}
By the relations \eqref{eq:grquadratic1} one can preform all the grafting before the contractions. By definition of virtual contraction, all glued edges are contracted.
\end{proof}

\begin{thm}[Structure Theorem II]
\label{thm:aggstructure}
Every morphism in $\Agg$ can be uniquely factored as $\sigma\phi_p$ with $\phi_p$ is a pure morphism and this decomposition can be further uniquely factored in two ways
\begin{equation}
\phi=\sigma \phi_{con}\phi_m=\sigma\phi_m\phi_{con}
\end{equation}
where $\sigma$ is an isomorphism, $\phi_m$ is a pure merger, and $\phi_{con}$ is a pure contraction, which in the first decomposition is a loop contraction. In the second decomposition,  $\phi_{con}$ can be further decomposed non--uniquely as
$\phi_{con}=\phi_{con\mdash l}\phi_{con\mdash f}$ where $\phi_{con\mdash l}$ is a pure loop contraction and $\phi_{con\mdash f}$ is a pure forest contraction.

Morphisms in $\Agg^{\rm ctd}$ uniquely decompose as $\phi=\sigma\phi_{con}$ and in $\Agg^{\rm forest}$ as $\phi=\sigma \phi_{con\mdash f}$.\end{thm}
\begin{proof}Since an aggregate has no edges, there are no contractions of actual edges and any grafted edge has to be contracted. The claims then follow from Theorem \ref{thm:grstructure} and Proposition \ref{prop:rel}.\end{proof}

\begin{rmk}
 \label{rmk:data}
 This theorem was obtained in \protect{\cite[\S4]{matrix}} and  formalizes the observations in \cite[\S 2.1.1]{feynman}.
 The ghost graph captures $\phi_{con}$ while the additional data to determine $\phi$ is the isomorphism $\sigma$, and in the presence of mergers, the information which components of the ghost graph belong to the same fibers of $\phi$.
 \end{rmk}

 \begin{cor}
 \label{cor:aggcrossed} The category $\Agg^{\rm ctd}$ is a crossed product of pure morphisms and isomorphisms.
  \end{cor}
\begin{proof}
Immediate from Corollary \ref{cor:grcrossed}.
\end{proof}

\begin{rmk}The category of graphs of \cite{BM} has a slightly different definition of $\imath_\phi$ which is only defined on the outer flags of $F\setminus \phi^F(F')$. This is equivalent to the current definition, the  canonical extension to the inner flags that are not in the image of $\phi^F$ being given by $\imath_\G$. By restricting to outer flags this keep track only of virtually contracted edges, while not doing this restriction keeps track of all contracted edges, virtual and actual. The ghost graph for the subcategory $\Agg$ and the composition of ghost graphs  was defined in \cite{feynman} and elaborated upon in \cite{matrix}, which also contains the tweak to the definitions of \cite{BM} used here.\end{rmk}

\subsection{Double category relating graphs and aggregates}
\label{par:double}

The two  categories $\Graphs$ and $\Agg$ are related by three functors. The first is the inclusion functor $i:\Agg \to \Graph$. The other two functors
are given by total dissection $s=\agg:\Graph\to \Agg$ and total contraction $t=-/-:\Graph \to \Agg$. From now on, we shall call $s(\G)$, resp. $t(\G)$ source and target aggregate of $\G$. We have the relations $s\circ i= t\circ i=id_{\Agg}$. The values of the source/target aggregate functors on simple generators are as follows:

\begin{equation}
s(\sigma)=\sigma\quad s(\mge{v}{w})=\mge{v}{w} \quad s(\gl{s}{t})=id
 \quad s(c_{\{s.t\}})=\begin{cases} \scirct &\text{if } \del s\neq \del t\\
 \circ_{st} &\text{if } \del s= \del t  \end{cases}
 \end{equation}
and
\begin{equation}
t(\sigma)=\sigma\quad   t(\mge{v}{w})=\begin{cases} \mge{\bar v}{\bar w}&\text{if } \bar v \neq \bar w\\
id &\text{if } \bar v= \bar  w \end{cases}  \quad t(\gl{s}{t})=\begin{cases} \scirct&\text{if } \del s = \del t\\
\circ_{st}&\text{if }\del s \neq \del t\end{cases} \quad t(c_{\{s,t\}})=id
  \end{equation}

The following proposition sheds light on graph insertion.

\begin{prop}There is a double category $(s,t):\Graph\rightrightarrows\Agg$ with internal identities $i:\Agg\to\Graph$ and internal composition ${}_t\circ_s:\Graph\,{}_t\!\!\times_s\Graph\to\Graph$ given by graph insertion.

The ghost graph induces a functor from the vertical category to the horizontal category of this double category.\end{prop}

\begin{proof}The proof reduces essentially to the observation that a graph $\G_1$ can be inserted into the vertices of the graph $\G_2$ whenever the target aggregate of $\G_1$ coincides with the source aggregate of $\G_2$, and that this graph insertion is compatible with graph morphisms on both sides. The second assertion amounts to the fact that composition of ghost graphs coincides with graph insertion.\end{proof}

Recall that a double category has $0$-cells (here the objects of $\Agg$), vertical $1$-cells (the morphisms of $\Agg$), horizontal $1$-cells (the objects of $\Graph$) and $2$-cells (the morphisms of $\Graph$). These $2$-cells are denoted by squares as in \eqref{eq:ghostmorph}, and compose as well vertically as well horizontally, with the obvious associativity and interchange relations.

\begin{prop} \label{prop:uniquefiller}
Any $2$--cell $\Phi:\G\to \G'$ defines a commutative diagram in $\Agg$ whose horizontal morphisms are pure  and  any such diagram  inversely defines a $2$ cell $\Phi$.

\begin{equation}
\label{eq:ghostmorph}
 \xymatrix{
X\ar[r]^{\G}\ar[d]_{\phi_L} \ar@{}[dr]|{\Downarrow  \, \Phi}&X'\ar[d]^{\hat \phi_R}\\
Y\ar[r]_{\G'}&Y'\\
}
\quad \raisebox{-5mm}{$\leftrightarrow$}\quad
\xymatrix{
X\ar[r]^{v_{\G}}\ar[d]_{\phi_L}  &X' \ar[d]^{\phi_R}\\
Y\ar[r]_{v_{\G'}}&Y'\\
}
\end{equation}

Moreover, there is a morphism of double categories $\Graph$ to  $\Box \Agg$, the double category of commutative squares in $\Agg$,
which is an isomorphism onto its image.
\end{prop}
\begin{proof}
Given $\Phi$, the square on the right is defined and we have to prove that it is commutative.
This is clear for the maps on vertices and flags. What remains to be checked are the involutions. These are determined by the ghost edges.
Now given any a commutative diagram $\phi_R\phi=\phi_L\phi'$

by definition, we have that
\begin{equation}
\label{eq:edgedecomp}
E_{\ghost (\phi_R\phi)}=E_{\ghost(\phi)}\sqcup E_{\ghost(\phi_R)}=E_{\ghost(\phi_L)}\sqcup E_{\ghost(\phi')}=E_{\ghost(\phi'\phi)}
\end{equation}
and vice--versa, if this equation holds then the composition of involutions agree.
Thus it remains to show this for the morphisms at hand. Since $\Phi$ is a morphism in $\Graph$ we have that the edges of $\Phi$ are either contracted or preserved. That is $E_\G\subset E_{\G'}\sqcup E_{\ghost(\phi_L)}$ and exhaust all edges of $\ghost(\phi_L)$ by definition.
The remaining edges of $E_{\G'}$ are exactly the ghost edges of $\phi_R$. Indeed, factoring the morphism $\Phi$ these are the newly grafted edges, which are contracted by $\phi_R$.
Conversely, by Corollary \ref{cor:ghostgraph}, we can assume the square is of the form on the right. Now $\phi_L$ defines a putative morphisms $\G\to \G'$ as the flags, vertices  of $\G$ and $\agg(\G)$  agree, as well as those of $\G'$ and $\agg(\G')$. It remains to check the compatibilities. The compatibility with $\del$ is also provided by $\agg$. For the compatibility of the involution, we need that
$E_\G\subset E_{\G'}$, which is guaranteed by \eqref{eq:edgedecomp}. Thus $\Phi:\G\to \G'$ is well defined and $s(\Phi)=\phi_L$, so it only remains to show that $t(\Phi)=\phi_R$. This is again clear for the morphisms on vertices and flags. For the involution, this follows from \eqref{eq:edgedecomp} as $\phi_R$ virtually contracts the inverse image of edges of $\G'$ which are not edges of $\G$.  This is clearly functorial under vertical composition, for horizontal compositions this is true by definition of the horizontal composition of $2$--cells using $\G=\ghost(v_\G)$.
\end{proof}

Restricting the vertical morphisms to be pure, this yields a thin structure.
\begin{cor}\label{cor:graphiso}
For a  morphism $\phi$ in $\Agg^{\rm ctd}$ there is an isomorphism $\Aut(\phi)$ in the category of arrows  with the automorphism group of $\ghost(\phi)$.
\end{cor}
\begin{proof}
Every such morphism $\phi$ is isomorphic to a pure morphism. For pure morphisms this follows from the proposition above.
\end{proof}

\begin{prop}\label{prop:holonomy}Restricting to pure morphisms in $\Graph^{\rm ctd}$ and $\Agg^{\rm ctd}$ the ghost graph is a holonomy  and the restricted double category has a pair of connections $\lrcorner(\phi)=c_{\ghost(\phi)}$ and $\ulcorner(\phi)={\rm gr}_{\ghost(\G)}$ and is thin.

\begin{equation}
\xymatrix{
X\ar[r]^{\ghost(\phi_p)}\ar[d]_{\phi_p} \ar@{}[dr]|{\Downarrow  \, c_{\ghost(\phi_p)}}&Y\ar[d]^{id}\\
Y\ar[r]_{Y}&Y
}
\quad
\xymatrix{
X\ar[r]^{X}\ar[d]_{id} \ar@{}[dr]|{ \rm{gr}_{\ghost(\phi_p)}\, \Downarrow }&X\ar[d]^{\phi_p}\\
X\ar[r]_{\ghost(\phi_p)}&Y
}
\end{equation}

 \end{prop}
 \begin{proof}
 Straightforward from the definitions.
 \end{proof}

\begin{rmk}Connections and holonomy and thin structure were introduced in \cite{BrownGhafar}. The second part of Proposition \ref{prop:holonomy} shows that in a sense the category $\Graph$ is deficient, since it does not encode all the morphisms of $\Agg$, but only the pure connected ones. This is due to the definition of the source and target aggregates. To capture the whole picture one would need to add the isomorphisms and the data of $\phi_V$, see Remark \ref{rmk:data}.
\end{rmk}

\section{Decorating functors}
\label{par:decofun}
Decorations are best understood in terms of functors and their \emph{categories of elements} in the sense of Grothendieck. Given a category $\F$ and a set-valued functor $\O:\F\to \Set$ the Grothendieck construction defines a new category $\F_{\rm dec}(\O)$ whose objects are pairs $(X,a_x)$, consisting of an object $X$ of $\F$ and a ``decorating'' element $a_x$ of $\O(X)$, and whose morphisms are given by$$\F_{\rm dec}(\O)((X,a_x),(Y,a_y))=\{\phi\in \F(X,Y):\O(\phi)(a_x)=a_y\}.$$
There is a projection functor $\pi:\F_{\rm dec}(\O)\to \F$ taking $(X,a_x)$ to $X$. Each natural transformation $\eta:\O\to\O'$ induces a functor $\F_{\rm dec}(\O)\to \F_{\rm dec}(\O')$ taking $(X,a_x)$ to $(X,\eta_X(a_x))$, and compatible with the respective projection functors.

In order to apply Grothendieck's construction, we thus have to promote  decorations to functors. We now define such ``decorating'' functors on the categories of the preceding chapter.
In $\Graph$ they decorate the graphs directly, in $\Agg$ they decorate aggregates and thereby the ghost graphs by pullback with respect to the source aggregate.

\subsection{Genus labeling as a decoration}
\label{par:genus}
The genus labeling on the vertices can be promoted to a \emph{monoidal functor} $\genus:(\Graph,\sqcup)\to (\Set,\times)$.
 On objects it is given by
$\genus(\G)=\N_0^{V_\G}$ and on generators as follows:   for $g:V_\G\to \N_0$
\begin{equation}
\begin{aligned}
\genus(\sigma)(g)(v')&=g(\sigma_V^{-1}(v')) &
\genus(\mge{v}{w})(g)(\{v,w\})&=g(v)+g(w)-1
\\
 \genus(\gl{s}{t})&=id
&(\genus{c_{\{s,t\}}}(g))(\{\del s,\del t\})&=\begin{cases}
g(\del s)+g(\del t)&\text { if } \del s \neq  \del t \\
g(\del s)+1&\text { if } \del s =  \del t
\end{cases}
\end{aligned}
\end{equation}
where all other non-indicated values of $g$ are unchanged.

A global formula is given by considering the fibers of the ghost graphs $\gh(\phi)_{v'}$ given by the vertices $\phi^{-1}(v)$, outer flags, $\phi^{F}(F_{v'})$ and the ghost edges whose vertices lie in $\phi^{-1}(v')$.
\begin{equation}
(\genus(\phi)(g))(v')=\sum_{ v\in \phi_V^{-1}(v)}[g(v)+(1-\chi(\ghost(\phi_{v}))]
\end{equation}
One easily computes that this is a functor using the relation or that composition of morphisms corresponds to the composition product of the ghost graphs.

This restricts to $\Agg$  along the inclusion $i:\Agg\to\Graph$ and gives the following values on the simple generators of $\Agg$:
\begin{equation}
\begin{aligned}
 \genus(\scirct)|_{\{\del s, \del t\}}:\N_0\times \N_0\overset{+}{\to} \N_0, \quad
  \genus(\circ_{st})|_{\del s\}}:\N_0\overset{+1}{\to}\N_0:n\mapsto n+1,\\
\genus(\mge{v}{w})|_{\{v,w\}}:\N_0\times\N_0\to\N_0:(m,n)\mapsto m+n-1.
\end{aligned}
\end{equation}
with other non-indicated values being of $g$ are unchanged.

The objects of the element category $\Graph_{\rm dec}(\genus)$ are genus labelled graphs. The restriction to $\Agg^{\rm ctd}$ is related to modular operads via $\Agg^{\rm ctd}_{\rm dec}(\genus)$, see \S \ref{par:feyoperads}.
On mergers this functor is determined by \eqref{eq:triangle}. This extension was introduced in \cite[\S7.A.3]{KWZ} to define the non--connected (nc) version of modular operads. The grading in {\em loc.\ cit.} is as in the open gluing of \cite[Appendix A]{KP} and corresponds to $1-\chi$.

\begin{rmk} The name genus stems from the connection to modular operads and the moduli spaces $\bar M_{g,n}$. A more appropriate interpretation and
the right combinatorial notion
is $1-\chi$, this is the negative reduced Euler characteristic. For a connected graph $1-\chi(\G)=b_1$ is the loop number. This is
 closely related to the geometric genus of associated surface representations, cf. \S\ref{par:surfinterpret} and Lemma \ref{lem:genus} below.
\end{rmk}

The extension to $\Graph$ above is the pull--back along $s$: $\genus=s^*(\genus i)=\genus i s$. There is also the pull--back by $t$ that is $\genus i\, t$. The objects of the element category ${\Graph_{\rm dec}}({\genus} i \, t)$ are  graphs whose components are marked by a genus.

\subsection{Polycyclic orders as decoration}
For each graph $\G$ define
\begin{equation}
\poly(\G)=\{\text{Polycyclic orders on each of the corollas $F_v$}\}\cong \prod_{v\in V}\Aut(F_v)
\end{equation}
This extends to a monoidal functor with respect to disjoint union of graphs. We define the action of the functor $\poly$ on simple generators of $\Graph$. On graph isomorphisms $\poly$ acts by relabeling, i.e.\ via conjugation by $\phi^F$. On graftings $\poly$ is the identity and on mergers $\poly$ is just the union of the polycyclic structures. For simple contractions, if $\{s,t\}$ is not a loop, i.e. $\del s \neq \del t$ then $\poly(c_{\{s,t\}})(\next_S,\next_T)$ is the polycyclic order induced by $\next_S\next_T\tau_{st}$ on $S\setminus \{s\} \sqcup T\setminus \{t\})$, where $\tau_{st}$ interchanges $s$ and $t$. Explicitly,
\begin{equation}
\poly(c_{\{s,t\}})(\next_S, \next_T)(x)=\begin{cases} \next_S(x)&\text{ if } x\in S\setminus \{s, \next_S^{-1}(s)\}\\
\next_T(x)& \text{ if }   x\in S\setminus \{t, \next_T^{-1}(t)\}\\
\next_T(t)& \text{ if } x=\next_S^{-1}(s)\\
\next_S(s)& \text{ if } x= \next_T^{-1}(t)\\
\end{cases}
\end{equation}
If $\{s,s'\}$ is a loop, i.e. $\del s =\del s'$ then $\poly(c_{\{s,s'\}})(\next_S)$ is the polycyclic order on $S\setminus\{s,s'\}$ given by:
\begin{equation}
\poly(c_{\{s,s'\}})(\next_S)(x)=\begin{cases}
 \next_S(x)&\text{ if } x\in S\setminus \{s, s', \next_S^{-1}(s),\next_S^{-1}(s')\}\\
\next_S(s')& \text{ if } x=\next_S^{-1}(s)\\
\next_S(s)& \text{ if } x= \next_S^{-1}(s')
\end{cases}
\end{equation}
There is no simple formula for general graph morphisms $\phi$, since the action does not only depend on the local structures.

Objects of the element category $\Graph_{\rm dec}(\poly)$ are polycyclic graphs.
Polycyclic graphs have also been called \emph{almost ribbon graphs} \cite{postnikov} or \emph{stable ribbon graphs} \cite{KontsevichAiry,Barrannikov} and  occur in the combinatorial compactification of moduli spaces, cf. \cite{Penner,postnikov}. Notice that there are several distinct notions of stability in this context, certain involving the condition of having at least trivalent vertices and/or the condition of having negative Euler characteristic, which we currently do not impose.
\begin{rmk}
Starting with a cyclic order $\next$ on a set $S$, and performing a loop contraction produces a polycyclic order on $S\setminus\{s,s'\}$ and similarly for mergers. This explains why polycyclic orders are unavoidable in modular situations. It is a tedious and futile, but sobering, exercise to try to  find  a well-defined self--composition for cyclic structures. As test case the reader should consider the mutation of Figure \ref{fig:mutation} for any definition that is proposed and see that the putative structure  will not be well defined.
\end{rmk}

\begin{lem}
Ribbon graphs form  a  full subcategory  $\Rib^{\rm forest}$ of $\Graph^{\rm forest}_{\rm dec}(\poly)$ and $\Rib^{\rm graft}$ of $\Graph^{\rm graft}_{\rm dec}(\poly)$, but are not stable with respect to all of the operations of $\Graph_{\rm dec}(\poly)$.
\end{lem}

\begin{proof}Isomorphisms preserve the number of orbits. A non--loop contraction of two cyclic orders $\sigma_S$ and $\sigma_T$ again produces  a cyclic order $\poly(c_{\{s,t\}})(\sigma_S,\sigma_T)$.  Graftings act as identities. Both loop contractions and mergers, however, produce polycyclic orders  if the input consists of cyclic orders only.
\end{proof}

The full ribbon graph subcategory of $\Graph^{\rm graft}_{\rm dec}(\poly)$ has been considered in \cite{WW} while the full ribbon graph subcategory of $\Graph^{\rm forest}_{\rm dec}(\rm \poly)$ has been used by \cite{Igusa} and is crucial for us, cf. \S\ref{par:compu}.

\begin{rmk}
If mergers are present, a polycyclic order is always the result of merging cyclic orders. This relates to the definition of stable ribbon graph of Kontsevich \cite{KontsevichAiry}, cf.\ Proposition \ref{prop:KPK}.
\end{rmk}

The functor $\poly$ restricts to aggregates and corollas as follows:
\begin{equation}
\poly(*_S)=\{\text{Polycyclic orders $\sigma_S$  on $S$}\}\cong Aut(S)
\end{equation}
On isomorphisms the $\poly$ acts via conjugation with $\phi^F$ as before. Polycyclic orders compose under mergers as before:
$\poly(\mge{v}{w})(\sigma_S,\sigma_T)=\sigma_S\times \sigma_T$.
On virtual contractions:
\begin{align}
\label{nonselfpolyeq}
\poly(\scirct)(\next_S,\next_T)(x)&=\begin{cases} \next_S(x)&\text{ if } x\in S\setminus \{s, \next_S^{-1}(s)\}\\
\next_T(x)& \text{ if }   x\in S\setminus \{t, \next_T^{-1}(t)\}\\
\next_T(t)& \text{ if } x=\next_S^{-1}(s)\\
\next_S(s)& \text{ if } x= \next_T^{-1}(t)
\end{cases}\\
\label{eq:selfpolyeq}
\poly(\circ_{s,s'})(\next_S)(x)&=\begin{cases} \next_S(x)&\text{ if } x\in S\setminus \{s, s', \next_S^{-1}(s),\next_S^{-1}(s')\}\\
\next_S(s')& \text{ if } x=\next_S^{-1}(s)\\
\next_S(s)& \text{ if } x= \next_S^{-1}(s')
\end{cases}
\end{align}

Note that $\sigma_S\scirct \sigma_T$ corresponds to the usual block composition of permutations.
We see that for aggregates only $\scirct$ preserves cyclic orders. The restriction to the wide subcategory generated by these defines
$\Agg^{\rm pforest}\subset \Agg^{\rm dec}(\poly)$, where $pforest$ stands for planar forests, because the cyclic orders on the vertices induces a planar embedding of the forest. This embedding is unique up to isotopy.

Pulling back along $t:\Graph\to\Agg$ the restriction $i^*\poly:\Agg\to\Set$ defines graphs equipped with a polycyclic order of the set of the outer flags.
In other words, the objects of the element category $\Graph_{\rm dec}(t^*i^*\poly)$ can be thought of as  graphs with a polycyclically ordered set of outer flags. Cyclically ordered sets of outer flags are only stable under grafting and subforest contraction.

\begin{rmk}
The construction of the monoidal functor $\poly$ can performed analogously for the Feynman category $\FFcyc$ using a commutative monoid. Adding a unit and  a distinguished  element allows one to extend the construction to Feynman operations of $\GGctd$,
 and, whenever the distinguished element is invertible, further to Feynman operations of $\FFagg$, cf.\ \cite{decorated}.
\end{rmk}

\subsection{Oriented surface types as decoration}
\label{surfsec}

In order to handle all combinatorial data appearing in 1+1 d open TQFT and in the compactification of moduli spaces with punctures and marked boundaries, one needs to consider a decoration induced by a monoidal functor $\surf:\Agg\to \Set$. This functor has a rather intricate combinatorial description. Geometrically, it corresponds to gluing surfaces with marked points on the boundary, and taking disjoint unions.

We will define the functor on the category of aggregates. The homonymous functor on the category of graphs will be defined by pullback along the source aggregate $s:\Graph\to\Agg$. Categorically the monoidal functor $\surf$ arises as a pushforward of the functor $\CycAss$ as we will see in \S\ref{par:compu}.

Since $\surf$ is monoidal, it suffices to define it on corollas $*_S$ where $S$ is the flag set. We put
\begin{equation}
\surf(*_S)=\N_0\times \N_0\times\Aut(S)
\end{equation}
An element of $\surf(*_S)$ is thus a triple $(g,p,\sigma_S)$ where $g$ and $p$ are natural numbers and $\sigma_S$ is a polycyclic order on the flag set $S$. Note that, due to the specific action of the morphisms of $\Agg$, the functor $\surf$ is {\em not} a direct product of functors.

Isomorphisms act as identity on $g,p$ and by conjugation on the polycyclic order $\sigma_S$ as above. The action of a simple gluing is given by:
\begin{multline}
\label{eq:surfedge}
\surf(\scirct)((g,p,\sigma_S),(g',p',\sigma_T))=
\begin{cases}(g+g',p+p'+1,\sigma_S\scirct \sigma_T)&\text{if  $\sigma_S(s)=s$ and $\sigma_T(t)=t$}\\
(g+g',p+p',\sigma_S\scirct \sigma_T)&\text{otherwise}\\
\end{cases}
\end{multline}
The self gluing (i.e. the action of a simple loop contraction) is given by
\begin{equation}
\label{eq:surfloop}
\surf(\circ_{st})(g,p,\sigma_S)=
\begin{cases}
(g+1,p,\circ_{st}(\sigma_S))&\text{if $st$ are not in the same $\sigma$ orbit}\\
&\text{and both not fixed by $\sigma$}\\
(g+1,p+1,\circ_{st}(\sigma_S))&\text{if  $\sigma_S(s)=s$\text{ and}   $\sigma_S(t)=t$}\\
(g,p,\circ_{st}(\sigma_S))&\text{If $s$ and $t$ are in the same $\sigma$ orbit}\\
&\text{but neither $\sigma_S(s)=t$ nor $\sigma_S(t)=s$}\\
(g,p+1,\circ_{st}(\sigma_S))&\text{if either  $\sigma_S(s)=t$ or $\sigma_S(t)=s$}\\
(g,p+2,\circ_{st}(\sigma_S))&\text{if  $\sigma_S(s)=t$ and $\sigma_S(t)=s$}
\end{cases}
\end{equation}
The action of a simple merger is given by:
\begin{equation}
\label{eq:surfmerge}
\surf(\mge{v}{w})((g,p,\sigma_S),(g',p',\sigma_T)):=(g+g',p+p',\sigma_S\sqcup\sigma_T)
\end{equation}
In order to understand the role of the second factor $p$ it is convenient to view $\sigma_S$ as a polycyclic order on $S$ represented by the set of orbits of the permutation $\sigma$. The number $p$ is then the number of ``empty'' orbits so that the pair $(p,\sg_S)$ may be viewed as a polycyclic order on $S$ with $p$ empty orbits. This may illuminate certain of the formulas above.

\begin{rmk}[Topological type]
The geometric interpretation of the decoration $(g,p,\sigma_S)$ is a topological type of a \emph{bordered oriented surface}, i.e.\ of an oriented surface $\Sigma$ with boundary $\del\Sigma$ having $p$ unmarked boundary components and as many marked boundary components as $\sigma_S$ has orbits. The marking of an individual boundary component consists of a choice of marked points (or subintervals) in one-to-one correspondence with the elements of the corresponding orbit of $\sigma_S$ respecting the cyclic order coming from the orientation of the surface $\Sigma$. Moreover, the associated \emph{closed surface} should have \emph{topological genus} $g$, cf. \S\ref{sct:borderedsurface} and \cite[Figure 5]{KP}.

This topological interpretation of the functor $\surf$ appears in \cite{KP} as the connected components of the open part of a c/o structure, see \cite[Appendix A] {KP}\footnote{\cite{KP} is more general, since it allows for general $D$--brane labels. Here we have the case of only one such label.}. The gluing operations of the functor $\surf$ correspond then to the gluing of cobordisms for open two-dimensional topological field theory, see e.g.\ \cite{KP,LaudaPfeiffer,Doubek,hoch2,hochnote}.
\end{rmk}

\begin{lem}\label{lem:genus}
There is a natural transformation of monoidal functors $\surf\to\genus$ taking a connected topological type $(g,p,\sigma_S)$ to $1-\chi(\Sigma(g,p,\Sigma_S))=2g+p+b-1$, where $b$ is the number of cycles of
$\sigma_S$.\end{lem}

\begin{proof}Note first that the natural transformation is already determined on connected objects because of the monoidality of the functors $\surf$ and $\genus$. For a surface $\Sigma$ of type $(g,p,\sigma_S)$, the associated closed surface $\bar \Sigma$ has Euler characteristic $\chi(\bar \Sigma)=2-2g=\chi(\Sigma)+(p+b)$ and $1-\chi(\Sigma)=2g+p+b-1$. We now check naturality with respect to the action of simple generators of $\Agg$. We go through this case by case.

For $\scirct$ we get $2g_1+2g_2+p+q+n+m-2=2(g_1+g_2)+(p+q+1)+(m+n)-1$
and hence there is no change in  $1-\chi$.  In the second case, we have $2g_1+2g_2+p+q+n+m-2=2(g_1+g_2)+(p+q)+(m+n-1)-1$ and again no change.

For the self--gluing $\circ_{ss'}$, $1-\chi$ needs to increase by one.
The output  needs to be $[2g+p+s-1]+1=2g+p+s$. Indeed we get in the first subcase $2(g+1)+p+(s-1)-1$, in the second subcase $2(g+1)+p+1+s-2-1$, in the third subcase $2g+(p+1)+s-1$, in the fourth subcase $2g+p+1+s-1$ and in the final subcase $2g+p+2+(s-1)-1$ which are all equal to $2g+p+s$.

For a merger, the output  should be $2(g+g')+b+b'-2+1$ which it indeed is.
\end{proof}

The natural transformation $\surf\to \genus$ belongs to a hexagon of natural transformations between monoidal functors which have appeared at various places in literature.  Namely, one can retain only part of the decoration $(g,p,\sg_S)$. This results in the following diagram of monoidal functors and natural transformations between them.

\begin{prop}
\label{prop:hex}
There is a hexagon of monoidal functors $\Agg\to \Set$ and natural transformations
\begin{equation}
\begin{tikzcd}[column sep=small, row sep=tiny]
&\eulerpoly\ar[r]&\genus\ar[dr]&\\
\surf\ar[ur]\ar[dr]&&&\final^{\Agg}\\
&\polyN\ar[r]&\poly\ar[ur]&\\
\end{tikzcd}
\end{equation}
where
\begin{itemize}
\item $\eulerpoly(*_S)=\N_0\times\Aut(S)$ and $\surf\to \eulerpoly$ takes $(g,p,\sigma_S)$ to $(2g+p-1,\sigma_S)$. The action of simple generators is induced accordingly by summing the first two entries as above to the first two entries in \eqref{eq:surfedge}, \eqref{eq:surfloop}, \eqref{eq:surfmerge}.
\item $\eulerpoly\to \genus$ is given by $(l,\sigma_S)\mapsto l+b$ where $b$ is the number of orbits of $\sigma_S$.
\item $\polyN(*_S)=\N_0\times\Aut(S)$ and $\surf\to \polyN$ takes $(g,p,\sigma_S)$ to $(p,\sigma_S)$. The action of simple generators is given by projecting to the two last components in \eqref{eq:surfedge}, \eqref{eq:surfloop}, \eqref{eq:surfmerge}.
\item The natural transformation $\polyN\to\poly$ takes $(p,\sigma_S)$ to $\sigma_S$.
\end{itemize}
\end{prop}

\begin{proof}The naturality follows from a computation similar to the proof of Lemma \ref{lem:genus}.
\end{proof}
Geometrically, the monoidal functor $\eulerpoly$ decorates the corollas by surfaces with boundary and no punctures/unmarked boundaries, using $1-\chi$ to summarily keep  track of the puncture/genus labeling. The  number $2g+p-1$ is $1$ minus the Euler characteristic of the surface, where the boundaries have been filled, but the punctures are still there. On the other hand $\polyN$ only keeps track of the punctures, but forgets the genus. Note that neither ${\polyN}$ nor ${\poly}$ maps to ${\genus}$.

\begin{rmk}
The hexagon pulls back along $s$ or $t$ to functors $\Graph\to\Set$.
There are also associated hexagons for the associated element categories over $\Graph$ and $\Agg$.\end{rmk}

\subsection{Directed graphs and rooted forests}
\label{par:dir}
A graph with in/outputs is a graph equipped with a map $\io:F\to \Z/2$. A flag $f$ is called input (resp. output) flag if $(\io)(f)=1$ (resp. $(\io)(f)=0$). A graph is said to have \emph{directed edges} if each edge contains one input and one output flag. 

The map $\io$ can be promoted to a decorating functor $\Oio(\G)=F^{\Z/2\Z}$. The action of a graph morphism $\phi$ is precomposition by $\phi^F$. We will consider the subcategory $\Graph^{\rm dir}$ of the category of elements $\Graph_{dec}(\Oio)$ whose objects are graphs with directed edges and whose morphisms satisfy that only outer flags of opposite orientation are grafted. This induces a full subcategory $\Agg^{dir}$ of directed aggregates with the property that its corollas have in/output flags and the ghost graphs are directed graphs. We further restrict to the subcategory $\Agg^{\rm dir\, forest}$ with the property that each corolla has a single output flag (its root) and the morphisms have forests as ghost graphs.

The projection $\pi:\Graph^{\rm dir}\to\Graph$ restricts to these subcategories and induces a functor $i:\Agg^{\rm dir\,forest}\to\Agg^{\rm forest}$. The decorating functor
$\Ass:\Agg^{\rm dir\, forest}\to\Set$ is then related to the decorating functor $\CycAss:\Agg^{\rm forest}\to\Set$ via pullback, namely $i^*\CycAss=\Ass$.

Indeed, $\CycAss(v_S)$ represents the set of cyclic orders on $S$. If $v_S$ is in the image of the functor $i$, then cyclic orders on $S$ are in one-to-one correspondence with total orders on the set of input flags of $S$ which is $S$ minus the root flag. This set of total orders is precisely the set of automorphisms of $v_S$ when viewed as an object of $\Agg_{dir}^{forest}$, which by definition is $\Ass(v_S)$.

\subsection{History}

The functor $\surf:\Agg^{ctd}\to \Set$ first appeared in the gluing description with the operations $\scirct$ and $\circ_{ss'}$ in \cite{KP} as the open part of a c/o structure given by connected components of the closed/open arc structure. The list in \S\ref{surfsec} corresponds to \cite[\S3, Figure 5]{KP}. In a non-obviously equivalent version it also appears in \cite{CardyLewellen,LaudaPfeiffer,Doubek}. The category $\Graph^{\rm graft}_{\rm dec}(s^*\poly)$ is what is taken as open gluing in \cite{WW} in lieu of the OTFT-gluing induced by the functor $\surf$.

\section{Feynman categories and their operations}
The categories of the last section are monoidal categories of a special type, they are \emph{Feynman categories}. Set-valued monoidal functors like in the preceding section are then their \emph{operations} which often can be identified with operad-like structures. Our formalism permits a uniform treatment of these structures, which is the basis of the further analysis.

\subsection{Basic definitions}

To each category $\V$ we associate the free symmetric  monoidal category  $\V^\otimes$ generated by $\V$. For any functor $\imath:\V\to\Cc$ with symmetric monoidal target category $\Cc$, there exists a unique symmetric monoidal functor $\imath^\otimes:\V^\otimes\to\Cc$. For any category $\F$ we denote by $\Iso(\F)$ the maximal groupoid contained in $\F$, i.e. the objects of $\F$ together with their isomorphisms.

\begin{df}[\cite{feynman}]
\label{commadef}
\label{feynmandef}
Let $\F$ be a symmetric monoidal category and $\imath:\V\hookrightarrow\F$ be the inclusion of a \emph{groupoid}. The triple $\FF=(\F,\V,\imath)$ is called a \emph{Feynman category} if
\begin{enumerate}
\renewcommand{\theenumi}{\roman{enumi}}
\item (Isomorphism condition)\label{objectcond}
The functor $\imath^{\otimes}$ induces an equivalence of symmetric monoidal categories between $\V^{\otimes}$ and $\Iso(\F)$.
\item (Hereditary condition) The functor $\imath^{\otimes}$ induces an equivalence of symmetric monoidal categories between $\Iso(\F\downarrow \V)^{\otimes}$ and
$\Iso(\F\downarrow\F)$.\label{morcond}
\item (Size condition) For each $v\in \asts$, the comma category $\clusters\downarrow v$ is
essentially small.
\end{enumerate}
A \emph{Feynman functor} $\ff:(\F,\V,\imath)\to(\F',\V',\imath')$ is given by a pair of functors $(f:\F\to\F',g:\V\to V')$ such that $f$ is strong symmetric monoidal and $\imath'g=f\imath$. We will usually suppress $g$ from notation and identify notationally $\ff$ and $f$.
\end{df}

Due to conditions (i) and (ii) every morphism in a Feynman category can be written essentially uniquely as a tensor product of morphisms with target in $\imath(\V)$. These morphisms are said to be the \emph{basic morphisms} of the Feynman category $\FF$. For more details on the general theory of Feynman categories we refer the reader to the book \cite{feynman}, a short introduction is contained in \cite{matrix}.

\label{opdef1}For a Feynman category $\FF=(\F,\V,\imath)$ an  \emph{operation} in a symmetric monoidal category $\Cc$ is a strong symmetric monoidal functor $(\F,\otimes_\F)\to(\Cc ,\otimes_\Cc)$. The category of such strong symmetric monoidal functors and symmetric monoidal natural transformations will be denoted $\FF\alg_\Cc$. If $(\Cc,\otimes_\Cc)=(\Set,\times)$ we suppress it from the notation. There is a monoidal structure on $\FF\alg_\Cc$ given by pointwise tensor product. The unit for this monoidal structure is the \emph{trivial operation} $\trivial^\FF$ defined by $\trivial^\F(X)=\trivial^\Cc$  and $\trivial^\F(\phi)=id_{\trivial^\Cc}$ where $\trivial^\Cc$ is the \emph{monoidal unit} of $\Cc$. Whenever $\trivial^\Cc$ is terminal in $\Cc$ (for instance if $\Cc=\Set$), the trivial operation $\trivial^\FF$ is terminal in $\FF\alg_\Cc$. To indicate this we will write $\final$.

\begin{rmk}$\FF$-operations or ops for short, depending on $\FF$, can be operads, algebras, algebras over operads, crossed simplicial objects and so on, see \cite{feynman,feynmanrep}.
\end{rmk}

The categories of the last section are Feynman categories and they are related by Feymnan functors. The first set of examples is related to the category of aggregates, cf. \cite[\S2]{feynman}: let $\Crl$ be the subcategory of corollas together with their isomorphisms and let $\imath:\Crl\to \Agg$ be the inclusion, then $\FFagg=(\Crl,\Agg,\imath)$ is a Feynman category. By restriction, we obtain the Feynman categories $\GGctd=(\Crl,\Agg^{ctd},\imath)$ and $\FFcyc=(\Crl,\Agg^{forest},\imath)$ whose \emph{basic} morphisms have connected graphs, respectively trees as ghost graphs. Decorations (\S\ref{par:decorated}) yield further Feynman categories.
The corresponding operations are operad-like, see \S\ref{par:feyoperads} and Table \ref{table:types}.

The second set of examples are Feynman categories of graphs, which have thus far not been considered. Let $\Ctd$ be the subgroupoid of $\Graph$ spanned by {\em connected graphs} and their isomorphisms and let $\imath$ be the inclusion functor, then $\GGraph=(\Ctd,\Graph,\imath)$ is a Feynman category. Indeed, every graph decomposes into a disjoint union of its connected components and isomorphims respect this decomposition. Furthermore, every morphism $\phi$ decomposes essentially uniquely into a disjoint union of the $\phi_{\overline{v'}}$ according to \eqref{eq:phidecomp}. Finally the slice categories are essentially small.

The functors $s,t,i$ of \S\ref{par:double} extend naturally to Feynman functors and $\GGraph$ and $\FFagg$ form a double Feynman category, that is a Feynman category internal to Feynman categories, using graph insertion as the horizontal morphisms. By restriction we obtain the Feynman categories
$\GGraph^{ctd}=(\Ctd,\Graph^{\rm ctd},\imath)$, $\GGraph^{\rm graft}=(\Ctd,\Graph^{\rm graft},\imath)$ and $\GGraph^{\rm forest}=(\Ctd,\Graph_{\rm forest},\imath)$.

\subsection{Pullback, pushforward and Frobenius reciprocity}One of the main features of Feynman categories is that restriction functors $f^*$ have computable left adjoints $f_!$.

For each Feynman functor $f:(\F,\V,\imath)\to(\F',\V',\imath')$, precomposition with $f$ defines a restriction functor $f^*:\FF'\alg\to\FF\alg$.  Its left adjoint pushforward functor $f_!:\FF\alg\to\FF'\alg$ can be computed like in ordinary category theory as pointwise left Kan extensions, the symmetric monoidal structure being guaranteed by the axioms of a Feynman category, cf. \cite{feynman}.

\begin{thm}[\cite{feynman}]
\label{thm:adjunction}
\label{leftKan}Any Feynman functor $f:\FF\to \FF'$ induces a ``induction-restriction'' Frobenius-reciprocity adjunction $f_!:\FF\alg\leftrightarrows\FF'\alg:f^*$
with left adjoint given by pointwise left Kan extension\begin{equation}
\label{eq:pushforward}
(f_!\O)(X)=\colim_{f(-)\downarrow X} \O(-)
\end{equation}
Indeed these are even  adjoint symmetric monoidal functors.
\end{thm}

There are two possible notations for the push--forwards. We adopt here the categorical notation $f_!$ which is commonly used for left Kan extensions. In \cite{feynman,WardSix} the notation $f_*$ was used instead, in order to avoid confusion with extension by zero.

\begin{rmk}\label{rem:extension}For $\Cc=\Set$ (or more generally if the monoidal unit is terminal in $\Cc$), an extension along $f$ preserves \emph{trivial operations} if and only if the comma categories $(f\downarrow X)$ are \emph{non-empty} and \emph{connected} for all objects $X$ of $\,\FF'$. Feynman functors with this property will be called \emph{connected}.

These extensions are computable if the comma categories $(f\downarrow X)$ are sufficently well understood. In a category with coproducts and coequalisers \emph{any} colimit over a small category is a coequaliser. In the special case $\Cc=\Set$, for a functor $f:I\to \Set$, the colimit can be computed as
\begin{equation}
\label{eq:colim}
(\colim_{I}f)(x)=\left(\bigsqcup_{i\in I}f(i)\right)/\sim
 \end{equation}
where $x\in f(i)$ is identified with  $y \in f(j)$ in the colimit if there is a morphism $\phi:i\to j$ in $I$ such that $f(\phi)(x)=y$.
In other words, the colimit $\colim_{I}f$ may be identified with $\pi_0(I_{dec}(f))$, the set of connected components of the category of elements of $f$, cf. Lemma \ref{connectedcomponents}.
In particular for the Kan extension, $I=(f(-)\downarrow \imath(*))$ has elements $(X,\phi:f(X)\to \imath(*))$ with $X\in \F$ and morphisms induced by $\psi\in \F(X,Y))$. That is $\psi :(X,\phi)\to (Y,\phi\circ f(\psi))$.
\end{rmk}

\begin{ex}
For a Feynman category $\FF=(\F,\V,\imath)$ and functor $\O:\V\to\Set$, let $\FF^\V=(\V,\V^\ot,\imath)$ be the free Feynman category on $\V$ and let $i^\ot:\V^\ot\to \FF$ be the induced inclusion of Feynman categories. Note that $\O$ extends canonically to a operation $\O^\ot$ of $\V^\ot$. Then $(i^\ot)_!(\O^\ot)$ is the free $\FF$-operation generated by $\O$. This construction is left adjoint to the obvious forgetful functor, see \cite[Example 1.6.3]{feynman}.
\end{ex}
\begin{ex}There are Feynman functors $i:\FFoper\to \FFcyc$ and $j:\FFcyc\to \FFmod$. The pushforwards $i_!$ and $j_!$ correspond respectively to the cyclic envelope of a symmetric operad, and to the modular envelope of a cyclic operad. While $i^*$ is the restriction of a cyclic operad to its underlying pseudo--operad and $j^*$ the restriction of a modular operad to its underlying cyclic operad. Note that all operads are not required to be unital.
\end{ex}
\subsection{Decorated Feynman categories}
\label{par:decorated}
The essential ingredient in the constructions at hand is the notion of a \emph{decorated Feynman category}  as introduced in  \cite{decorated}. Decorated Feynman categories are the ``Feynman analogs''  \emph{categories of elements}, cf. \cite{BergerKaufmann} for a parallel treatment of both constructions.
More precisely, we have $$\FFdec(\O)=(\Fdec(\O),\Vdec(\O),\idec(\O))$$
where the functor $\idec(\O):\Vdec(\O)\to\Fdec(\O)$ takes $(v,a)$ to $(\iota(v),a)$.
With slight modifications, the decoration also exists for non--Cartesian $\C$, see \cite{decorated}. If $\O$ is $\Set$ valued,
we call the projection $\FF_{dec}(\O)\to\FF$ a \emph{covering of Feynman categories} following the terminology of \cite{BergerKaufmann}. Among category theorists such coverings are usually called \emph{discrete opfibrations}.

\begin{thm}[\cite{decorated,BergerKaufmann}]
\label{thm:decorated}$\FFdec(\O)$ is indeed a Feynman category.
  Projecting to the first factor is a canonical Feynman functor $\FFdec(\O)\to\FF$. Decorations are functorial with respect to
   Feynman functors   and natural transformations of algebras $\sigma:\O\to \P$, that is the following squares exist and commute
\begin{equation}
\xymatrix{\Fepair \ar[r]^{\ff^{\O}} \ar[d]_{\pi} & \Fe'_{\rm dec}(f_!\O) \ar[d]^{\pi'} \\
\Fe \ar[r]^\ff & \Fe'}
\quad
\xymatrix{\Fepair \ar[d]_{\sigma_{dec}} \ar[r]^{\ff^{\O}} &\Fe'_{dec}(f_!(\O)) \ar[d]^{\sigma'_{dec}}\\
 \Fe_{dec}(\Po) \ar[r]_{\ff^{\Po}}
& \Fe'_{dec}(f_!(\Po)) }
\end{equation}and a diagram of adjoint functors for categories of Feynman operations
 \begin{equation}
 \label{eq:adjunctiondiagram}
\xymatrix{\FF_{dec}(\O)\alg \ar@/^/[r]^{f'_!}\ar@/_/[d]_{\pi_!} & \FF'_{dec}(f_!(\O))\alg  \ar@/^/[l]^{(f')^*}\ar@/^/[d]^{\pi'_!} \\
\FF\alg \ar@/^/[r]^{f_!}\ar@/_/[u]_{\pi^*} & \FF'\alg\ar@/^/[u]^{\pi'^*}\ar@/^/[l]^{f^*} }
\end{equation}
such that the square of left adjoints and the square of right adjoints commute.
\end{thm}


\begin{rmk}\label{lem:covering}
A Feynman functor $f:(\F,\V,\imath)\to(\F',\V',\imath')$ is a covering if and only if the underlying functor $f:\F\to\F'$ is a covering. The characterisation of coverings of categories is well-known: for each $\phi':X'\to Y'$ in $\F'$ and each $X$ in $\F$ such that $f(X)=Y$ there exists one and only one $\phi:X\to Y$ such that $f(\phi)=\phi'$. In particular, $f$ is a full functor. If $f$ is a covering then the Feynman category $\FF$ may be identified with the decorated Feynman category $\FF'_{dec}(f_!\trivial^\FF)$,cf. \cite{BergerKaufmann}.\end{rmk}
\begin{ex}
The functors $s,t:\GGraph^{\rm ctd}\to\GGctd$ satisfy the characteristic property of a  covering, cf. Remark \ref{lem:covering}. The decorating functors are thus given by $s_!(\final)$ and by $t_!(\final)$.

Lemma \ref{connectedcomponents} shows that the second decorating functor $(t_!\final)(\ast_S)$ is the set of isomorphism classes of connected graphs $\G$ such that the total contraction $\G/\G$ is $\ast_S$. This is the set of isomorphism classes of connected graphs with outer flag set $S$. Picking representatives we get $(t_!\final)(\scirct)=\gl{s}{t}$ and $(t_!\final)(\circ_{st})=\gl{s}{t}$. The trivial operation $\final$ has a surface interpretation via a cutting curve system, cf. \S\ref{sct:borderedsurface}. The pushforward by $t$ forgets the cutting curves. The morphisms $\circ_{st}$ and $\scirct$  glue the boundaries while the morphism $\gl{s}{t}$ glues the boundaries {\em and} remembers the boundaries as new cutting curves.
\end{ex}

 \begin{prop}
 \label{prop:decopullback}
 Let $\O$ be a set-valued operation of a Feynman category $\FF'$. Each Feynman functor $f:\FF\to\FF'$ induces a commutative diagram of Feynman functors
 \begin{equation}
 \xymatrix{\FF_{dec}(i^*(\O))\ar[d]\ar[r]^{i'}&\FF'_{dec}(\O)\ar[d]\\
 \FF\ar[r]^i&\FF'
  }
 \end{equation}
 \end{prop}
 \begin{proof}
 Follows from Theorem \ref{thm:decorated} and the natural transformation $i_!i^*\O\to\O$.
 \end{proof}

The next result is the precise analog for Feynman functors of the \emph{comprehensive factorisation} of an ordinary functor (into initial functor followed by discrete opfibration), first established by Street-Walters \cite{SW}. The proof is \emph{mutatis mutandis} the same.

\begin{thm}[\protect{\cite{BergerKaufmann}}]\label{prop1}\label{decopar}\label{genusdecopar}
\label{thm:comprehensive}
Every Feynman functor $f:\FF\to\FF'$ factors essentially uniquely as a connected Feynman functor $\FF\to\FFpdec(f_!(\trivial^\FF))$ followed by a covering $\FFpdec(f_!(\trivial^\FF))\to\FF'$.\end{thm}

\subsection{Operad-like structures as Feynman operations}
\label{par:feyoperads}
The usual operad-like structures can be recovered in the formalism of Feynman categories and their operations.
For reference, we briefly review the main characters here, which are summarized in Table \ref{table:types}; cf.\ \cite[\S2]{feynman} and \cite[\S4] {matrix} for more examples and details.

Cyclic operads have been introduced by Getzler-Kapranov \cite{GKcyclic}. A operation $\O$ of $\Fcyc$ is equivalent to a non--unital cyclic operad \cite[\S 2.3.1]{feynman}. The correspondence in the usual unbiased notation is given by
$\O(\ast_S)=\O((S))$ and $\O(\scirct)=\scirct:\O((S))\ot\O((T))\to \O((S\setminus \{s\}\amalg T\setminus \{t\}))$, cf. \cite{GKcyclic}. This is the action of a virtual edge contraction.  See \cite[\S4]{matrix} for more details.

Non--unital symmetric operads (aka {\em pseudo--operads}, cf. \cite{Markl}) are equivalent to Feynman operations of $\FFoper$ where the Feynman category $\FFoper=(\Crl^{\rm rt},\Agg_{dir}^{forest},\imath)$ has been introduced in \cite[\S2.2.1]{feynman}.
For its groupoid $\Crl^{\rm rt}$ of rooted corollas $v_{S,s}$, see \S\ref{par:dir}.

The correspondence in unbiased notation is given by $\O(*_{S,s})=\O(S\setminus s)$ and in biased notation by $\O(\ast_{\{0,\dots, n\},0})=\O(n)$, cf.\ \cite{MSS,feynman}. Forgetting the distinction of the root flags defines a Feynman functor $i:\FFoper\to \FFcyc:v_{S,s}\mapsto v_S$.

The operations of the Feynman category $\GGctd$ have been introduced in \cite{decorated} under the name unmarked modular operads. They additionally come equipped with operations $\O(\circ_{st})=\circ_{st}:\O(S)\to \O(S\setminus \{s,t\})$ induced by virtual loop contractions. There is a Feynman category inclusion $k:\FFcyc\to\GGctd$.

To obtain the modular operads of Getzler-Kapranov \cite{GKmodular} one has to add genus labeling. The category $\FFmod$ for modular operads has as groupoid genus labelled corollas $v_{S,g}$ and their automorphisms.
The morphisms of $\FFmod$ are those of the subcategory $\Agg^{\rm ctd}$ with the  constraint that $\circ_{s,t}:*_{S,g}\to *_{S\setminus \{s,t\},g+1}$ and $\scirct: *_{S,g} \sqcup *_{T,g'}\to *_{S\setminus \{s\}, T\setminus \{t\},g+g'}$, see \cite{feynman}. The correspondence is via  $\O(\ast_{S,g})=\O((S,g))$ and $\O(\scirct)=\scirct$ and $\O(\circ_{st})=\circ_{st}$ in the standard notation (cf. \cite{MSS}; see \cite[\S4]{matrix} for more details.

Forgetting genus labeling yields a Feynman functor $\pi:\FFmod\to \GGctd$ which is a covering. There is also a Feynman category inclusion $j:\FFcyc\to \FFmod$ taking $\ast_S$ to $\ast_{S,0}$ and $\scirct$ to $\scirct$. This Feynman functor is connected. The composite Feynman functor $\pi\circ j$ is precisely $k:\Fcyc\to\GGctd$. According to Theorem \ref{thm:comprehensive}, the Feynman category $\FFmod$ for modular operads can thus be formally deduced from the Feynman category $\GGctd$ for unmarked modular operads by comprehensive factorisation.
This is a decoration by $\genus$ \cite[\S 6.4.2 ]{decorated} and the forgetful functor forgetting the genus marking $\pi:\FFmod\to \GGctd$ is a covering. These facts can also be derived from Lemma \ref{lem:covering} and Theorem \ref{thm:decorated}.

The genus gives a grading to objects and morphisms additive  under composition and monoidal structure. With
$\deg(*,l)=l, \deg(\phi)=b_1(\ghost(\phi))$. For a basic morphisms this is $1-\chi(\ghost(\phi))=-\bar{\chi}$ where $\bar\chi$ is the
reduced Euler characteristic and for a general morphism $\deg(\phi)=\deg(\bigsqcup_v \phi_v)=-\sum_v \bar\chi(\ghost(\phi_v))$.
Thus any morphism $\phi=\phi_l\circ\phi_\tau\circ\phi_0:X\to Y$  satisfies
\begin{equation}
\label{eq:genusconstraint}
b_1(\ghost(\phi))=|l|=\deg(X)-\deg(Y)
\end{equation}

The version of modular operads considered by Schwarz \cite{Schwarz}, called MOs, amounts to operations of the Feynman category $\FFagg$ by Theorem \ref{thm:aggstructure}. His $\nu_{m.m}$ are induced by mergers and his $\sigma^{(m)}$ by virtual loop contractions. He also considers genus labeling as an additional grading. The category of these genus graded  MOs is equivalent to the category of operations of $\FFncmod$.

\begin{df}
\label{df:fcatdef}
Define $\FFnsoper=\FFoper_{\dec}(\Ass)$, $\FFnscyc=\Fcyc_{\dec}(\CycAss)$ and  $\FFnsmod:=\FFmod_{dec}(\surf)$.
\end{df}
\begin{prop}
\label{prop:decocats} The category of operations of $\FFnsoper$, resp. $\FFnscyc$, resp. $\FFnsmod$ is equivalent to the category of non-symmetric, resp.\ non-$\Sigma$-cyclic, resp.\ non-$\Sigma$-modular operads of Markl \cite{Marklnonsigma}.\end{prop}

\begin{proof}
This is contained in \cite{decorated} and is straightforward from the definitions in the first two cases. In the last case, it follows from compairing the results of \cite{Marklnonsigma} with the reinterpretation of the pair $(p,\sigma_S)$ as giving a partition of $S$ into $p+b$ subsets such that $p$ are empty and $b$ are non-empty, each equipped with a cyclic order.\end{proof}

\begin{rmk}Note that in $\FFnsoper$ the automorphism group of $v_{S,s}$ is trivial, while in $\FFnscyc$ the automorphism group of $v_{S,\sigma_S}$ is cyclic of order the cardinality of $S$, and in $\FFnsmod$ the automorphism group of $v_{g,p,\sigma_S}$ is $(\Z/n_1\Z\times \dots \times \Z/n_b\Z)\wr\SS_b$ whenever $\sigma_S$ has $b$ orbits of length $n_i,i=1\kdk b$.

Therefore, the terminology non-$\Sigma$ may be confusing. We call operations of $\FFnscyc$ planar-cyclic operads and operations of $\FFnsmod$ surface-modular operads. The aforementioned automorphisms groups are important for structures on the coinvariants, such as Gerstenhaber brackets,  Lie brackets and BV structures, see \cite{KWZ}.
\end{rmk}
Via Theorem \ref{thm:decorated}, the natural transformations of Proposition \ref{prop:hex} yield two hexagons of coverings.

\begin{prop}\label{prop:feynmanhexagon}
There is a hexagon of coverings:
\begin{equation}
\label{eq:agghex}
\begin{tikzcd}[column sep=tiny,row sep=tiny]
&\FFagg_{dec}(\eulerpoly)\ar[r]&\FFagg_{dec}(\genus)\ar[dr]&\\
\FFagg_{dec}(\surf)\ar[ur]\ar[dr]&&&\FFagg\\
&\FFagg_{dec}(\polyN)\ar[r]&\FFagg_{dec}(\poly)\ar[ur]&\\
\end{tikzcd}
\end{equation}

\noindent Restriction to $\GGctd$ yields the hexagon of coverings:
\begin{equation}
\begin{tikzcd}[column sep=tiny,row sep=tiny]
&\GGctd_{dec}(\eulerpoly)\ar[r]&\FFmod\ar[dr]&\\
\FFnsmod\ar[ur]\ar[dr]&&&\GGctd\\
&\GGctd_{dec}(\polyN)\ar[r]&\GGctd_{dec}(\poly)\ar[ur]&\\
\end{tikzcd}
\end{equation}
The types of ghost graphs are given in Table \ref{graphdecotable}.
\end{prop}

\subsection{History}
The interpretation of the genus labelling $g$ as $1-\chi$ occurs in the open modular part of the c/o
structure in \cite[Appendix A.3]{KP}. The use of the  grading $1-\chi$ in the presence of mergers is in \cite[VII A 3]{KWZ}. An extension of the operations of ${\poly}$ is given in the form of brane-labeling in \cite[Appendix A.6]{KLP}. Restricting to a single brane label defines ${\poly}$.

Implicitly $\FFnsmod$ and its operations occur as specialisations of algebras over the c/o structure $\pi_0(\widetilde{\Arc}(S,T))$, cf. \cite[\S5]{KP}. The morphism $\surf\to \eulerpoly$ is used in \cite{hoch2} to define the correlation functions. This is explicit in the formula (4.3) of \cite{hoch2}. The necessity to work with the full decorating functor $\surf$, i.e.\ the full indexing by topological surface types appears when the Hochschild cochain complex is viewed as an algebra in the open/closed case \cite{ochoch}. The fact that the book keeping must include the internal punctures in the open case is explicitly stated there. Furthermore, the generalisation to the associative case given in \cite{hochnote} shows that $\surf$ and $\FFnsmod$ are needed to provide compatible correlation functions.

The relation to modular operads was outlined and clarified in \cite{Marklnonsigma}, especially the role of empty cycles, viz.\  punctures or unmarked boundaries. The description of stable ribbon graphs using ${\poly}$ is in \cite{Barrannikov,postnikov}.
The necessity of internal punctures in the open/closed case was discussed in \cite{ochoch}.  Although the correlation functions exist without punctures \cite{TZ}, the gluing introduces them in the open case.

On the chain level, even in the closed case, punctures appear due to the differential. It is possible to factor these contributions out using a filtration \cite{hoch1} or a stabilisation \cite{postnikov}. For actions on Hochschild complexes, the Euler class has to be the unit for the stabilisation to act. In this case, one obtains an $E_{\infty}$-structure on the Hochschild cochain complex \cite{ROMP}.
The suppression of punctures works on the chain level by setting the respective components to zero, which has been exploited by \cite{Barrannikov}. This can now also be understood via a right Kan extension.

\section{Computing pushforwards}
\subsection{Main diagram}\label{pushresthm}
Consider the morphisms $i:\FFoper\to \FFcyc$, and $k:\FFcyc\to \GGctd$.
Then from Theorems \ref{thm:decorated} and \ref{thm:comprehensive} and Proposition \ref{prop:decopullback} have the following commutative diagram:

 \begin{equation}
\label{eq:decodiag}
\xymatrix{
\FFoper_{dec}(i^*\CycAss)\ar[r]^{i'}\ar[d]&\FFcyc_{dec}(\CycAss) \ar[d] \ar[r]^{j'} \ar[d] & \FFmod_{dec}(k_!\CycAss)\ar[d] \\
\FFoper\ar[r]_i&\FFcyc\ar[r]_j\ar[dr]_k &\GGctd_{dec}(k_!\final)\ar[d]^\pi\\
&&\GGctd\\
}
\end{equation}  the vertical functors are coverings and $k=fj$ is the comprehensive factorization into a connected morphism and a covering and $j'$ are connected.
The only input $\oper$ for the construction besides $\final$ is the cyclic operad  $\CycAss\in\FFcyc\alg$, that is cyclic orders, the existence of the other categories, functors and operad types is now a consequence of push--forward and decorations.

In view of Definition \ref{df:fcatdef}
to finish the proof of Theorem B of the introduction it remains to establish the following identifications of Feynman operations:
\begin{equation}
i^*(\Ass)\cong \CycAss,\quad  k_!(\final^{\FFcyc})\cong \genus, \quad  k_!(\CycAss)\cong \surf
\end{equation}

The first identification has been described in \S\ref{par:dir}. The other two identifications will be established in Propositions \ref{prop:ctdenvelope} and \ref{prop:modularenvelope} respectively.
The following lemma will be most useful.

\begin{lem}\label{connectedcomponents}The colimit of a functor $F:I\to\Set$ on a small category $I$ can be identified with the set of connected components $\pi_0(I_{dec}(F))$ of the category of elements $I_{dec}(F)$ of $F$, cf. \S\ref{par:decofun}.\end{lem}

\begin{proof}This follows from combining the following three facts:
(1) the projection $\pi:I_{dec}(F)\to I$ is a covering, (2) $\pi_!(\trivial^{I_{dec}(F)})=F$ and (3) left Kan extensions compose. Alternatively, the set $\pi_0(I_{dec}(F))$ can directly be identified with the canonical coequaliser presentation of the colimit $\colim_IF$, cf. formula \eqref{eq:colim} in Remark \ref{rem:extension}.\end{proof}

In order to compute left Kan extensions it suffices thus to determine the connected components of certain well defined categories. In our case, these categories will be categories of structured graphs, and one of the main issues consists of describing them explicitly. The way to proceed is to identify complete graphical invariants of the connected components above.

\subsection{Genus labeling $\genus$ as pushforward}
\label{par:compu}

\begin{lem}
\label{lem:slice}
The the slice category $(\FFagg \downarrow *_S)$ is equivalent
to the subcategory of $\Graph^{\rm contr}$  whose objects are connected graphs whose set of tails is $S$ and morphisms given by contracting spanning sub--graphs and isomorphisms.
\end{lem}

Note that the slice categories $(\FFagg \downarrow \ast_S)$ are also the essential fibres of $t:\Graph^{\rm ctd}\to \Agg^{\rm ctd}$.

\begin{proof}
Consider with $\FFagg$ and $X=\ast_S$. Up to isomotphism, we can restrict to pure morphisms $\phi_p: X\to *_S$ up to isomorphism. In this situation Proposition \ref{prop:uniquefiller} applies with $\phi_R=id$ and provides the identification via the ghost graphs. Since there are no mergers, the respective morphisms on the ghost graphs are  contractions.
 \end{proof}

For later computations, we need the following precise version of Corollary \ref{cor:aggcrossed} for contractions with one-vertex target.

\begin{lem}
\label{lem:technical}
Any morphism $\phi:X\to\ast_S$ in $\GGctd$ decomposes as $\phi=\sigma_v\phi_{p} \sigma$ where $\phi_{p}:X'\to\bar v_{S}$ is the total \emph{pure contraction}, and $\sigma:X\to X'$ and $\sigma_v:\bar v_S\to\ast_S$ are isomorphisms with $\sigma_v$ fixing the outer flag set $S$.

A canonical choice is provided by the unique decomposition $\phi=\sigma\circ\phi_{p}$ of Theorem \ref{thm:aggstructure}, which can be rewritten as $\phi=\sigma_v\phi_{p}\hat \sigma$ where $\hat\sigma$ fixes all vertices and extends $\sigma^F$ by the identity on inner flags.

Given two such decompositions $\sigma_v\phi_{p}\sigma$ and $\sigma_{v'}\phi'_{p}\sigma'$ of $\phi$ the pair ($\sigma'\sigma^{-1},\sigma_{v'}^{-1}\sigma_v$) defines an isomorphism from $\phi_p$ to $\phi'_p$ in the arrow category.
\end{lem}
This implies that  $\sigma_2\sigma^{-1}_1$ as a morphism $\ghost(\phi_p)\to \ghost(\phi'_p)$ is an isomorphism. In particular fixing $\phi_p$ in the decomposition
$\sigma$ is fixed up to a graph automorphism fixing the set of outer flags.
\begin{proof}By Theorem \ref{thm:aggstructure} there is a unique decomposition $\phi=\sigma\phi_p$ where $\phi_p$ is pure and $\sigma$ is an isomorphism. Besides the identification of the single vertex of $\G(\phi)/\G(\phi)$ with the single vertex of $\ast_S$, $\sigma$ induces a bijection between $S$ and the outer flags of $\Gamma(\phi)$. We define $X'$ to be the aggregate obtained by replacing the  flags of $X$ corresponding  to $S$ with the flags  $S$. It is then obvious that there is an isomorphism $\hat\sigma:X\to X'$ giving rise to the asserted decomposition. The last claim is immediate.\end{proof}

To make reduce the categories one can use the following standard labeling of vertices of \emph{reduced graphs}, i.e. admitting at most one vertex without leaves. A reduced graph has a standard vertex set if $v=\del^{-1}(v)$. We will now denote the standard corolla by $\ast_S$. Every reduced graph is isomorphic to a graph with standard vertex labeling by an isomorphism with $\phi^F=id$, that is $\sigma^n:\G\to \G^n$ where $(\sigma^n)^F=id$ and $\sigma^n_V(v)=F_v$ is the isomorphism that assigns to each vertex its standard name.

A standardized pure morphism is defined to be $\phi^n_p=\sigma^n\phi_p$, we will assume that $t(\phi_p)$ has standard vertex labeling. This yields a unique standard decomposition of a morphisms as $\phi=\sigma\phi_p^n$ via Theorem \ref{thm:aggstructure} $\phi=\sigma\phi_p=\sigma(\sigma^n)^{-1}\phi_p^n$.

If the graph is not reduced, there may be several one vertex components without flags and these vertices would need different names. This can be achieved by introducing a skeletal labeling in terms of a number.

Let $k:\FFcyc\to\GGctd$ denote the canonical inclusion. We will now describe the comma categories $k\downarrow *_S$. An object is a pair $(X,\phi_X)$ consisting of an aggregate $X$ in $\Fcyc$ and a morphism $\phi_X:k(X)\to \ast_S$ in $\GGctd$. A morphism in the comma category is given by a morphism of aggregates $\psi:X\to Y$ in $\Fcyc$
rendering commutative the following triangle:

\begin{equation}\label{triangle2}
\xymatrix{
X\ar[d]_{\phi}&k(X)\ar[r]^{\phi_X}\ar[d]_{k(\phi)}&\ast_{S}\\
Y&k(Y)\ar[ur]_{\phi_Y}&
}
\end{equation}
\begin{lem}
\label{lem:comma}
The comma category $(k \downarrow\ast_S)$ is equivalent to the subcategory whose objects are standard contractions $\vc_{\G}^n:X^n\to \ast_S$ and whose morphisms are generated by standard spanning forest  contractions $\vc^n_f:\vc^n_\G\to \vc_{\G/f}$ and isomorphisms $\sigma:\vc_{\G}\to \vc_{\G'}=\sigma\vc_{\G}$ given by isomorphisms of the underlying source aggregate fixing the image $\phi^F(S)$.
\end{lem}

\begin{proof}We show that the inclusion is an equivalence. First, we show the essential surjectivity. Given $\phi:X\to *_s$, we can can decompose $\phi=\sigma \vc^n_{\ghost(\phi)}$. Furthermore by Lemma \ref{lem:technical} this is equal to $\vc^n_{\ghost(\phi)}\sigma'$ which is isomorphic in the comma category to  $\vc^n_{\ghost(\psi)}$ by precomposition with $\sigma'$. The functor is clearly faithful. To show that it is full consider any $\psi:\vc^n_{\G}\to \vc^n_{\G'}$, then $\psi=\sigma \psi_p^n$ with $\psi_p^n=\vc_f$ a forest contraction. Considering the diagram
\begin{equation}
\xymatrix{
Y^n\ar[r]^{\vc_{\G'}}&\ast_S\\
X^n\ar[u]^{\vc_f}\ar[ur]^{\vc_\G}
}
\end{equation}
We can identify $f\subset \G$ as a spanning forest since $\ghost(\vc_{\G'}\circ \vc_f)=\ghost(\vc_{\G'})\circ \ghost(\vc_f) =\ghost(\vc_\G)$ and thus
$E_{\G}=\vc_f^F(E_{\G'})\amalg E_{f}$.\end{proof}

We consider the following category $\Igusa_S$, where $\Igusa$ stands for an Igusa type category. This is the subcategory of $\Graph^{\rm ctd}$
whose objects are standard connected graphs $\G$,  with outer flag set $S$ and standard vertex set, and  whose morphisms are generated by
standard subforest contractions and isomorphisms leaving the outer flag set fixed.

\begin{prop}
\label{prop:Igusa}The comma category $(k \downarrow\ast_S)$ is equivalent to  $\Igusa_S$.
\end{prop}

\begin{proof}
 The equivalence is given by combining Lemmas \ref{lem:slice} and \ref{lem:comma}.
\end{proof}

\begin{prop}\label{prop:ctdenvelope}The pushforward of the trivial operation $\trivial^{\FFcyc}$ along $k:\FFcyc\to\GGctd$ yields genus labeling $\genus$.\end{prop}

\begin{proof}In virtue of Theorem \ref{leftKan} and Lemma \ref{connectedcomponents}, the pushforward of the trivial operation is given for any aggregate $X$ by the set of connected components of the comma category $(k\downarrow X)$. Since this is a Feynman operation of $\GGctd$ (i.e.\ a monoidal functor), it suffices to compute the connected components of $(k\downarrow\ast_S)$. By Proposition \ref{prop:Igusa} these correspond to those of $\Igusa_S$. Thus, we  have  to compute the connected components of the category of connected graphs and standard subforest contractions fixing the outer flag set.

The loop number $b_1$ of a connected graph  remains unchanged under subforest contraction. Moreover, any connected graph lies in the same connected component as the one-vertex graph obtained by contracting a spanning subtree. Since any two one-vertex graphs with same loop number are isomorphic, the loop number is a complete invariant, and we can identify $k_!(\trivial^{\Fcyc})(\ast_S)=\pi_0(k\downarrow \ast_S)$ with the set $\N_0=\genus(*_S)$.

In order to identify this $\GGctd$-operation with $\genus$ we have to compare the actions of the morphisms of $\GGctd$. On the comma categories, the morphisms act by postcomposition. It suffices to check on the generators. Now $b_1(\ghost{\scirct\phi})=b_1(\ghost_\phi)$ since a non-loop gluing do not change the loop number, and $b_1(\ghost(\circ_{st}\phi)=b_1(\ghost(\phi))+1$ as the loop number is increased, just as under the action of $\genus$.
\end{proof}

\subsection{Surface type labeling $\surf$ as pushforward}We begin by describing some inherent structural difficulties of the comma categories $(k\downarrow\ast_S)$.

Two spanning trees of the same connected graph $\G$ are called {\em adjacent} if they share all but one edge.
Passing from one spanning tree to an adjacent one is called a {\em mutation}.
 The spanning trees of a connected graph $\G$ form again a graph $\mathcal{T}(\G)$, the so-called {\em spanning tree graph} of $\G$. The set of vertices of $\mathcal{T}(\G)$ is the set of spanning trees of $\G$ with an edge between any two adjacent spanning trees. The following theorem refines  the connectivity result of Proposition \ref{prop:ctdenvelope}.

\begin{thm}[\cite{Cummins}]\label{treethm}For any connected graph $\G$, the spanning tree graph $\mathcal{T}(G)$ is connected.
\end{thm}

\begin{rmk}
There is even a Hamiltonian cycle, i.e.\ a cycle passing through all edges. Such a Hamiltonian cycle can be determined algorithmically \cite{Kamae}. Observe that the edge which has been removed from the first and the edge which has been added to the second of two adjacent spanning trees belong to a common Hamiltonian cycle of $\mathcal{T}(G)$.

\end{rmk}

We have seen in the proof of Proposition \ref{prop:ctdenvelope} that any object of $(k\downarrow\ast_S)$ maps to an object with a one-vertex ghost graph. Since  objects of $(k\downarrow\ast_S)$ with one-vertex ghost graph have non-trivial automorphism groups, they are \emph{not} terminal, and there are parallel morphisms into any such object.
Such a parallel pair $X\rightrightarrows Y$ is related by an \emph{elementary mutation} if there exists a diagram $X\to Y'\rightrightarrows Y$ composing to the given parallel pair such that the ghost-graph of $Y'$ has \emph{two} vertices and the parallel pair $Y'\rightrightarrows Y$ represents contraction to each of the two vertices. This corresponds to a mutation of the underlying spanning trees.

\begin{prop}In the comma category $(k\downarrow\ast_S)$ parallel morphisms into objects with one-vertex ghost graph are connected by a finite sequence of elementary mutations.\end{prop}

\begin{proof}Any two parallel morphisms  correspond to two spanning trees of the corresponding graph in $\Igusa_S$. By Theorem \ref{treethm} these two spanning trees are related by a finite sequence of mutations. It suffices thus to show that any mutation of spanning trees factors through an elementary mutation.

Consider a standard morphisms $\vc_\Gamma$ and choose two different spanning trees $\tau$ and $\tau'$. Then then there are two decompositions $\vc_l\circ \vc_\tau=\phi_\Gamma=\vc'_l\circ \vc_{\tau'}$. Where $b_1(\ghost(\vc_l))=b_1(\ghost(\vc_2))=b_1(\G)$ Furthermore, there is an isomorphism $\sigma$ given by any $\sigma_{T',T}$ that preserves the incidence conditions with
$\phi'_l=\phi_l\circ\sigma$. That is, there is a diagram

\begin{equation}
\label{paralleleq}
\xymatrix{
&\ast_{S}\ar@{}[d]|{/\!/}&\\
\ast_{S\amalg T}\ar[ur]^{\vc_l}&&\ast_{S\amalg T'}\ar[ul]_{\vc'_l}\ar[ll]_{\sigma}^\simeq\\
&X\ar[ul]^{\vc_\tau}\ar[ur]_{\vc_{\tau'}}&
}
\quad
\raisebox{-.5cm}{
\xymatrix{
\ast_{S\amalg T}\ar[r]^{\vc_l}&\ast_{S}\\
X\ar@/^1pc/[u]^{\vc_{e_1}}\ar@/_1pc/[u]_{\sigma\circ \vc_{e_2}}\ar@(r,d)[ur]_{\vc_{\G}}&
}
}
\end{equation}
where the upper triangle commutes, but the lower {\em does not} in general.
The choice of $\sigma$ and hence the diagram is unique up to unique automorphism of $\vc_l$.
Recall that $T$ and $T'$ are the sets of flags that are not in the spanning tree, and these are different (not only by name).
 This yields the two parallel morphisms $\vc_{e_1}$ and $\sigma\vc_{e_2}$ on the right.
 The mutation is depicted in  Figure \ref{fig:mutation}, where    $e_i=\{t_i,\bar t_i\}$, $Y=*_{S_1\amalg T_1\amalg \{t_1,t_2\} } \amalg *_{S_1\amalg T_1\amalg \{\bar t_1,\bar t_2 \}}$ with $S_1\amalg S_2=S$ and $T_1\amalg T_2\amalg e_2=T$
and $T_1\amalg T_2\amalg e_2=T'$.

\begin{figure}
\includegraphics[height=8cm]{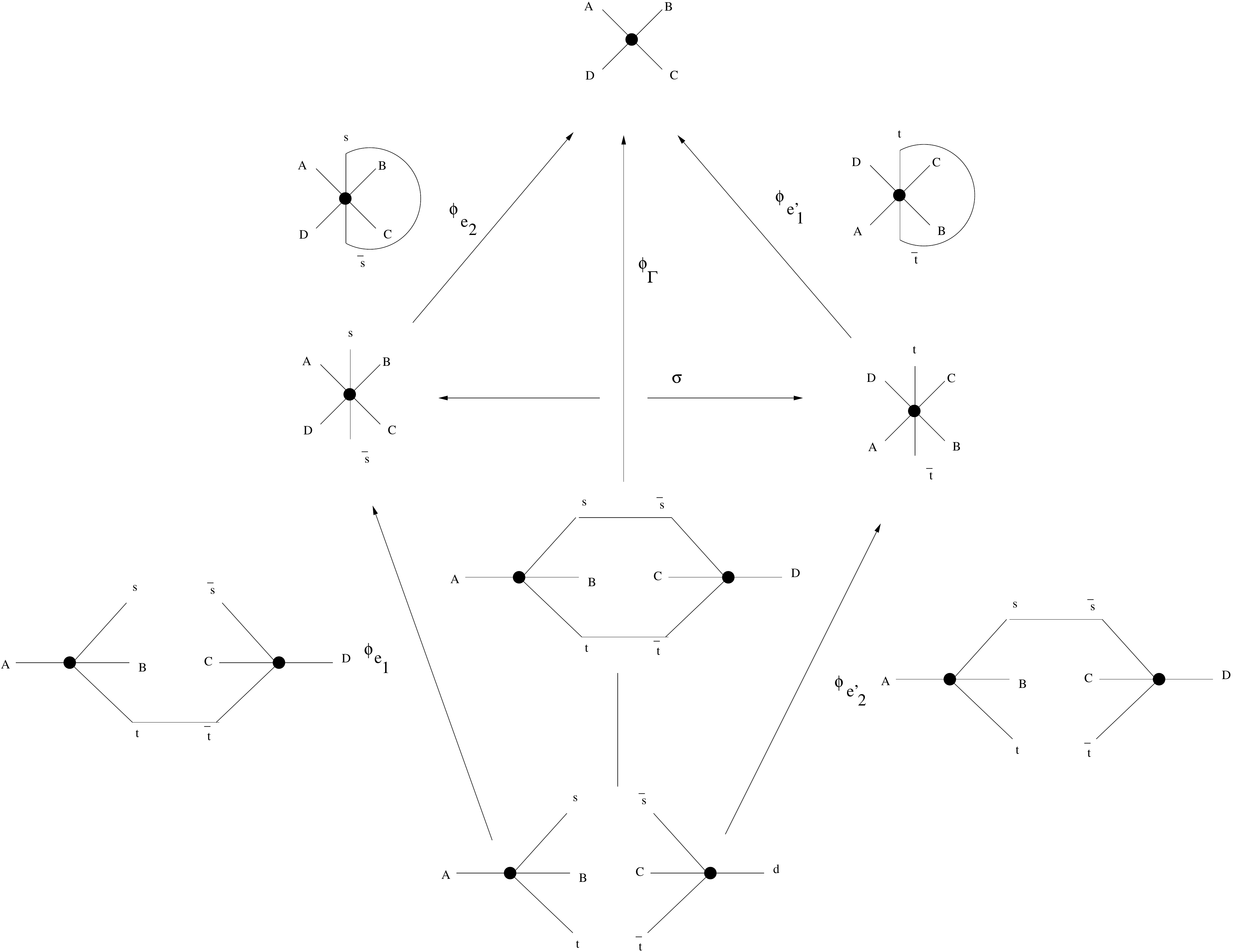}
\caption{\label{fig:mutation} An elementary mutation: a parallel pair induced by different contractions}
\end{figure}

Suppose $\tau_1$ and $\tau_2$ are adjacent in the spanning tree graph, let $\tau$ be their
common subtree and let $\tau_i$ have an additional edge $e_i$, then the we can factor
$\vc_{\tau_i}=\vc_{e_i}\circ\vc_\tau$, where $\vc_{e_i}$ is a
simple edge contraction.
Then two parallel morphisms $\vc_{\tau_1}$ and $\sigma \vc_{\tau_2}$ factor through an elementary
mutation:
\begin{equation}
\label{correq}
\xymatrix{
\ast_{R_1}\ar[r]^{\vc_1}&\ast_{S}\\
X\ar@/^1pc/[u]^{\vc_{e_1}}\ar@/_1pc/[u]_{\sigma\vc_{e_2}}&\\
Y\ar[u]^{\vc_{\tau}}\ar@(r,d)[uur]^{\vc_{\G}}\\
}
\end{equation}
\end{proof}

\begin{prop}
\label{prop:RibbonIgusa}
 The comma category $(k \pi_2\downarrow \ast_{S})$ is equivalent to $\RIgusa_{S} \subset \Rib^{\rm for}$. The objects of $\RIgusa_{S}$ are standard connected ribbon graphs with outer flag set $S$. Morphisms are standard subforest contractions fixing $S$.
\end{prop}
\begin{proof}
This  follows from Proposition \ref{prop:Igusa}, Proposition \ref{prop:ctdenvelope}, via Proposition \ref{prop:feynmanhexagon} and Theorem \ref{thm:decorated} and the fact that a $\CycAss$-decoration is a ribbon structure.
\end{proof}
\begin{prop}\label{prop:modularenvelope}$k_!(\CycAss)=\surf$.\end{prop}

\begin{proof}Combining Proposiion \ref{prop:RibbonIgusa} and Lemma \ref{connectedcomponents} implies that $k_!(\CycAss)(\ast_{S})=(k\pi_2)_!(\final^{\FFnscyc})(\ast_{S})$ can be identified with $\pi_0(\RIgusa_{S})$. It thus remains to be shown that the nc--modular operad $\surf$ of oriented surface types may be identified with the connected components of the categories $\RIgusa_{S}$.

Each connected ribbon graph contracts to a one-vertex ribbon graph by contraction of a spanning tree. We will show that each one-vertex ribbon graph is equivalent to a one-vertex ribbon graph in \emph{normal form}, and that each connected component of $\RIgusa_{S}$ contains a single one-vertex ribbon graph in normal form. Then we describe a one-to-one correspondence between one-vertex ribbon graphs in normal form and connected topological types respecting the modular operad structures.

One extra-information of our proof is the fact that two one-vertex ribbon graphs are in the same connected component if and only if they are \emph{mutation-equivalent}, i.e. transformable into each other by a finite sequence of \emph{elementary mutations} where an elementary mutation between \emph{one-vertex} ribbon graphs is defined to be a \emph{two-vertex} ribbon graph which contracts to both of them.

We represent one-vertex ribbon graphs as \emph{cyclic words} of their flags where two inner flags making up a loop are denoted $t,\bar{t}$ and outer flags get capital letters. A one-vertex ribbon graph is \emph{in normal form} if the representing cyclic word is of the form

\begin{equation}
\label{eq:normalform}
(S_1 l_1 S_2 \bar l_1 l_2  S_3 \bar l_2 \dots  l_{b-1} S_{b} \bar l_{b-1} e_1\bar e_1\cdots e_p \bar e_p a_1b_1\bar a_1\bar b_1 \dots a_gb_g\bar a_g\bar b_g)
\end{equation}

\noindent where $S_i$ denotes a cyclic flag set of cardinality $p_i$ and $0<p_1\leq\cdots\leq p_b$ so that $S=(S_1,\dots,S_b)$ is a polycyclic set $(S,\sigma_S)$ with $b$ cycles. It follows from Lemma \ref{mutalem} and Corollary \ref{mutacor} below that every cycle word is mutation-equivalent to a cyclic word in normal form.
The normal form is determined by and determines the triple $(g,p,\sigma_S)$.

 To understand the functor  $k_!(\CycAss)$ on morphisms, we only have to consider post--composition with the generators.
For isormorphisms post--composition is the usual action by isomorphisms. A virtual loop contraction $\circ_{st}$ adds a loop, by renaming $t$ to $\bar s$ to the respective
one--vertex ribbon graph. The polycyclic structure on the $S_i$'s is the one given by $\circ_{st}$. If the new pair is adjacent in the normal form, then it produces an empty partition, that is $p$ increases by $1$. This is exactly \eqref{eq:surfloop}.

For the operation $\scirct$, two standard words are concatenated and $t$ is renamed $\bar s$ providing a new pair. The genus $g$ and the number of pairs $e_i\bar e_i$
is additive. The effect on the polycyclic structure of the $S_i$ is  $\scirct$.
The relative position of this corresponds exactly to the cases in
\eqref{eq:surfedge}.  If  they are from different sets $S_i, S_j$, where we can assume that $S_i=S_1$, they introduce a new interleaved  pair which increases the genus and if   they are additionally both the only element in their set, then   $p$ also increases by one. If $s,t$ are both in the same $S_i$, we can assume that this is $S_1$.
If $S_2=\{s,t\}$ then this pair is empty and $p$ increases by $2$. If $s,t$ are adjacent, but are not the only two elements of $S_i$ $p$ only increases by $1$. If they are not adjacent, then then $p$ stays constant.
\end{proof}

\begin{lem}
\label{mutalem}
There is an elementary mutation to the effect $(AtBC\bar tD)\leftrightarrow (DtC B\bar t A)$
\end{lem}
Thus   we may cyclically permute  the letters between an occurrence of $t$ and $\bar t$ as well inside (between $t$ and $\bar t$) as well as outside (between $\bar t$ and $t$).
\begin{proof}
Given a cyclic word $(AtB\bar{t}C)$ we split the unique vertex of $\G$ into two vertices joined by two parallel edges one being $(t\bar t)$, the other $(s \bar s)$, in such a way that contraction of $(s \bar s)$ yields $\G$. This implies that the outer flags $B$ and $C$ sit inside the circle defined by $(\bar t t s \bar s)$. Contracting the edge $(t \bar t)$  then produces  a one-vertex ribbon graph $\G'$ represented by the cyclic word $(AsCB\bar sD)$. Up to renaming $s$ by $t$ this yields the desired result, see Figure \ref{fig:mutation}.
\end{proof}
\begin{cor}\label{mutacor}
\mbox{}
\begin{enumerate}
\renewcommand{\theenumi}{\roman{enumi}}
    \item  There is a mutation to the effect $(AtB\bar t C)\leftrightarrow(ACtB\bar t)$.
    \item  There is a mutation to the effect $(sAtB\bar t C\bar s)\leftrightarrow (sAC\bar stB\bar t)$.
    \item  There is a mutation to the effect $(AsBtC\bar s D\bar t E)\leftrightarrow (ADCBEst\bar s\bar t)$.

\end{enumerate}
In particular, every cyclic word is mutation-equivalent to one in normal form.
\end{cor}
\begin{proof}
For (i) this is the special case of Lemma \ref{mutalem} where $C=\emptyset$.
For (ii) we may apply (i) to  move the letters $C\bar s$ across the loop $tB\bar t$.
For (iii) we need a sequence of mutations of the previous types $(\dots A s B t C \bar s D \bar t E \cdots )\leadsto  (\dots  A D s E t B \bar s  C \bar t \cdots)\leadsto
(\cdots ADC s t E \bar s B \bar t \cdots)$
$ \leadsto (ADCBSt\bar s E \bar t \cdots) \leadsto (\cdots ADCBEst\bar s\bar t \cdots)$
where in the first mutation we moved  $D$  left over the $s,\bar s$ loop, $E$ left over the $t, \bar t$ loop and moved $B$ and $C$ to the right inside the loops
$s,\bar s$  and  $t, \bar t$. The next step iterates this process until everything is moved out to the left.

Now using (ii) we unnest, using (iii) we isolate interleaved pairs, and in a final step, we move all the remaining letters that are not in between $t$ and $\bar t$ to the left using (i).
\end{proof}

\subsection{Combinatorial realizations of cyclic words}\label{types}

There are other combinatorial presentations of one-vertex ribbon graphs which can be used for an alternative proof of the existence and uniqueness of normal forms, see  Figures \ref{fig:reps},\ref{fig:rainbow} and \ref{fig:comp}. For Lemma \ref{mutalem} in the respective formalism, see Figures \ref{fig:mutapoly}, \ref{fig:mutachord}  and \ref{fig:mutarainbow}.

\begin{figure}
\includegraphics[width=.4\textwidth]{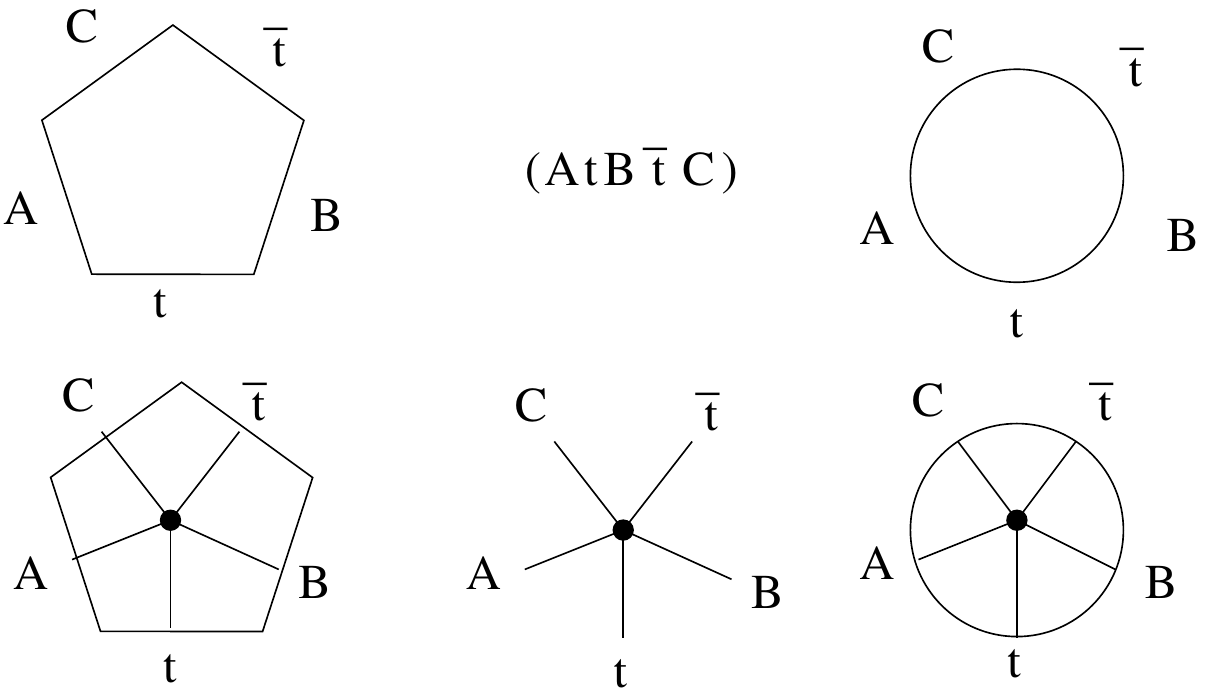}
\caption{\label{fig:reps}
The representation of the cyclic word as a polygon, a word written on a circle and a planar corolla. The corolla can be superimposed into the polygon or the circle, which then play dual roles as either the sided or the vertices are labelled.}
\end{figure}

\subsubsection{Labelled polygons and oriented surfaces}

The flags of a one-vertex ribbon graph correspond one-to-one to the sides of a polygon, the loops are realized by self-gluings. The resulting bordered oriented surface has the same homotopy type as the one constructed in section \S \ref{par:surfinterpret}. We refer the reader to \cite{munkres} where this kind of structure is been used for a complete classification of bordered oriented surfaces following \cite{Massey}. The nc--modular operad structure is visible on this level.

One can blow up the vertices of the polygons to intervals and thereby obtain $2n$-gons with alternating sides that are labelled. In this way a  triangle turns into a planar pair of pants. This  point of view is common for open TFT \cite{CardyLewellen,LaudaPfeiffer}. It also corresponds to looking at $\pi_0$ in the arc picture \cite{KLP,KP} and basically goes back to triangulations of surfaces with boundary and hyperbolic geometry \cite{TravauxdeThurston}. It has later been used under the name of cogwheels or tabs \cite{ChuangLazarev,Marklnonsigma}.

The composition $\scirct$ of planar corollas is called mating spiders in \cite{ConantVogtmann}.

 \begin{figure}
    \centering
    \includegraphics[width=.3\textwidth]{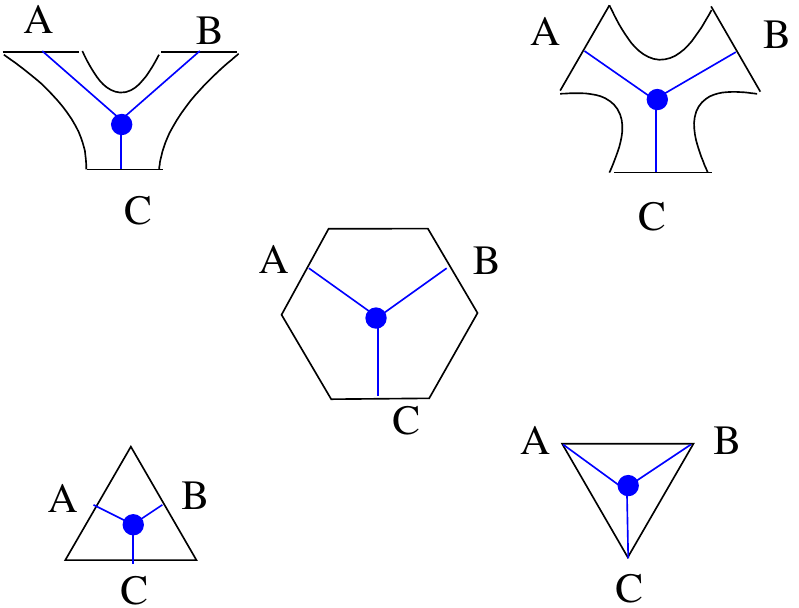}
    \caption{The result of ``thickening'' a 3--valent vertex with a cyclic order. The result resmbles a cogwheel (upper right) or a square pair of pants (upper left)  which is topologically equivalent to a hexagon (in general an $n$--vertex will yield a 2n--gon)   (middle) whose sides labelled by the flags $A,B,C$ in an alternating fashion. The lower line depicts two triangles which are the result of contracting one set of alternating edges. The boundary markings are extra markings.}
    \label{squarepants}
\end{figure}

\subsubsection{Chord/rainbow diagrams}

The flags of a one-vertex ribbon graph are represented by points on a circle, the loops are realized by segments between the two points representing the internal flags of the loop. One obtains in this way a chord diagram. Cutting the circle at one point, the chord diagram becomes a rainbow diagram. The composition now is given by cutting open the chord diagram at the marked vertices and connecting the outer circles according to the orientation.

The gluing in terms of chord diagrams is related to  Kontsevich's coproduct on chord diagrams \cite{Bar-Natan}. More precisely, if one considers the Feynman category of one vertex ribbon graphs in $\Rib^{\rm for,con}$, the coproduct dual to the composition \cite{HopfPart2} is indeed the Kontsevich coproduct.

\begin{figure}
\includegraphics[scale=.6]{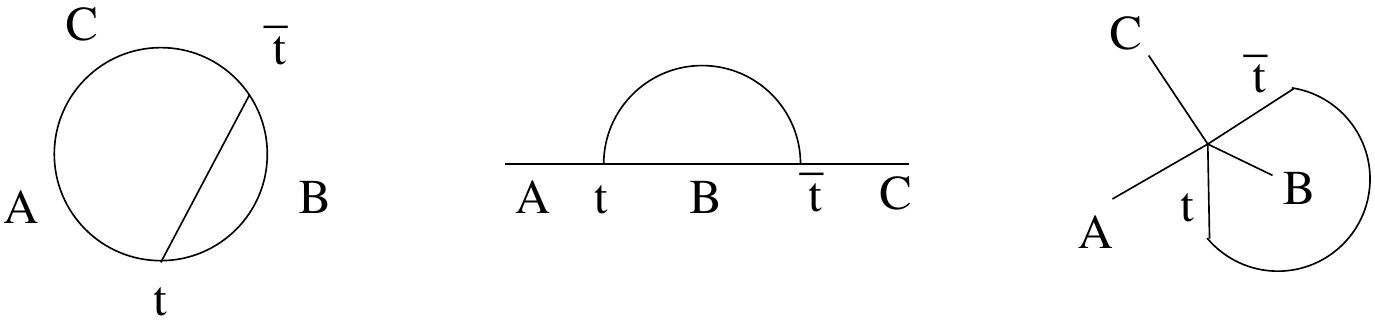}
\caption{\label{fig:rainbow}A chord diagram, a rainbow diagram and a one-vertex ribbon graph representing the same cyclic word}
\end{figure}

\begin{figure}
\includegraphics[width=.6\textwidth]{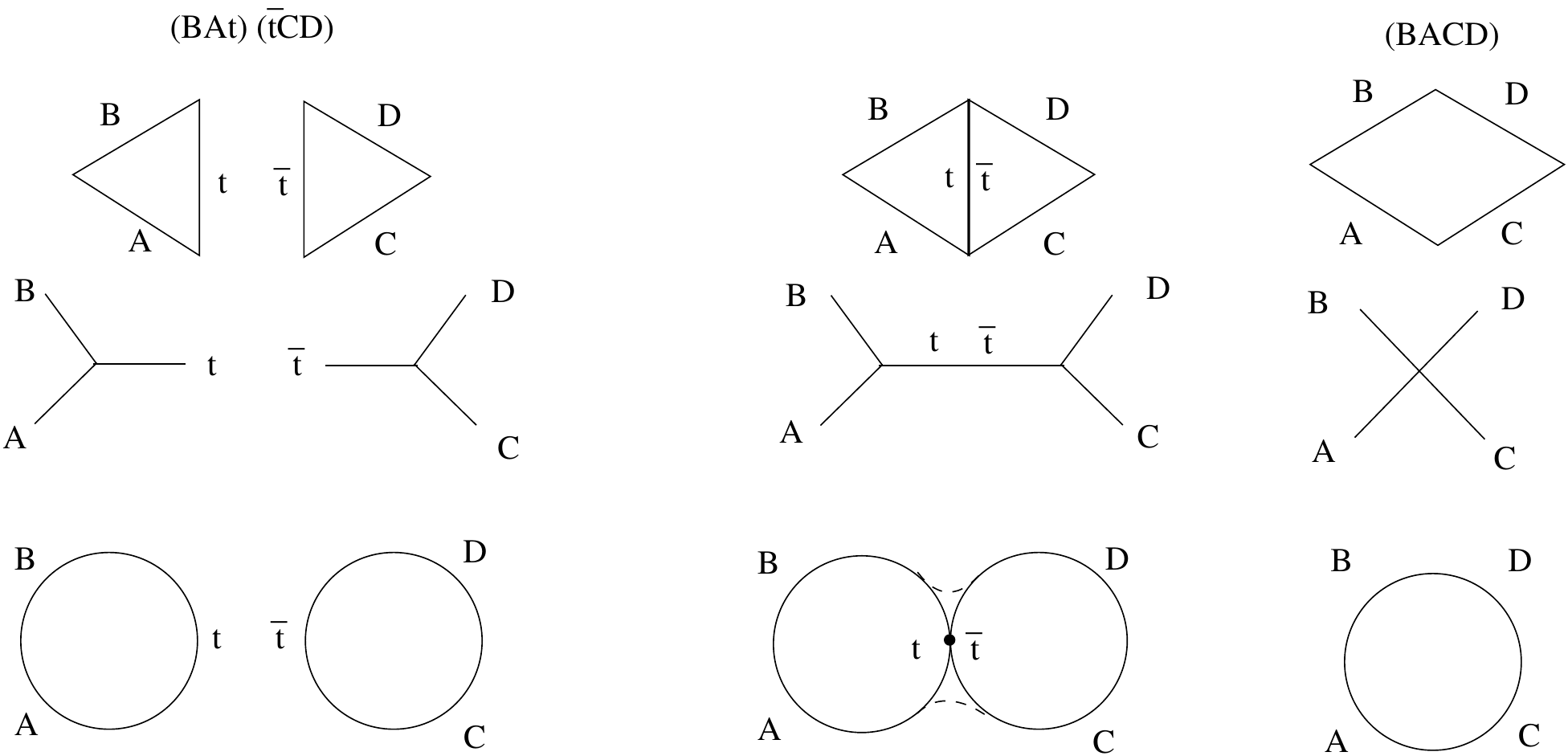}
\caption{\label{fig:comp}Composition of cyclic words as gluing of polygons along sides,
gluing and contracting edges of corollas, or on circles, where first and intersection is formed identifying $t$ and $\bar t$, that point deleted and the boundry sutured by the dotted arcs.}
\end{figure}

\begin{figure}
\includegraphics[width=.7\textwidth]{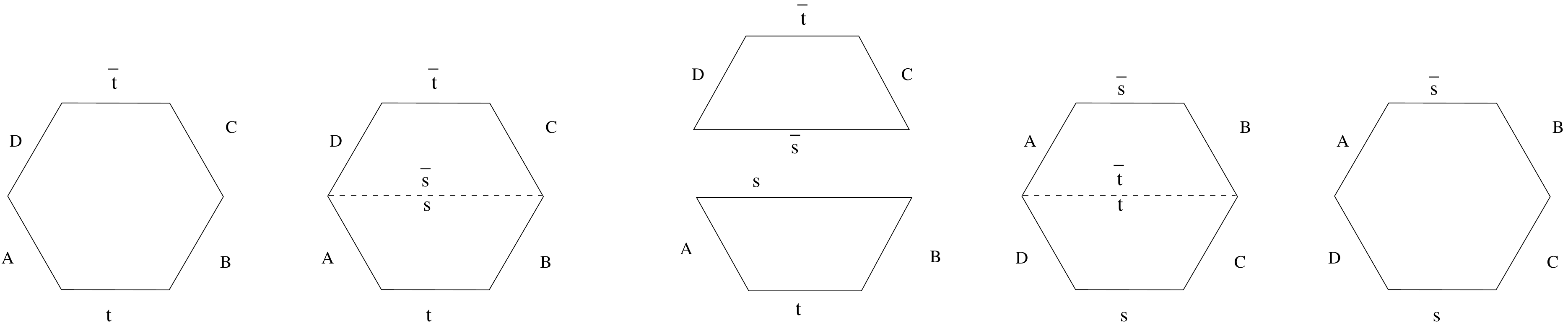}
\caption{\label{fig:mutapoly}An elementary mutation of labeled polygons}
\end{figure}

\begin{figure}
\includegraphics[scale=.32]{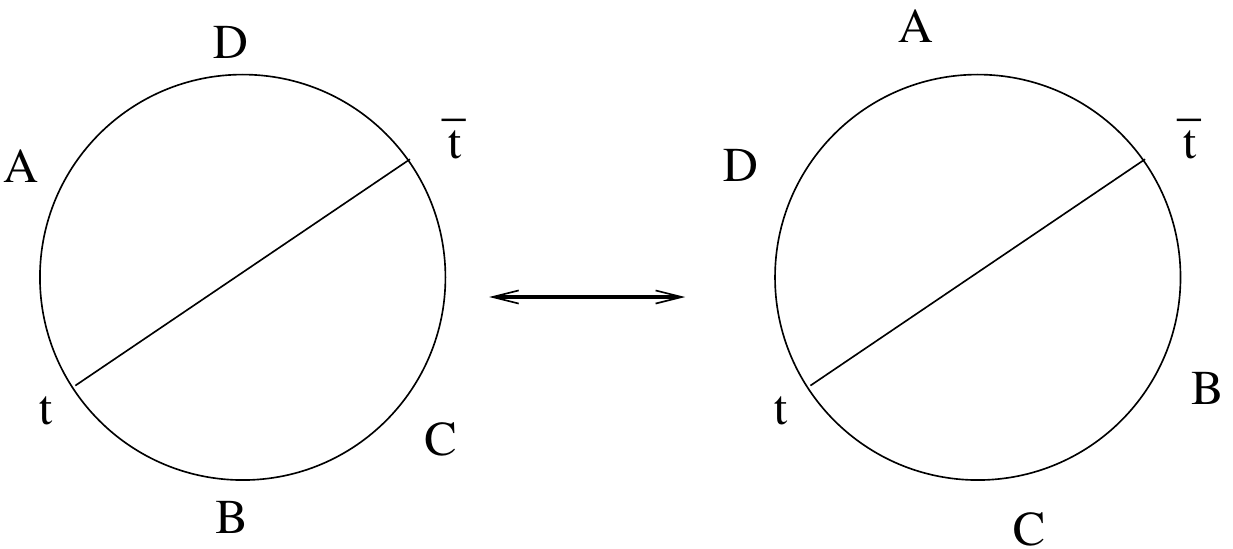} \quad
\includegraphics[scale=.32]{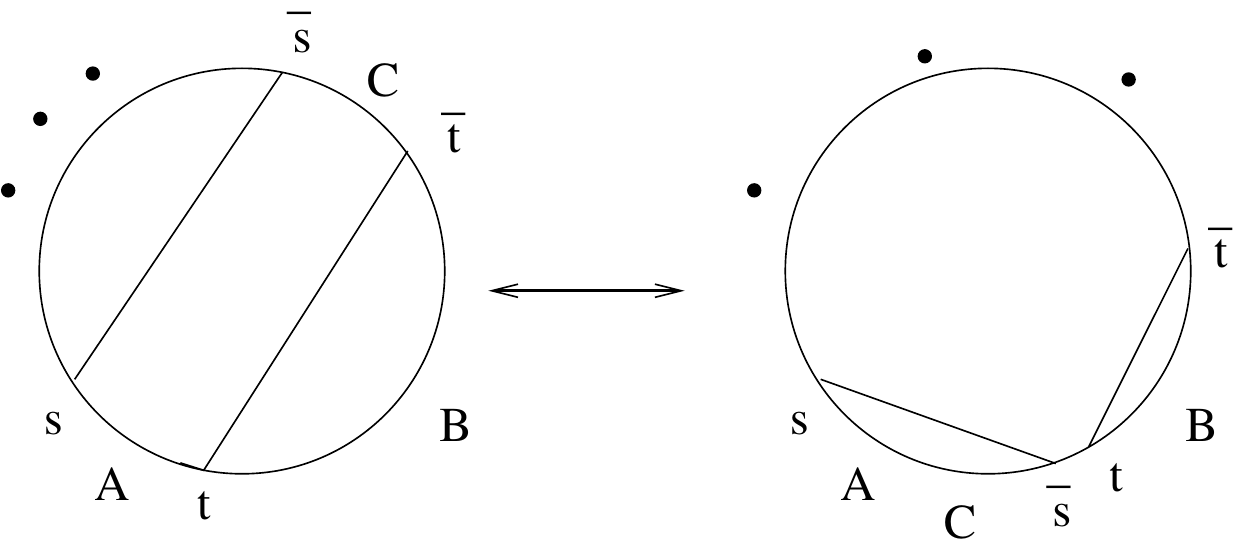} \quad
\includegraphics[scale=.32]{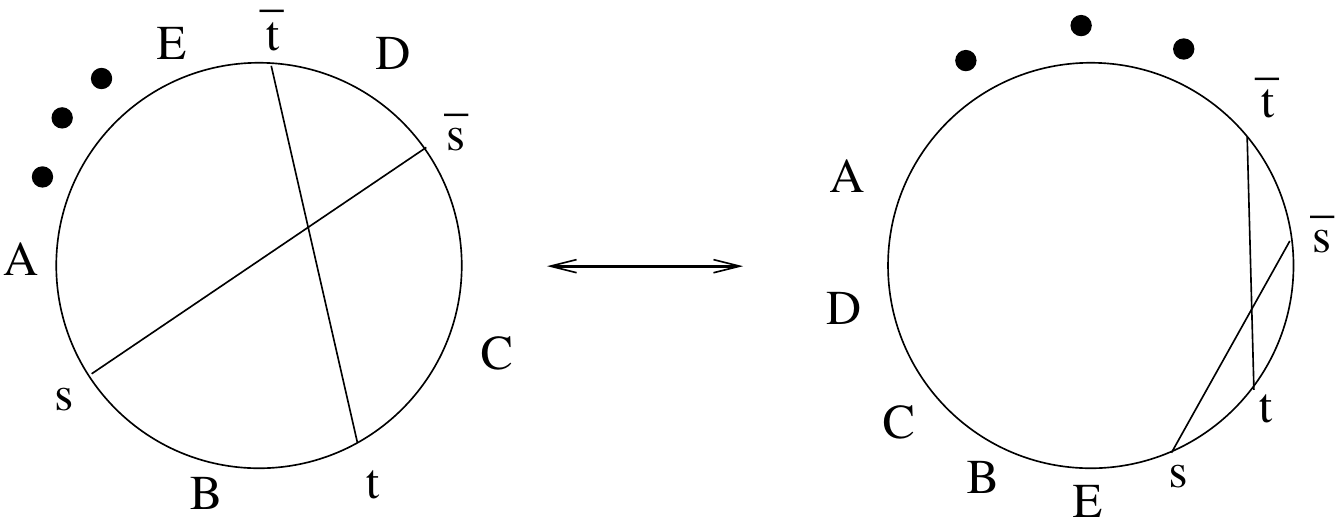}
\caption{\label{fig:mutachord} Lemma \ref{mutalem} in chord diagrams.}
\end{figure}

\begin{figure}
\includegraphics[scale=.38]{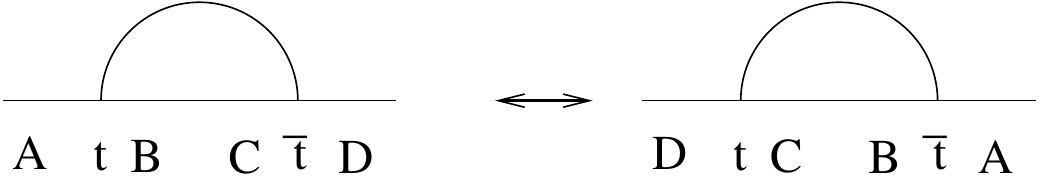} \quad
\includegraphics[scale=.38]{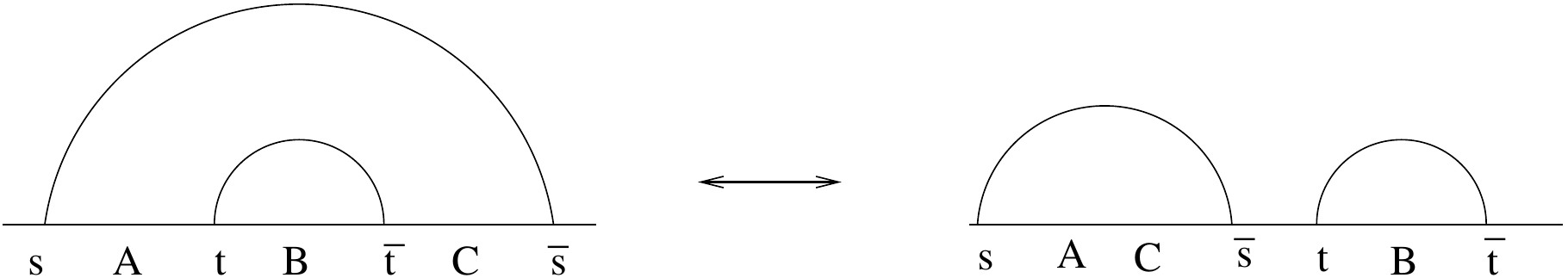} \quad
\vskip 5mm
\includegraphics[scale=.38]{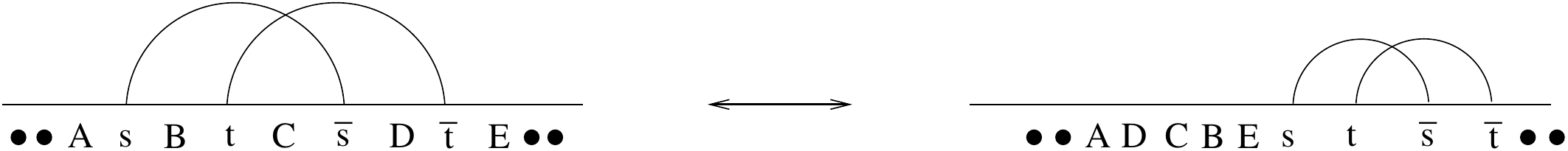}
\caption{\label{fig:mutarainbow}Lemma \ref{mutalem} in rainbow diagrams. This can be seen as sliding handles and flags}
\end{figure}

\begin{figure}
\includegraphics[width=.7\textwidth]{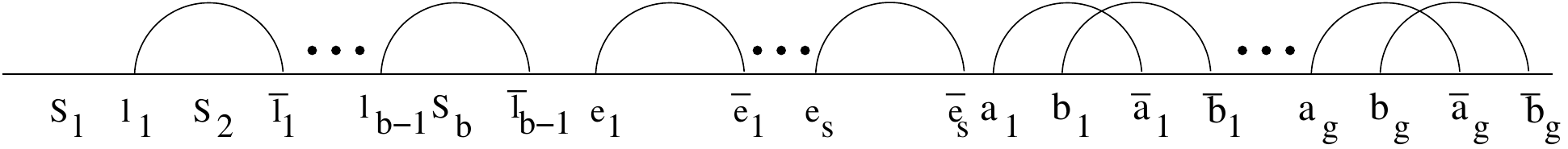}
\caption{\label{fig:rainbownormal}Normal form in rainbow diagrams}
\end{figure}

\subsection{Pushforwards to $\FFagg$}
One can furthermore study pushforward along the inclusion $l:\GGctd\to\FFagg$. These pushforwards carry more structure and allow us to keep track of several components at a time.

For a partition $P=S_1\sqcup\dots\sqcup S_k$ of $S$, we set $\Aut(S/P)=\Aut(S)/(\Aut(S_1)\times \dots \times \Aut(S_n)\wr \SS_n)$.
\begin{lem}
\label{lem:nccomp}
The connected components $\pi_0(l\downarrow *_S)$ are given by pairs consisting of a partition $S=S_1\sqcup \dots\sqcup S_n$ into possibly empty sets and an element $\sigma\in \Aut(S/P)$.\end{lem}

\begin{proof}
Using Theorem \ref{thm:aggstructure} we can factor any $\phi$ uniquely as $\phi=\sigma\phi_{m}\phi_{v\mdash con}$. Precomposing with an isomorphisms of the source, we stay in the same fiber, but can assume that   $\ghost(\phi)$ has $S$ as outer flag set. Precomposing with $\phi_{v\mdash con}$ this decomposition receives a map from $\sigma\phi_m$ where $\phi_m$ is a merger $t(\ghost(\phi)) \to \ast_S$ and $\sigma:\ast_S\to \ast_S$ is an isomorphism. Here the $S_{\bar v}$ are the outer flags of the component of $\ghost(\phi)$ indexed by $\bar v$. The image of $\Aut(t\ghost(\phi))$ in $\Aut(\ast_S)$
is precisely $\Aut(S_1)\times \dots \times \Aut(S_n)\wr \SS_n$ under the crossed structure of Corollary \ref{cor:aggcrossed}, see equation \eqref{eq:crossed}. This also shows that the partition together with an element completely classifies the fibre.
\end{proof}

\begin{cor}
$l_!(\genus)$ is given by $${\genusnc}(*_S)=
\{\textrm{genus labelled partitions of }S\}\times \Aut(S)/\Aut(P)$$\end{cor}
A typical element is an unordered tuple $[(g_1,S_1),\dots,(g_n,S_n))]$ where
the $S_i$ are a partition of $S$ by possibly empty subsets and $g_i\in \N_0$.
This is an unordered tuple, the order of the entries does not matter and we may have repetitions. We have the following behaviour under morphisms:
Isomorphisms act naturally on the partition and $Aut(P/S)$.

For the compositions, say $s\in S_i$, $t\in T_j$:
\begin{multline}
{\genusnc}(\scirct)([(g_1,S_1),\dots,(g_n,S_n)], [(g'_1,T_1),\dots,(g'_m,T_m)])=\\
[(g_1,S_1),\dots,\widehat{(g_i,S_i)}, \dots (g_n,S_n),(g'_1,T_1),\dots,\widehat{(g'_i,T_i)}, \dots(g'_m,T_m),
(g_i+g'_j,S_i\scirct T_j)]
\end{multline}
and the elements of the automorphisms groups are given by the restriction along $S\setminus \{s\}\sqcup T\setminus\{t\}\to S\sqcup T$.
If $s\in S_i$ and $s'\in S_j\neq S_i$ then
\begin{multline}
\genusnc(\circ_{ss'})([(g_1,S_1),\dots,(g_n,S_n)])=\\
[(g_1,S_1),\dots,\widehat{(g_i,S_i)}, \dots,\widehat{(g_j,S_j)},\dots (g_n,S_n), (g_i+g_j,S_i\ccirc{s}{s'}S_j)]
\end{multline}
and if  $s,s'\in S_i$.
\begin{multline}
\genusnc(\circ_{ss'})([(g_1,S_1),\dots,(g_n,S_n)])=
[(g_1,S_1),\dots,\widehat{(g_i,S_i)}, \dots, (g_n,S_n), (g_i+g_j,\circ_{ss'}S_i)]
\end{multline}
with the elements of the automorphisms groups again given by restriction.
Finally, mergers just merge lists.
\begin{multline}
{\genusnc}(\vmgew):([(g_1,S_1),\dots,(g_n,S_n)], [(g'_1,T_1),\dots,(g'_m,T_m)])\mapsto\\
[(g_1,S_1),\dots,(g_n,S_n),(g'_1,T_1),\dots,(g'_m,T_m)]]
\end{multline}
with the elements of the automorphisms given by inclusion $ \Aut(S/P)\times \Aut(T/P')\to \Aut((S\sqcup T)/(P\sqcup P'))$.

\begin{proof}
This is a straightforward computation following from Lemma \ref{connectedcomponents} and Lemma \ref{lem:nccomp}.
\end{proof}

There is a natural transformation $\genus^{nc}\to \genus$ given by
\begin{equation} ([(S_1,g_1),\dots,(S_n,g_n)],\sigma)\mapsto  1-\chi=1-n+\sum_i g_i
\end{equation}

\begin{rmk}
The surface interpretation is a disconnected surface. Note that the automorphisms groups cannot mix boundary components of the different components of the surface. To get the action on all of them, one has to induce up the automorphisms groups, which is what $\sigma$ keeps track of. The natural transformation is what is used in \cite{Zwiebach,Schwarz,HVZ,KWZ} to forget the internal disconnected structure. The upshot of including the nc case is a BV structure vs\ just differential, see \cite{KWZ}. The set $\Aut(S/P)$ also appears in the theory of PROPs when regrading the PROP generated by an operad or more generally by a properad.\end{rmk}

A polycyclic partition $(P,\sigma_P)$ of $S$ is a partition $S=S_1\sqcup\dots\sqcup S_n$ with individual polycyclic structures $\sigma_i$, i.e. $\sigma_P\in \prod_i \Aut(S_i)$.
We set $\Aut(S/(P,\sigma_P)):=\Aut(S)/Stab(\sigma_P)$. A general element is given by  $([(g_1,p_1,S_1,\sigma_1)\kdk (g_1,p_1,S_1,\sigma_1)],\sigma_P)$. We will write $S_i^\cord$ for $(S_i,\sigma_i)$.

\begin{cor}$l_!(\surf)$ is given by the set $\surf^{nc}(*_{S})$ consisting of polycyclic partitions $(P,\sigma_P)$ of $S$ together with two natural numbers for each element in the partition.\end{cor}
The action of isomorphisms is via pullback,
the composition  for mergers is joining of lists as above.
For the morphisms $\scirct$ the composition is that of $\surf$ on the two entries $$\surf(\scirct)((g_i,p_i,S_i^\cord,(g_j,p_j,T_j)^\cord),$$ while the others are unchanged. Similarly if $s\in S_i, s'\in S_j. i\neq j$ then the action on the only changed entries is $\surf(\ccirc{s}{s'}((g_i,p_i,S_i^\cord),(g_j,p_j,T_j)^\cord))$, while if
$s,s'\in S_i$ only one entry changes $\surf(\circ_{ss'})(g_i,p_i,S_i^\cord)$.
The poly--polycyclic structures $\sigma_P$ compose via inclusion as above.

\begin{proof}This again follows from Lemma \ref{connectedcomponents} and Lemma \ref{lem:nccomp}. With the addition that the automorphisms group in the decorated index category has to fix the polycyclic structure.
\end{proof}

There is a natural transformation $\O^{nc}_{surf}\to\surf$ given by

\begin{multline}
([(g_1,p_1,S_1^\cord)\kdk (g_n,p_n,S_n^\cord)],\sigma_P)
\mapsto(1-n+\sum_i g_i,\sum_i p_i,\{S^\cord_1,\dots S^\cord_n,T^\cord_1, \dots, T^\cord_m\})
\end{multline}

\subsection{Connected sum as a $B_+$ operator}\label{par:connectedsum}There is another operation which we can perform, and this is to take two tuples and simply merge them.
This is how the polycyclic structures arise in Kontsevich's description, see Propositon \ref{prop:KPK}.
\begin{equation}
B_+:(([(g_1,p_1,S_1^\cord)\kdk (g_n,p_n,S_n^\cord)],\sigma_P)
=(\sum_i g_i,\sum_i p_i,\{S^\cord_1,\dots S^\cord_n,T^\cord_1, \dots, T^\cord_m\})
\end{equation}

This is not a natural transformation of $\FFagg$ operations,  as  the equation \ref{eq:triangle} does not hold.
It does however define a new Feynman category. The relationship is as in \cite[\S3.2.1]{feynman}In terms of surfaces $F_1$ and $F_2$, this corresponds to the connected sum $F_1\#F_2$ and in terms of physics it is a $B_+$ operator in the sense of Connes and Kreimer \cite{CK}. This also plays a role in string topology, which will be explained in \cite{Ddec2}.

\begin{rmk}Geometrically the $B_+$ is the connected sum operation. This means that the boundary components of the different components are now boundary components of the same connected component.
\end{rmk}

\section{Actions}

The structure of the category of aggregates, in particular the adjunction between pushforward and pullback functors, has a direct application to 1+1 dimensional TFTs. Beyond this there is an interpretation for the correlators \cite{hoch2} giving rise to algebraic string topology operations as well as to operations on the Tate--Hochschild complex \cite{hochnote,KWang}.

\subsection{Algebras via reference functors}
For operads, the usual definition of an algebra in a closed symmetric monoidal
category $\mathcal{C}$ is an object $A$ of $\mathcal{C}$ together with a morphism of operads $\rho:\mathcal{O}\to \underline{End}_A$, where  $\underline{End}_A(n)=\inthom(A^{\ot n},A)$ denotes the \emph{endomorphism operad} of $A$ and $\inthom$ denotes the internal hom of $\Cc$. Likewise, for a PROP $\mathcal{P}$, an algebra is a pair consisting of an object $A$ and a morphism of PROPs $\rho:\mathcal{P}\to \underline{End}_A^A$ where now the endomorphism PROP of $A$ is  $\underline{End}_A^A(n,m)=\inthom(A^{\ot n},A^{\ot m})$.

In order to generalize these notions, we define a {\em reference functor} for $\FF$ to be a monoidal functor $\mathcal{E}\in[\mathcal{C}, [\mathcal{F},\mathcal{C}]_\otimes]_\otimes$.
\begin{df}
Given a reference functor $\mathcal{E}$ and a $\F$-operation $\mathcal{O}\in [\mathcal{F},\mathcal{C}]_\otimes$, an  algebra over $\mathcal{O}$ with values in $\mathcal{E}$ is a pair $(X,\rho)$ consisting of an object $X$ of $\Cc$ and a natural transformation $\mathcal{O}\to\mathcal{E}(X)$.
\end{df}
This is functorial in all variables when regarded as  elements of the functor $[\mathcal{F},\mathcal{C}]_\otimes\times \mathcal{C}\times [\mathcal{C},[\mathcal{F},\mathcal{C}]_\otimes]\to Set$, given by $(\mathcal{O},X,\mathcal{E})\to {\it Nat}(\mathcal{O},\mathcal{E}(X))$, i.e. evaluation and application the hom--functor in the functor category $[\mathcal{F},\mathcal{C}]_\otimes$.
Reference functors transfer between Feynman categories via pullback.

\subsection{Reference functors  for $\FFagg$ and correlation functions}

Consider any functor $\O:\FFagg\to \C$. First, $\O(\ast_\emptyset)$ forms a monoid in $\C$ under $\O(\mge{\ast_\emptyset}{\ast_\emptyset}$ (cf. \cite[\S 2.9.1]{feynman}) and by changing $\C$ if necessary to objects over $\O(\ast_\emptyset)$, we may assume that $\O(\ast_\emptyset)=\unit_\C$.  Second, there is an operation $\O(\ccirc{0}{0}):\O(\ast_{\{0\}})\ot\O(\ast_{\{0\}}) \to \O(\ast_{\emptyset})=\unit$.  Setting $W=\O(\ast_{\{0\}})$ makes  $P:=\O(\ccirc{0}{0})\in Hom(W^{\ot 2},\unit)$ into a pairing. The pairing is  symmetric, as there is only one morphism $\ast_{\{0\}}\sqcup \ast_{\{0\}}\to \ast_\emptyset$ whose automorphism group is given by interchanging the two factors.

The existence of a pair $(W,P)$ is thus common to all functors on $\FFagg$. This motivates the construction of a particular reference functor.
Let $\C_\P$ be the category of pairs $(W,P)$ with $W\in Obj(\C)$ and $P$ a symmetric pairing on $W$. Morphisms are the subsets $Hom_{\C_\P}((W,P),(W',Q))\subset Hom_\C(W,W')$ given by those morphisms $\phi:W\to W'$ which respect the pairings under pullback: $\phi^*(Q)=Q\circ(\phi\otimes\phi)=P$.

\begin{df}

Each pair $(W,P)$ in $\C_\P$ defines a $\FFagg$-operation $\Cor_{(W,P)}$ called the \emph{universal $W$--correlation functions with pairing $P$} defined as follows:
$\Cor_W(*_S)=W^{\otimes S}$. \\
For any  morphism $\phi:X\to Y$, the correlation functions $\Cor_{(W,P)}(\phi):W^{\otimes F(X)}\to W^{\otimes F(Y)}$ are given by contracting the along the ghost edges using $P$:
\begin{equation}
\begin{CD}
S^{\otimes F(X)} =W^{\ot (\phi^F)^{-1}(F(Y))} \otimes  (W\ot W)^{E(\gh(\phi))}=W^{\ot F(Y)} \otimes  (W\ot W)^{E(\gh(\phi))}
@>id^S\otimes P^{\otimes E(\gh(\phi))}>> W^{\ot F(Y)}
\end{CD}
\end{equation}
where two tensors factors of $W$ indexed $f$ and $f'$ for each ghost edge $e=\{f,f'=\imath (f)\}$ of $\gh(\phi)$ are contracted with $P$, which is well defined as $P$ is symmetric, and we used the bijection of $\phi^F$ onto its image.
The action by isomorphisms is by permutations and relabelling of factors.
The action of mergers is the multiplication in the tensor algebra.
\end{df}
 \begin{lem}
  $\Cor_{(V,P)}:\C_P\to \FFagg\text{-}\opcat_\C$ is functorial and provides a reference functor of $\FFagg\alg$.
  \end{lem}
\begin{proof}
Straightforward.
\end{proof}

Define $\vee:\C\to \C^{op}$ as usual by $V\to \check V=\underline{Hom}(V,\unit_\C)$.
$P$ is non--degenerate, if $\vee_P$ is an isomorphism.

\begin{ex}[Correlations functions from propagators.]
\label{ex:corpropagator}
Often, for instance in physical and geometric applications, $W=\check V$ and the pairing on $W$ is given by a propagator or Casimir element, that is
a symmetric element $C\in V\otimes V$ which yields a pairing $P\in  \underline{Hom}(\check V^{\otimes 2}, \unit)$ by evaluation. Physically, if $V$ is a space of fields,
then an element in $W^{\ot S}$ thought of as a morphisms $V^{\ot S}\to k$ is a correlation function, whence the name.
A geometric example is furnished by $V=H^*(M)$ for $M$ a compact manifold and $\bar P$ is the class of the diagonal in $H_*(M)\ot H_*(M)$, cf.\ \cite{hochnote}.
 Thus the present formalism is the most general.
 In the non--degenerate case, these formulations are  equivalent and are induced via the isomorphism $\vee_P$.
\end{ex}

\begin{rmk}
For the special case of $\C=k\mdash\Vect$ the notion of an algebra over a cyclic and modular operad was defined
 in \cite{GKcyclic,GKmodular} where it is assumed that $P$ is non--degenerate. The even/odd distinction was stressed in \cite{ConantVogtmann} and pairings of different degrees were treated in \cite{Barrannikov}, see also \cite{KWZ}. Without the assumption of non--degeneracy this treatment also yields the notion of {\em abstract correlation functions} of \cite{hoch2} where also the values were taken in twisted $Hom$ functors---a necessary step for Deligne's conjecture. The formalism of contracting tensors goes back to \cite{G} and is used in Gromov--Witten theory \cite{KoMa,ManinBook}.
\end{rmk}

 If $\FF$ has a functor to $\FFagg$, then let $\mathcal B$ be the underlying functor  $\F\to \Agg$. We define $\Cor^{\FF}_{(V,P)}=\B^*(\Cor_{(V,P)})=\Cor_{(V,P)}\circ \B$. If it is obvious from the context, we will omit the superscipt $\FF$.
 An algebra over an $\FF$--operation $\O$ in $\C$ is defined  to be an algebra over $\O$ with values in $\Cor^{\FF}$. These are given by an object $(V,P)\in \C_\P$ and a natural transformation $N$ from $\O$ to $\B^*(\Cor_{(V,P)})$. An algebra is hence a tuple $((V,P),\O,N)$.
In the non--degenerate case, we the usual notation  for $\Cor_{(V,P)}$ is $\End_{(V,P)}$.

\begin{ex}
 For $\FFcyc$, it is common to work with a skeleton of $\FinSet$, cf.\cite{GKcyclic}. This means that one uses a standard set of corollas,  $*_{[n]}$ with vertex $*$ and with flag sets $\{0, \dots, n\}$. For $n=-1$ the flag set is empty by convention.

In this setting, one also defines $\End_{(V,P)}(*_{[n]})=\inthom(V^{\otimes n},V)\simeq \check V\otimes V^{\otimes n}$ with the first factors of $V$ called inputs and the last factor of $V$ the output. This is the dualisation of $\Cor_{(V,P)}$ in the target variable using $\vee_P$. The compositions are given by contracting the ``out'' $V$ with an ``in'' $V$. The condition of non--degeneracy then implies that under $\vee_P$, this corresponds precisely to contracting with $P$. The equivariance is harder to formulate in this framework and is not as natural, see e.g.\ \cite{woods,feynman}.
\end{ex}

Table \ref{opertable} contains algebras over given operations. The first two are well known and establishing the remaining entries is the goal of this section.
\begin{rmk}

There is a directed version of  Feynman categories indexed over the directed version of aggregates of  \S \ref{par:dir}, cf.\ \cite[\S 2.2]{feynman}, which has a simpler reference functor $\mathcal{E}$ given by $\mathcal{E}_{W}(*_{S_{\rm in}\amalg T_{\rm out}})=\inthom(V^{\otimes S_{\rm in}},V^{\otimes T_{\rm out}})$ and the functor uses evaluation on each of the ghost edges, which have one $V$ (``out'') and one $\check V$ (``in'') associated to them. This explains why there is no need to choose a pairing or propagator for algebras over  operads or PROPs.
Additionally, there is a generalisation to the coloured context \cite[\S 2.5]{feynman}, where now there is a set of objects in $\C_p$, one for each color.
\end{rmk}

An algebra $(W,P),Y)$ over the trivial operation $\trivial$ yields elements in each $\Cor_{(W,P)}(X)$
via $Y_{X}:\trivial(X)=\unit\to \Cor_{(W,P)}(X)$. These are called {\em correlation functions}.
Since $Y$ is a natural transformation, these correlators are not independent, but have to satisfy compatibilities.
\begin{lem}
\label{lem:nat}
Given a set of elements $Y_S=Y_{*_S}$ the condition for being a correlation function  corresponding to the different generators of $\FFagg$ are:
\begin{enumerate}
\item For an isomorphism given by the bijection $\sigma^F:T \to S$, the compatibility is equivariance $Y_S=\sigma^F_*Y_T$.

\item The compatibility for $\scirct$ is
 $
 \imath_{\bisub{s}{P}{t}} Y_S \ot T_T=Y_{(S\setminus \{s\})\amalg T\setminus \{t\}}
$
 where $\imath_{\bisub{s}{P}{t}}$ contracts the tensors in positions $s$ and $t$, then $Y_S$ is a set of correlation functions for $\FFcyc$.
\item The compatibility with $\circ_{s,s'}$ is $\imath_{\bisub{s}{P}{s'}}Y_S=Y_{S\setminus \{s,s'\}}$.
\item Finally, the correlation functions are compatible with $\vmgew$ if
$Y_S\ot Y_{T} = Y_{S\amalg T} $.
\end{enumerate}
For being correlation functions on $\FFcyc$  (1) is necessary and sufficient, (1) and (2) are for $FFctd$, and all are for $\FFagg$.
\end{lem}

\begin{proof} This is an application of naturality. Since $N$ is a natural transformation, the diagram below commutes and gives the equality for (1).
\begin{equation}
\begin{CD}
\trivial^\FF(*_S\amalg *_T)=\unit\ot \unit @>\trivial(\scirct)=l_{\unit}>>\trivial^\FF(*_{S\setminus \{s\}\amalg T\setminus \{t\}})=\unit\\
 @VN_{*_S\amalg *_T}=Y_S\otimes Y_TVV @VVN_{*_{S\setminus \{s\}\amalg T\setminus \{t\}}}V\\
\Cor_{(V,P)}(*_S)\otimes \Cor_{(V,P)}(*_T) @>\Cor_{(V,P)}(\scirct)>>\Cor_{(V,P)}(*_{S\setminus \{s\}\amalg T\setminus \{t\}})
\end{CD}
\end{equation}
where $l_{\unit}$ is the unit constraint and the morphism $ \Cor_{(V,P)}(\scirct)$ is the contraction with $P$ in the positions $s$ and $t$.
The rest is analogous for isomorphisms $\scirct,\circ_{ss'}$ and $\vmgew$.
Since the 4 classes of morphisms generate, we get the necessary part. For the sufficient part, one has to check the relations, but this is straightforward, since the edges of the ghost graph are contracted with $P$ and it does not matter in which order this is done.
\end{proof}

Traditionally, many calculations are done in a skeletal version.
Here the standard notation for $Y_{*_{\{1,\dots n\}}}$ is $Y_n$. In the case of $\FFnscyc$ the corolla $*_{{\{1,\dots n\}}}$ is taken to  have the standard cyclic order on $\{1,\dots n\}$ and $Y_n$ denotes the respective correlation function.

\begin{df}
An $\FFcyc$ or $\FFnscyc$ algebra over $\trivial$  given by $((W,P),Y)$ is {\em unital}  if $Y_2$ and $P$ are inverse to each other, i.e.\
 the image of $Y_2\ot P\in  W^{\otimes 2}\ot \check W^{\ot 2}$ under the canonical pairing on the second and third factor is $id_W$.
\end{df}
This implies that $Y_2$ is non--degenerate and $W$ and $\check W$ are dual and $P$ is the Casimir element for the form pulled back to $V=\check W$. For elements of $V$, using Sweedler notation for $P$, this is equivalent to the familiar $\sum \la a, P^{(1)}\ra P^{(2)}=a$.

Note that the property of being unital is natural in $\C_\P$.

\begin{prop}
\label{frobstructureprop}
\label{reductioncor}  $\FFcyc$ resp.\ $\FFnscyc$ algebras $(W,P),Y)$ over $\trivial$
are classified up to isomorphism by pairs $(Y_1,Y_3)$ consisting of an element $Y_1\in W$ and symmetric resp.\ cyclicly  invariant tensor $Y_3\in W^{\ot 3}$,
which satisfy the three compatibility equations
 \begin{equation}
 \label{eq:compats}
 \imath_{\bisub{1}{P}{1}}(Y_1\ot \imath_{\bisub{1}{P}{1}}Y_1\ot Y_3 )=Y_1 \quad
 \imath_{\bisub{2}{P}{1}}(\imath_{\bisub{1}{P}{1}}Y_1\ot Y_3 \ot Y_3)=Y_3 \quad  \imath_{\bisub{3}{P}{1}}(Y_3\ot Y_3)
 \text{ is cyclically invariant}
 \end{equation}
\end{prop}

\begin{proof}
By Proposition \ref{lem:nat} part (1), picking skeletal  objects, we can reduce to the $Y_n$. From part (2) $Y_n$ , $Y_{0}$, and $Y_2$ satisfy:
\begin{equation}
Y_n=\imath_{\bisub {n-1}{P}{1}}Y_{n-1}\ot Y_3 \quad Y_2=\imath_{\bisub{1}{P}{1}}Y_1\ot Y_3 \quad Y_0= \imath_P Y_2\in \unit
\end{equation}
this allows to reduce to $Y_3$ and $Y_1$ and explains the necessity of the first two compatibility equation.
The third compatibility   concerns two different virtual edge contractions that both result in $*_{\{1,2,3,4\}}$, see Figure \ref{fig:Whitehead}. Notice that these are all cyclic relations and they thus lift to $\FFnscyc$.

The fact that these relations generate all relations, follows from standard arguments, see e.g.\ \cite{dijkgraafthesis,KP}.
Geometrically, this is the fact that Whitehead moves act transitively on pairs of  pants decompositions or diagonal compositions of polygons. Combinatorially this is the case, since the space of (planar)  trees with fixed tails is connected by  edge contractions and expansions, which amount to mutations.
A purely algebraic proof is in e.g.\ in \cite{hochnote}.

\end{proof}

\subsection{Commutative and symmetric (aka\ closed and open) Frobenius algebras}
There are several equivalent characterizations for symmetric Frobenius algebras, cf.\ e.g.\ \cite{ManinBook,hochnote,drin}.
We will discuss two convenient forms
using the  standard notation. This is $W=\check V$, $Y_n=\la\; ,\dots, \;\ra_n$.
The element $\la\;\,\;\ra_1$ is usually denoted by $\eps$ or $\int$ and the element $\la\;,\;\ra_2$ simply by $\form$.

\begin{df}
A symmetric (aka\ open) Frobenius algebra is a unital associative algebra with a symmetric non--degenerate bilinear form $\form$ which is invariant $\la a,bc\ra=\la ab,c\ra$.

A commutative (aka\ closed) Frobenius algebra is a symmetric Frobenius algebra which is also commutative.
\end{df}

\begin{prop}
\label{frobprop}
The following is an equivalent definition of symmetric, resp. commutative Frobenius algebras, namely a quadruple $(V,\eps,\form,\la\;,\;,\;\ra_3)$ where
\begin{enumerate}
\item $V$ is a vector space,
\item $\eps:V\to \unit$ a so--called counit
\item $\form$ is a symmetric  non--degenerate bilinear producton $V$
\item  $\la\;,\;,\;\ra:V^{\otimes 3}\to \unit$, is a  3--tensor which is cyclically invariant in the symmetric case and $\SS_3$ invariant in the commutative case.
\end{enumerate}
which satisfies the  compatibility equations \eqref{eq:compats}. Where in these equations $P\in V\otimes V$ is dual to the metric $\form\in \check V\otimes \check V$.

\end{prop}

\begin{proof}
A Frobenius algebra furnishes the data satisfying the axioms:
 Set $\eps(a)=\la 1,a\ra$ and $\la a,b,c\ra_3=\la ab , c\ra$. The cyclicity of $\la \;\,\;\ra_3$ then follows from the symmetry and invariance of the metric:  $\la b,c,a \ra=\la bc,a\ra=\la a,bc\ra=\la ab,c\ra=\la a,b,c\ra$.
For the  first compatibility equations  one calculates:
\begin{equation}\label{calc1eq}
\sum \eps(P^{(1)})\la P^{(2)},a\ra=\sum  \la 1, P^{(1)}\ra\la P^{(2)},a\ra=
\la 1,a\ra=\eps(a)
\end{equation}
As $\form$ is non--degenerate,
this also shows that $\sum \eps(P^{(1)})P^{(2)}=1$.
Using this the second equation follows immediately:
\begin{equation}
\label{calc2eq}
\sum  \eps(P_1^{(1)})\la P_1^{(2)}, a, P_2^{(1)}\ra \la P_2^{(2)} ,b, c\ra=\la a,bc \ra=\la 1a,b\ra=\la a,b,c\ra
\end{equation}
Finally, for the third condition:
\begin{equation}
\label{calc3eq}
\begin{aligned}
\sum \la a,b, P^{(1)}\ra\la P^{(2)},c,d\ra&=\sum \la ab,P^{(1)}\ra\la P^{(2)}c,d\ra\sum \la ab,P^{(1)}\ra\la P^{(2)},cd\ra\\
=\la ab,cd\ra=\la a,b(cd)\ra=&
\la b(cd),a\ra=\la (bc)d,a\ra
=\la bc,da\ra\\
=\sum \la bc, P^{(1)}\ra\la P^{(2)},da\ra&=\la bc, P^{(1)}\ra\la P^{(2)}d,a\ra\\&=
\la b,c, P^{(1)}\ra\la P^{(2)},d,a\ra
\end{aligned}
\end{equation}
In the commutative case, $ab=ba$, thus $\la a,b,c\ra=\la ab,c\ra=\la ba,c\ra =\la b,c,a\ra$ which together with the cyclic symmetry implies the full $\SS_3$ symmetry.

The data and axioms define a Frobenius algebra:
Set $1=\sum \eps(P^{(1)})P^{(2)}$, and define the multiplication via $\la ab, c\ra =\la a,b,c\ra$.
The invariance of the metric follows from the cyclicity of $\corra_3$ and symmetry of $\form$: $\la ab,c\ra=\la a,b,c\ra=\la b,c,a\ra=\la bc,a\ra=\la a,bc\ra$.
The first and second equations of \eqref{eq:compats} guarantees that $1$ is indeed a unit, see \eqref{calc2eq}.
The associativity follows from the third equation.
\begin{equation}
\begin{aligned}
&\la (ab)c,d\ra=\la ab,cd\ra=\sum \la ab,P^{(1)}\ra\la P^{(2)},cd\ra=\sum\la a,b, P^{(1)}\ra\la P^{(2)},c,d\ra\\
&=\sum\la b,c, P^{(1)}\ra\la P^{(2)},d,a\ra
=\sum\la bc, P^{(1)}\ra\la P^{(2)},d,a\ra=\la bc,d,a\ra=\la a,bc,d\ra\\
&=\la a(bc),d\ra
\end{aligned}
\end{equation}
Furthermore, a full $\SS_3$ symmetry of $\corra_3$ implies that the multiplication is commutative: $\la ab,c\ra=\la a,b,c\ra=\la b,a,c\ra=\la ba,c\ra$.
\end{proof}
\begin{rmk}
A Frobenius algebra also gives rise to a comultiplication.  Using the non--degenerate form $\la \;,\;\ra_{\otimes}=\form \otimes\form\circ (23)$ on $V\otimes V$, one defines $\Delta:=\mu^\dagger$ that is
$\la \Delta(a),b\otimes c\ra_{\otimes}=\la a, bc\ra$. The dual of the unit $\nu:\unit \to V$ is a counit $\eps:V\to \unit$ and the
algebra and coalgebra structure satisfy the compatibility
\begin{equation}
\label{frobeq}
(\mu\otimes id)\circ (id\otimes \Delta)=\Delta\circ \mu=(id\otimes \mu)\circ (\Delta \otimes id)
\end{equation}
as maps $V^{\otimes 2}\to V^{\otimes 2}$.
The counit is again given by $\eps(a)=\la a, 1 \ra$ and is indeed a counit for $\Delta$:
\begin{equation}
\label{couniteq}
\la (id \otimes \eps)\Delta(a),b\ra=\sum \la a^{(1)} \eps(a^{(2)}),b\ra=\sum \la a^{(1)},b\ra\la a^{(2)},1\ra= \la \Delta(a),b\otimes 1\ra_\otimes= \la a,b\ra
\end{equation}
the equation for $\eps\otimes id$ is analogous.
\end{rmk}

This allows one to define weaker structures which naturally occur for instance in the setting of $K$--theory, cf.\ e.g.\ \cite[\S 3.1--3.3]{drin} and string topology \cite{CohenGodin,Ssigma,hoch2,hochnote}.
\begin{df}
A Frobenius object in a symmetric monoidal category $\C$ is an object $V$, together with an associative multiplication $\mu:V^{\otimes 2}\to \unit$ and  a coassociative comultiplication $\Delta:A\to V^{\otimes 2}$ which satisfy the compatibility equation \eqref{frobeq}.
A Frobenius algebra object in a symmetric monoidal category $\C$ is a Frobenius object together with a unit for the multiplication and a counit for the comultiplication.
\end{df}

\begin{rmk}
Having a multiplication and a morphism $\eps: V\to \unit$ produces a form $\la a, b\ra=\eps(ab)$. An element $u$ and a co--multiplication gives a propagator $P=\Delta(u)$. Requiring both $\eps$ to be a co--unit and $u$ to be a unit, makes the bi--linear form non--degenerate as the contraction of $\form$ with $P$ in one variable yields the map $a\mapsto (\eps\otimes id)(\mu\otimes id)(\Delta\otimes id)(1\otimes a)=(\eps\otimes id)\Delta\mu(1\otimes a)=a$ which is the identity map. Note, by \eqref{couniteq}, if $u$ is indeed a unit, then $\eps$ is automatically  a co--unit.
\end{rmk}

By  Theorem \ref{thm:decorated} algebras over $\CycAss$ are in one-to-one correspondence with algebras over the trivial operation of $\FFnscyc$,
and an algebra over $\CycAss$ is unital, if its corresponding $\FFnscyc$ algebra is unital.
The following in different guises is part of folklore, for detailed examples on the needed algebraic manipulations, see e.g.\ \cite{hochnote}, but the presentation in this framework is new as well as the treatment of the non--unital  case.
\begin{thm} \mbox{}
\begin{enumerate}
\item Unital algebras over $\trivial^{\FFcyc}$ are commutative Frobenius algebras;
\item Unital algebras over $\trivial^{\FFnscyc}$ (resp. unital algebras over $\CycAss$) are symmetric Frobenius algebras;
\item Algebras over $\trivial^{\FFcyc}$ are commutative Frobenius objects, with a trace $\eps$ and a propagator;
\item Algebras over $\trivial^{\FFnscyc}$ are symmetric Frobenius objects, with a trace $\eps$ and a propagator.
\end{enumerate}
\end{thm}

\begin{proof}
(1) and (2) follow immediately from Propositions \ref{frobprop} and  \ref{frobstructureprop}.  Without the non--degeneracy assumption, we can define a multiplication by dualising $\corra_3$ in the last variable using $\vee_P$ and a comultiplication by dualising in the last two variables.
The Frobenius equation is then a straightforward check using \eqref{eq:compats}. The trace $\eps=Y_1$ and $P$ give the extra structures. Conversely, these dually allow to recover the $Y_n$ from the multiplication and comultiplication.
\end{proof}

\begin{rmk}\mbox{}
\begin{enumerate}
\item If one sets $u=(id\otimes \eps) (P)$, then one obtains a second propagator $Q=\Delta(u)$. These two propagators agree
if the form is non--degenerate.
\item $u$ plays the role of a unit in the sense that $\la a_1,\dots, u,\dots, a_n\ra_{n+1}=\la a_1,\dots, a_n\ra_n$.
\item Dualizing $\form$ defined via $P$ in one variable gives a morphism $p:V\to V$. It is easy to check that this is a projection $p^2=p$.
\item The quantity $\mu\Delta(1)=e$, (here $e$ stands for the Euler element, cf. \cite{hochnote}), is important, see also Remark \ref{qdimrmk} below.
 For instance if $A=H^*(M)$ for a compact oriented manifold $M$, with cup product and evaluation at the fundamental class, then $e$ is the Euler--class in top degree and $\eps(e)=\chi(M)$. It is the obstruction for the lift to a $\trivial^{\FFmod}$ algebra, viz.\ by Lemma \ref{lem:nat} the lift is possible if and only if $e=1$.
 This corresponds to the possibility to pass to a stabilzation cf.\ \cite{postnikov,MM2}, which morever appears in the theory of Steenrod operations \cite{KMM}.
\item We see that $\corra_0=\eps(u)$ and in the non--degenerate case this is $\eps(1)$. This quantity is sensitive to nilpotent vs.\ semisimple Frobenius algebras, cf.\ \cite{ManinBook,ROMP,hochnote}. In the geometric  case above $A=H^*(M)$, one sees that unless
$\dim(M)=0$, $\eps(1)=0$.
\end{enumerate}
\end{rmk}

\subsection{Adjunction and 1+1 d QTFTs}
The functor $j:\FFcyc\to \FFmod$ provides interesting adjunctions.
Unital algebras over $j_!(\trivial^{\FFcyc})=\trivial^{\FFmod}$ are known as 1+1 d TQFTs since they associate a correlation function to each $*_{S,g}$, which can be thought of as an oriented surface of genus $g$ with $S$ boundaries.
Similarly, unital algebras over $j_!(\CycAss)=\trivial^{\FFnsmod}$ are 1+1 d open TQFTs since they associate a correlation function to each
$*_{g,p,S_1,\dots, S_b}$, which can be viewed as an oriented surface of genus $g$ with $p$ marked points in the interior, $b$ boundary components, or equivalently unmarked boundaries, and $S_i$ marked points on boundary $i$. In both cases, the composition along a graph corresponds to sewing together the surfaces along the respective boundaries, thus realizing a version of a cobordism category.

Part of the following is folklore and has been proven several times in the literature \cite{dijkgraafthesis,ManinBook,Abrams} in different settings. We add the   novel feature is that everything follows from adjunctions. Our presentation also makes the constructions of \cite{Costellounpublished} clear.

\begin{thm} \mbox{}
\begin{enumerate}
\item Algebras over $\trivial^{\FFmod}$ are equivalent to algebras over $\trivial^{\FFcyc}$.
\item Unital algebras over $\trivial^{\FFmod}$, i.e.\ 1+1 d closed TQFTs, are equivalent to commutative Frobenius algebras.
\item  The following are equivalent:
\begin{enumerate}
\item Algebras over $\trivial^{\FFnsmod}$;
\item Algebras over  $\surf=j_!(\CycAss)$, i.e. the modular envelope of $\CycAss$.
\item Algebras over $\trivial^{\FFnscyc}$;
\item Algebras over $\CycAss$.
\end{enumerate}
\item Unital algebras for any of the 4 equivalent cases (a)-(d), i.e. 1+1 d open TQFTs, are equivalent to symmetric Frobenius algebras.
\item
Without the assumption of being unital, the algebras are  commutative, resp. symmetric, Frobenius objects with trace and propagator.
\end{enumerate}

\end{thm}
\begin{proof}Using the main diagram \eqref{eq:decodiagintro}, Propositions \ref{prop:ctdenvelope} and \ref{prop:modularenvelope},
 Theorem \ref{thm:decorated} and Theorem \ref{thm:adjunction}, we obtain adjunctions from which the first two statements follow:
\begin{equation}
 Nat(\trivial^{\FFcyc}, \Cor^{\FFcyc}_{V,P})=Nat(\trivial^{\FFcyc}, j^*\Cor^{\FFmod}_{V,P})
\leftrightarrow Nat(j_!(\trivial^{\FFcyc}), \Cor^{\FFmod}_{V,P})=Nat(\trivial^{\FFmod}, \Cor^{\FFmod}_{V,P})
 \end{equation}
 The third and fourth statement follow from the adjunctions:
\begin{equation}
\begin{aligned}
&Nat(\CycAss, \Cor^{\FFcyc}_{V,P})=
Nat(\pi_{2!}\trivial^{\FFnscyc}, \Cor^{\FFcyc}_{V,P})\leftrightarrow
Nat(\trivial^{\FFnscyc}, \Cor^{\FFnscyc}_{V,P})\\
&=
Nat(\trivial^{\FFnscyc}, j^{\prime *}\Cor^{\FFnsmod}_{V,P})\leftrightarrow
Nat(j'_!(\trivial^{\FFnscyc}), \Cor^{\FFmod}_{V,P})=Nat(\trivial^{\FFnsmod}, \Cor^{\FFnsmod}_{V,P})\\
&=
Nat( \trivial^{\FFnsmod} , \pi_3^*\Cor^{\FFmod}_{V,P})\leftrightarrow Nat(\pi_{3!}\trivial^{\FFnsmod} , \Cor^{\FFmod}_{V,P})=
Nat(\surf, \Cor^{\FFmod}_{V,P})
\end{aligned}
\end{equation}
\end{proof}

\subsection{Algebraic string topology operations}
The framework also naturally yields the correlation functions of \cite{hoch2,hochnote} which underly the algebraic string topology operations. For this we have to pull back the correlation functions graphs using the source functor $s:\Graph\to \Agg$ promoted to a Feynman functor $\FF^{\Graph}\to \FFagg$.

\begin{thm}The correlation functions of \cite{hoch2} in the general setting for symmetric Frobenius algebra $A$ \cite{hochnote} are given by
the natural transformation $s^*(Y)\in Nat[s^*\surf,\Cor_{A,P}]$.
\end{thm}
\begin{proof}

Pulling back along $s$ using Theorem \ref{thm:decorated} one has $s^*(Y)\in Nat[s^*\surf,s^*\Cor_{A,P}]$.
For a given surface decorated graph $\G$ we have that $s^*(\G)$ is the underlying corolla set. As we are dealing with monoidal functors,
we obtain $s^*(Y)(\G,a_\G\in s*\surf(X))=\bigotimes_{v\in V_\G}Y(v_{S,g,\sigma_{\F_v})}$ which is the formula (3.1) of \cite{hoch2} generalized to surface marked graphs as detailed in Corollary 5.2 of \cite{hochnote}, see equation (5.10), where $\G$ is dual to the surface with arcs as explained in \S\ref{par:surfinterpret}.
\end{proof}

\begin{rmk}
\label{qdimrmk}.
\begin{enumerate}
\item The value $Y_{*_{1,\empty}}\in \unit$ is $Tr(P)=\form\circ P=:e$, that is the quantum dimension.
In the unital case this is $\mu\Delta(1)$. (This follows from the morphism $\circ_{0,1}:*_{0,[1]}\to *_{1,\empty}$.)
\item If $V$ is commutative then $Y(*_{g,p,S_1,\dots,S_b})(\bigotimes_{s\in S}(a_s)$  is $\int \prod_{s\in S} a_s e^{-\chi(\Sigma)+1}$ where $\Sigma$ is the corresponding surface, cf.\ e.g.\ \cite{hoch2,hochnote}.
For the general formula in the non--commutative case, which is an algebraic analog of the chord diagrams used in the computations, see  \cite{hochnote}. It is essentially given by the normal form \eqref{eq:normalform}.
\item For string topology $A=H^*(M)$ or using a propagator given by the diagonal the correlation functions lift to $C^*(M)$, cf.\ \cite{hoch2,hochnote}.
\end{enumerate}

\end{rmk}

\begin{rmk}
Note that gluing on outer flags {\em is not} the PROP structure for string topology neither closed nor open, cf. \cite{ochoch},
which involves {\em gluing} on the boundary components of the polycyclic graph $\G$ as in \cite{hoch1,hoch2,hochnote}.
 This will be treated in \cite{Ddec2}.
\end{rmk}

\bibliographystyle{halpha}
\bibliography{Dennisbib}

\end{document}